\pdfoutput=1
\RequirePackage{ifpdf}
\ifpdf 
\documentclass[pdftex]{sigma}
\else
\documentclass{sigma}
\fi

\usepackage{euscript,tikz}

\renewcommand{\phi}{\varphi}
\renewcommand{\kappa}{\varkappa}

\renewcommand{\epsilon}{\varepsilon}
\newcommand{\C}{\mathbb C}
\renewcommand{\P}{\mathbb P}
\newcommand{\A}{\mathbb A}
\newcommand{\Z}{\mathbb{Z}}
\newcommand{\R}{\mathbb{R}}
\newcommand{\Q}{\mathbb{Q}}
\newcommand{\bL}{\mathbb{L}}
\newcommand{\kk}{\mathbf{k}}
\newcommand{\KK}{\mathbb K}
\newcommand{\bE}{\mathbf E}
\newcommand{\bF}{\mathbf F}
\newcommand{\OO}{\mathbb O}

\DeclareMathOperator{\val}{val}
\DeclareMathOperator{\cMot}{\overline{Mot}}
\DeclareMathOperator{\Mot}{Mot}
\DeclareMathOperator{\Spec}{Spec}
\DeclareMathOperator{\loc}{loc}

\DeclareMathOperator{\length}{length}
\DeclareMathOperator{\cl}{cl}

\DeclareMathOperator{\Id}{Id}
\DeclareMathOperator{\Fl}{Fl}
\DeclareMathOperator{\Hook}{Hook}
\DeclareMathOperator{\res}{res}
\DeclareMathOperator{\Res}{Res}
\DeclareMathOperator{\rk}{rk}
\DeclareMathOperator{\Hom}{Hom}
\DeclareMathOperator{\End}{End}

\DeclareMathOperator{\Pow}{\mathrm{Pow}}
\DeclareMathOperator{\Exp}{Exp}
\DeclareMathOperator{\Log}{Log}
\DeclareMathOperator{\Stab}{Stab}
\DeclareMathOperator{\tr}{tr}
\DeclareMathOperator{\Sym}{Sym}
\DeclareMathOperator{\Ker}{Ker}
\DeclareMathOperator{\Ext}{Ext}

\newcommand{\gl}{\mathfrak{gl}}

\newcommand{\cB}{\mathcal B}
\newcommand{\cC}{\mathcal C}
\newcommand{\cD}{\mathcal D}

\newcommand{\cH}{\mathcal H}

\newcommand{\cK}{\mathcal K}

\newcommand{\cO}{\mathcal O}

\newcommand{\cS}{\mathcal S}

\newcommand{\cY}{\mathcal Y}
\newcommand{\cX}{\mathcal X}
\newcommand{\Bun}{\mathcal{B}{un}}

\newcommand{\Pair}{\mathcal{P}{air}}
\newcommand{\Fib}{{\mathcal F}{ib}}
\newcommand{\wFib}{\overline{{\mathcal F}{ib}}}
\newcommand{\Conn}{\mathcal{C}{onn}}
\newcommand{\Higgs}{\mathcal{H}{iggs}}
\newcommand{\END}{\mathcal{E}{nd}}
\newcommand{\HIGGS}{\mathcal{H}{iggs}}
\DeclareMathOperator{\higgs}{Higgs}

\newcommand{\fg}{\mathfrak g}

\renewcommand{\Re}{\mathop{\mathrm{Re}}}
\renewcommand{\Im}{\mathop{\mathrm{Im}}}

\numberwithin{equation}{section}

\newtheorem{Theorem}{Theorem}[section]
\newtheorem{Corollary}[Theorem]{Corollary}
\newtheorem{Lemma}[Theorem]{Lemma}
\newtheorem{Proposition}[Theorem]{Proposition}
 { \theoremstyle{definition}
\newtheorem{Definition}[Theorem]{Definition}
\newtheorem{Remark}[Theorem]{Remark} }

\begin{document}
\allowdisplaybreaks

\newcommand{\arXivNumber}{1910.12348}

\renewcommand{\thefootnote}{}

\renewcommand{\PaperNumber}{070}

\FirstPageHeading

\ShortArticleName{Motivic Donaldson--Thomas Invariants of Parabolic Higgs Bundles}

\ArticleName{Motivic Donaldson--Thomas Invariants of Parabolic\\ Higgs Bundles and Parabolic Connections on a Curve\footnote{This paper is a~contribution to the Special Issue on Integrability, Geometry, Moduli in honor of Motohico Mulase for his 65th birthday. The full collection is available at \href{https://www.emis.de/journals/SIGMA/Mulase.html}{https://www.emis.de/journals/SIGMA/Mulase.html}}}

\Author{Roman FEDOROV~$^\dag$, Alexander SOIBELMAN~$^\ddag$ and Yan SOIBELMAN~$^\S$}

\AuthorNameForHeading{R.~Fedorov, A.~Soibelman and Y.~Soibelman}

\Address{$^\dag$~University of Pittsburgh, Pittsburgh, PA, USA}
\EmailD{\href{mailto:fedorov@pitt.edu}{fedorov@pitt.edu}}

\Address{$^\ddag$~Aarhus University, Aarhus, Denmark}
\EmailD{\href{mailto:asoibel@qgm.au.dk}{asoibel@qgm.au.dk}}

\Address{$^\S$~Kansas State University, Manhattan, KS, USA}
\EmailD{\href{mailto:soibel@math.ksu.edu}{soibel@math.ksu.edu}}

\ArticleDates{Received November 19, 2019, in final form July 10, 2020; Published online July 27, 2020}

\Abstract{Let $X$ be a smooth projective curve over a field of characteristic zero and let~$D$ be a non-empty set of rational points of $X$. We calculate the motivic classes of moduli stacks of semistable parabolic bundles with connections on $(X,D)$ and motivic classes of moduli stacks of semistable parabolic Higgs bundles on $(X,D)$. As a by-product we give a~criteria for non-emptiness of these moduli stacks, which can be viewed as a~version of the Deligne--Simpson problem.}

\Keywords{parabolic Higgs bundles; parabolic bundles with connections; motivic classes; Donaldson--Thomas invariants; Macdonald polynomials}

\Classification{14D23; 14N35; 14D20}

\renewcommand{\thefootnote}{\arabic{footnote}}
\setcounter{footnote}{0}

\section{Introduction and main results}
\subsection{Overview} Let $\kk$ be a field of characteristic zero and $X$ be a smooth geometrically connected projective curve over~$\kk$ (geometric connectedness means that $X$ remains connected after the base change to an algebraic closure of $\kk$). In~\cite{FedorovSoibelmans} we calculated the motivic classes of moduli stacks of semistable Higgs bundles on $X$. These motivic classes are closely related to Donaldson--Thomas invariants, see~\cite{KontsevichSoibelman08,KontsevichSoibelman10}. In~\cite{FedorovSoibelmans} we also calculated the motivic classes of moduli stacks of vector bundles with connections on $X$ by relating them to the motivic classes of stacks of semistable Higgs bundles.

In this paper, we extend these results to the parabolic case. Some of our results are parallel to the results of A.~Mellit in the case of finite fields (see~\cite{MellitPunctures}). One difference is that we fix the eigenvalues of the residues. Another difference is that by working over a field of characteristic zero, we are also able to treat bundles with connections. We also note that the calculation of the motivic class requires subtler techniques, than the calculation of the volume of the corresponding stack over a finite field.

\subsection{Moduli stacks} Let us briefly describe the moduli stacks whose motivic classes we will be interested in. There will be three classes of stacks.

\subsubsection{Parabolic bundles with connections}\label{sect:IntroConn}
Let $D\subset X(\kk)$ be a non-empty set of rational points of $X$. A \emph{parabolic bundle} of type $(X,D)$ is a collection $\bE=(E,E_{\bullet,\bullet})$, where $E$ is a vector bundle over $X$ and $E_{x,\bullet}$ is a flag in its fiber $E_x$ for $x\in D$:
 \[
 E_x=E_{x,0}\supseteq E_{x,1}\supseteq\dots\supseteq E_{x,l}\supseteq\cdots,\qquad E_{x,l}=0\quad \text{for} \ l\gg0.
 \]
A \emph{connection} on $E$ with \emph{poles bounded by $D$} is a morphism of sheaves of abelian groups $\nabla\colon E\to E\otimes\Omega_X(D)$ satisfying Leibniz rule. (Here $\Omega_X$ is the canonical line bundle on $X$.) In this case for $x\in D$ one defines the residue of the connection $\res_x\nabla\in\End(E_x)$. Let $\zeta=\zeta_{\bullet,\bullet}=(\zeta_{x,j})$ be a sequence of elements of $\kk$ indexed by $D\times\Z_{>0}$ such that $\zeta_{x,j}=0$ for $j\gg0$. Let $\Conn(X,D,\zeta)$ denote the moduli stack parameterizing collections $(E,E_{\bullet,\bullet},\nabla)$, where $(E,E_{\bullet,\bullet})$ is a parabolic bundle of type $(X,D)$, $\nabla$~is a connection on $E$ with poles bounded by $D$ such that $(\res_x\nabla-\zeta_{x,j}1)(E_{x,{j-1}})\subset E_{x,j}$ for all $x\in D$ and $j>0$. We usually skip $X$ and $D$ from the notation as they are fixed, denoting $\Conn(X,D,\zeta)$ simply by $\Conn(\zeta)$. We call the points of $\Conn(\zeta)$ \emph{parabolic bundles with connections of type $(X,D)$ with eigenvalues $\zeta$.}

For a parabolic bundle $\bE=(E,E_{\bullet,\bullet})$ define the \emph{class} of $\bE$ as the following collection of integers:
\begin{equation}\label{eq:class}
 (\rk E,\dim{E_{x,j-1}}-\dim{E_{x,j}},\deg E)\in\Z_{\ge0}\times\Z_{\ge0}[D\times\Z_{>0}]\times\Z.
\end{equation}
We also set $\rk\bE:=\rk E$.

The stack $\Conn(\zeta)$ decomposes according to the classes of parabolic bundles; denote the component corresponding to parabolic bundles of class $\gamma$ by $\Conn_\gamma(\zeta)$. We will see that this stack is an Artin stack of finite type over $\kk$. One of our main results (see Section~\ref{sect:ExplAnswers} and Theorem~\ref{th:ExplAnsw2}) is the calculation of the motivic class of this stack.

\subsubsection{Parabolic Higgs bundles}\label{sect:IntroHiggs} Let $\zeta$ be as above. A \emph{parabolic Higgs bundle with eigenvalues $\zeta$} is a triple $(E,E_{\bullet,\bullet},\Phi)$, where $(E,E_{\bullet,\bullet})$ is a parabolic bundle of type $(X,D)$, $\Phi\colon E\to E\otimes\Omega_X(D)$ is a morphism of $\cO_X$-modules (called a \emph{Higgs field on $(E,E_{\bullet,\bullet})$}) such that for all $x\in D$ and $j>0$ we have
\[ (\Phi-\zeta_{x,j}1)(E_{x,{j-1}})\subset E_{x,j}\otimes\Omega_X(D)_x.
\] Denote the category and the stack of such Higgs bundles by $\Higgs(\zeta)$. Unfortunately, this stack is not of finite type over $\kk$, and in fact, has an infinite motivic volume. To resolve the problem we endow the category with a stability structure. Let $\sigma=\sigma_{\bullet,\bullet}$ be a sequence of real numbers indexed by $D\times\Z_{>0}$. Let $\kappa\in\R_{\ge0}$. We define \emph{the $(\kappa,\sigma)$-degree} of a parabolic bundle $\bE=(E,E_{\bullet,\bullet})$ by
\[
 \deg_{\kappa,\sigma}\bE:=\kappa\deg E+\sum_{x\in D}\sum_{j>0}\sigma_{x,j}(\dim E_{x,j-1}-\dim E_{x,j})\in\R.
\]
If $\bE\ne0$, we define the \emph{$(\kappa,\sigma)$-slope} of $\bE$ as $\deg_{\kappa,\sigma}\bE/\rk\bE$.

We say that a sequence $\sigma=\sigma_{\bullet,\bullet}$ of real numbers indexed by $D\times\Z_{>0}$ is a \emph{sequence of parabolic weights} if for all $x\in D$ we have
\begin{equation}\label{eq:StabCond}
 \sigma_{x,1}\le\sigma_{x,2}\le\cdots
\end{equation}
and for all $x$ and $j$ we have $\sigma_{x,j}\le\sigma_{x,1}+1$. Let $\sigma$ be a sequence of parabolic weights. Let $\bE=(E,E_{\bullet,\bullet})$ be a parabolic bundle. Let $F\subset E$ be a saturated vector subbundle (that is, $E/F$ is torsion free). Set $F_{x,j}:=F_x\cap E_{x,j}$. Then $\bF:=(F,F_{\bullet,\bullet})$ is a parabolic bundle. We say that a Higgs bundle~$(\bE,\Phi)$ is \emph{$\sigma$-semistable}, if for all saturated subbundles $F$ of $E$ preserved by~$\Phi$ the $(1,\sigma)$-slope of the corresponding parabolic bundle $\bF$ is less than or equal to that of $\bE$. We have an open substack $\Higgs_\gamma^{\sigma-{\rm ss}}(\zeta)$ of $\Higgs_\gamma(\zeta)$ classifying $\sigma$-semistable parabolic Higgs bundles. This stack is of finite type over $\kk$; we will calculate its motivic class (see Section~\ref{sect:ExplAnswers} and Theorem~\ref{th:ExplAnsw}).

We note that condition~\eqref{eq:StabCond} is imposed on $\sigma$ to ensure that we have Harder--Narasimhan filtrations for parabolic Higgs bundles. We also note that, scaling $\kappa$ and $\sigma$ by the same positive real number scales all the slopes by the same number. This is why we restrict to the case $\kappa=1$ above (see Remark~\ref{rm:kappa} for more details).

\subsubsection{Semistable parabolic bundles with connections}\label{sect:IntroConnSS} We can also impose stability conditions on parabolic bundles with connections. Moreover, for non-resonant connections we can work with more general stability conditions, than those for Higgs bundles defined in the previous paragraph. A sequence $\zeta$ as above is called \emph{non-resonant} if for all $x\in X$ and all $i,j>0$ we have $\zeta_{x,i}-\zeta_{x,j}\notin\Z_{\ne0}$. Take $\kappa\in\R_{\ge0}$ and a sequence $\sigma$ of real numbers indexed by $D\times\Z_{>0}$ and satisfying condition~\eqref{eq:StabCond}.

\looseness=1 Assume that $\zeta$ is non-resonant. We define $(\kappa,\sigma)$-semistability of parabolic bundles with connections similarly to semistability of Higgs bundles but using the $(\kappa,\sigma)$-slope. Denote the corresponding moduli stack by $\Conn_\gamma^{(\kappa,\sigma)-{\rm ss}}(\zeta)$; this is an open substack of $\Conn_\gamma(\zeta)$. If $\zeta$ is resonant, then $\sigma$ has to satisfy some additional conditions (see Proposition~\ref{pr:HN2} and Re\-mark~\ref{rm:resonant}).

\subsection{Motivic Donaldson--Thomas invariants}\label{sect:DT} Our formulas for motivic classes of the moduli stacks above are all given in terms of certain motivic classes $\overline B_\gamma$ called \emph{motivic Donaldson--Thomas invariants} (see Section~\ref{sect:Aftermath} for this terminology), which we are going to define. First of all, we recall that in~\cite[Section~2]{FedorovSoibelmans} we defined (following earlier works \cite[Section~1]{Ekedahl09}, \cite{Joyce07}, and~\cite{KontsevichSoibelman08}) the ring of motivic classes of Artin stacks denoted $\Mot(\kk)$. We also defined its dimensional completion $\cMot(\kk)$. For an Artin stack $\cS$ of finite type over $\kk$ we have its motivic class $[\cS]\in\Mot(\kk)$. We denote its image in $\cMot(\kk)$ by the same symbol.

For a curve $X$ and a partition $\lambda$ we defined the series $J_\lambda^{\rm mot}(z),H_\lambda^{\rm mot}(z)\in\cMot(\kk)[[z]]$ in~\cite[Section~1.3.2]{FedorovSoibelmans}. The definitions (especially of $H_\lambda^{\rm mot}(z)$) are somewhat long, so we will not recall them here inviting the reader to look into~\cite{FedorovSoibelmans}. We only note that $J_\lambda^{\rm mot}(z)$ and $H_\lambda^{\rm mot}(z)$ are defined in terms of the motivic zeta-function of~$X$ (cf.~\eqref{eq:MotZeta} below). In particular, they only depend on~$X$ but not on~$D$. In this paper, we will denote them by $J_{\lambda,X}^{\rm mot}(z)$ and $H_{\lambda,X}^{\rm mot}(z)$ respectively to emphasize that they depend on the curve $X$ and to ensure that they are not confused with motivic modified Macdonald polynomials $\tilde H_\lambda^{\rm mot}(w_\bullet;z)$ and with motivic Hall--Littlewood polynomials $H_\lambda^{\rm mot}(w_\bullet)$ defined below.

The modified Macdonald polynomials $\tilde H_\lambda(w_\bullet;q,z)$ are symmetric functions in variables $w_\bullet=(w_1,w_2,\dots)$ with coefficients in $\Z[q,z]$. In~\cite[Definition~2.5]{MellitPunctures} the modified Macdonald polynomials are defined as symmetric functions with coefficients in $\Q[q,z]$ but it is well-known that the coefficients are integers (see, e.g., \cite{HaglundEtAlOnMacdonaldPoly} and references therein). Note that, formally speaking, symmetric functions are not polynomials (they become polynomials upon plugging in $w_{N+1}=w_{N+2}=\dots=0$). Let $\bL=\big[\A_\kk^1\big]$ be the motivic class of the affine line. We denote by $\tilde H_\lambda^{\rm mot}(w_\bullet;z)$ the symmetric function with coefficients in $\Mot(\kk)[z]$ obtained from $\tilde H_\lambda(w_\bullet;q,z)$ by substituting $\bL$ for~$q$.
We denote their images in the ring of symmetric functions with coefficients in $\cMot(\kk)[z]$ by $\tilde H_\lambda^{\rm mot}(w_\bullet;z)$ as well.

Let $\Gamma_+$ denote the commutative monoid of sequences $(r,r_{\bullet,\bullet},d)$, where $r$ is a nonnegative integer, $r_{\bullet,\bullet}$ is a sequence of nonnegative integers indexed by $D\times\Z_{>0}$, $d$ is an integer, subject to the following conditions:
\begin{enumerate}\itemsep=0pt
\item[(i)] For all $x\in D$ we have $\sum\limits_{j=1}^{\infty}r_{x,j}=r$. In particular, $r_{x,j}=0$ for $j$ large enough.
\item[(ii)] If $r=0$, then $d=0$ (and so $r_{x,j}=0$ for all $x$ and $j$).
\end{enumerate}

The operation on $\Gamma_+$ is the componentwise addition. For $\gamma=(r,r_{\bullet,\bullet},d)\in\Gamma_+$ we set $\rk\gamma=r$. The significance of the monoid $\Gamma_+$ is that the class of a parabolic bundle $\bE$ defined by~\eqref{eq:class} is an element of $\Gamma_+$. We also need a submonoid $\Gamma_+'\subset\Gamma_+$ given by $d\le 0$. Consider the completed monoid ring $\Mot(\kk)[[\Gamma_+']]$, we write its elements as $\sum\limits_{\gamma\in\Gamma_+'}A_\gamma e_\gamma$, where $A_\gamma\in\Mot(\kk)$, $e_\gamma$ are basis vectors. It is convenient to identify $e_\gamma$ with a monomial
\[
 w^r\prod_{x\in D}\prod_{j=1}^{\infty}w_{x,j}^{r_{x,j}}z^d,
\]
where $w_{\bullet,\bullet}=(w_{x,j})$ is a sequence of variables indexed by $D\times\Z_{>0}$. Then we identify $\Mot(\kk)[[\Gamma_+']]$ with a subring of $\Mot(\kk)\big[\big[w,w_{\bullet,\bullet},z^{-1}\big]\big]$. Similarly, we consider the completed monoid ring $\cMot(\kk)[[\Gamma_+']]$. We note that these rings are closely related to completed quantum tori considered in~\cite{KontsevichSoibelman08,KontsevichSoibelman10}. In our case, they are commutative, essentially because we are working with 2-dimensional Calabi--Yau categories; see Sections~\ref{sect:Aftermath} and~\ref{sect:CatsOverPar} for more details.

Finally, we need the notion of \emph{plethystic exponent and logarithm}. Let $\Mot(\kk)[[\Gamma_+']]^0$ denote the subset of $\Mot(\kk)[[\Gamma_+']]$ consisting of elements with zero constant terms. Then we have a~bijection
\[
 \Exp:\Mot(\kk)[[\Gamma_+']]^0\to1+\Mot(\kk)[[\Gamma_+']]^0
\]
called the plethystic exponent. We refer the reader to Section~\ref{sect:Plethystic} for the definition. Let the plethystic logarithm $\Log$ be the inverse bijection. Let us write
\[
 \bL\cdot\Log\left(\sum_\lambda w^{|\lambda|} J_{\lambda,X}^{\rm mot}\big(z^{-1}\big)H_{\lambda,X}^{\rm mot}\big(z^{-1}\big)\prod_{x\in D}\tilde H_\lambda^{\rm mot}\big(w_{x,\bullet};z^{-1}\big)\right)=
 \sum_{\gamma\in\Gamma'_+}\overline B_\gamma e_\gamma,
\]
where the sum in the LHS is over all partitions. We call the elements $\overline B_\gamma$ the \emph{Donaldson--Thomas invariants}. Note that $\overline{B}_0=0$.

When $X=\P^1_\kk$, we can define motivic Donaldson--Thomas invariants $B_\gamma$ by a simpler formula valid in $\Mot(\kk)$:
\begin{equation}\label{eq:DT_P1intro}
 \bL\cdot\Log\left(\sum_\lambda\frac{w^{|\lambda|}\prod\limits_{x\in D}\tilde H_\lambda^{\rm mot}\big(w_{x,\bullet};z^{-1}\big)}
 {\prod\limits_{h\in\Hook(\lambda)}
 \big(\bL^{a(h)}-z^{-l(h)-1}\big)\big(\bL^{a(h)+1}-z^{-l(h)}\big)}\right)=
 \sum_{\gamma\in\Gamma'_+}B_\gamma e_\gamma,
\end{equation}
where $\Hook(\lambda)$ stands for the set of hooks of $\lambda$, $a(h)$ and $l(h)$ stand for the armlength and the leglength of the hook $h$ respectively. We show that for $X=\P^1$ the images of $B_\gamma$ in $\cMot(\kk)$ are equal to $\overline B_\gamma$.

\subsection{Explicit formulas}\label{sect:ExplAnswers} The following explicit formulas for the motivic classes are parts of Theorem~\ref{th:ExplAnsw2}, Theorem~\ref{th:ExplAnsw}, and Theorem~\ref{th:ExplAnsw3} respectively. Let $\gamma=(r,r_{\bullet,\bullet},d)\in\Gamma_+$, $\gamma\ne0$. Let $\zeta$ be as in Section~\ref{sect:IntroConn}. For $\kappa\in\kk$ we define the $(\kappa,\zeta)$-degree and the $(\kappa,\zeta)$-slope of parabolic bundles similarly to $(\kappa,\sigma)$-degree and $(\kappa,\sigma)$-slope defined in Section~\ref{sect:IntroConnSS}; the only difference is that the $(\kappa,\zeta)$-degree and $(\kappa,\zeta)$-slope take values in~$\kk$.

For each $\tau\in\kk$, define the elements $C_\gamma(\zeta)\in\cMot(\kk)$, where $\gamma$ ranges over elements of $\Gamma_+'$ such that $\gamma=0$ or the $(1,\zeta)$-slope of $\gamma$ is $\tau$, by the following formula
\begin{equation*}
\sum_{\substack{\gamma\in\Gamma_+'\\ \deg_{1,\zeta}\gamma=\tau\rk\gamma}}\bL^{-\chi(\gamma)}C_\gamma(\zeta)e_\gamma=
 \Exp\left(\sum_{\substack{\gamma\in\Gamma_+'\\ \deg_{1,\zeta}\gamma=\tau\rk\gamma}}
 \overline B_\gamma e_\gamma
 \right),
\end{equation*}
where $\chi(\gamma):=(g-1)r^2+\sum\limits_{x\in D}\sum\limits_{j<j'}r_{x,j}r_{x,j'}$, $g$ is the genus of~$X$. Let $\gamma\in\Gamma_+$ be such that $\deg_{1,\zeta}\gamma=0$. Then, according to Theorem~\ref{th:ExplAnsw2}, the stack $\Conn_\gamma(\zeta)$ is of finite type over $\kk$ and we have in $\cMot(\kk)$
\[
 [\Conn_\gamma(\zeta)]=C_{(r,r_{\bullet,\bullet},d-Nr)}(\zeta),
\]
whenever $N$ is large enough. If $\deg_{1,\zeta}\gamma\ne0$, then the stack $\Conn_\gamma(\zeta)$ is empty.

Next, assume that $\zeta$ and $\sigma$ are as in Section~\ref{sect:IntroHiggs}, for $\tau\in\R$ define the elements $H_\gamma(\zeta,\sigma)\in\cMot(\kk)$ by the following formula
\begin{equation*}
\sum_{\substack{\gamma\in\Gamma_+'\\ \deg_{0,\zeta}\gamma=0\\ \deg_{1,\sigma}\gamma=\tau\rk\gamma}}\bL^{-\chi(\gamma)}H_\gamma(\zeta,\sigma)e_\gamma=
 \Exp\left(\sum_{\substack{\gamma\in\Gamma_+'\\ \deg_{0,\zeta}\gamma=0\\ \deg_{1,\sigma}\gamma=\tau\rk\gamma}}
 \overline B_\gamma e_\gamma
 \right).
\end{equation*}
Assume that $\deg_{0,\zeta}\gamma\!=\!0$, where $\gamma\in\Gamma_+'$. Then, according to Theorem~\ref{th:ExplAnsw}, the stack $\Higgs_\gamma^{\sigma-{\rm ss}}(\zeta)$ is of finite type over $\kk$ and we have
\[
 \big[\Higgs_\gamma^{\sigma-{\rm ss}}(\zeta)\big]=H_{(r,r_{\bullet,\bullet},d-Nr)}(\zeta,\sigma),
\]
whenever $N$ is large enough. If $\deg_{0,\zeta}\gamma\ne0$, then the stack $\Higgs_\gamma^{\sigma-{\rm ss}}(\zeta)$ is empty.

Finally, assume that $\zeta$ and $(\kappa,\sigma)$ are as in Section~\ref{sect:IntroConnSS}. For $\tau\in\kk$, $\tau'\in\R$ define the elements $C_\gamma(\zeta,\kappa,\sigma)\in\cMot(\kk)$ by the following formula
\begin{equation*}
\sum_{\substack{\gamma\in\Gamma_+'\\ \deg_{1,\zeta}\gamma=\tau\rk\gamma\\ \deg_{\kappa,\sigma}\gamma=\tau'\rk\gamma }}\bL^{-\chi(\gamma)}C_\gamma(\zeta,\kappa,\sigma)e_\gamma=
 \Exp\left(\sum_{\substack{\gamma\in\Gamma_+'\\ \deg_{1,\zeta}\gamma=\tau\rk\gamma\\ \deg_{\kappa,\sigma}\gamma=\tau'\rk\gamma }}
 \overline B_\gamma e_\gamma
 \right).
\end{equation*}
Let $\gamma\in\Gamma_+$ be such that $\deg_{1,\zeta}\gamma=0$. Then, according to Theorem~\ref{th:ExplAnsw3}, we have
\[
 \big[\Conn_\gamma^{(\kappa,\sigma)-{\rm ss}}(\zeta)\big]=C_{(r,r_{\bullet,\bullet},d-Nr)}(\zeta,\kappa,\sigma),
\]
whenever $N$ is large enough. If $\deg_{1,\zeta}\gamma\ne0$, then the stack $\Conn_\gamma^{(\kappa,\sigma)-{\rm ss}}(\zeta)$ is empty.

If $X=\P^1$, we get similar results valid in $\Mot(\kk)$, by replacing $\overline B_\gamma$ with $B_\gamma$ defined by a~simpler formula~\eqref{eq:DT_P1intro}.

\begin{Remark} We note that each of the above motivic classes depends only on finitely many DT-invariants. Indeed, $[\Conn_\gamma(\zeta)]$, $\big[\Higgs_\gamma^{\sigma-{\rm ss}}(\zeta)\big]$, and $\big[\Conn_\gamma^{(\kappa,\sigma)-{\rm ss}}(\zeta)\big]$ depend only on $\overline B_{\gamma'}$ with $\rk\gamma'\le\rk\gamma$ and for a given $\gamma$ there are only finitely many such $\gamma'\in\Gamma'_+$.
\end{Remark}

\begin{Remark} We note also that all the stacks whose motivic classes we are calculating are of finite type over $\kk$, so their motivic classes are defined in $\Mot(\kk)$. However, we can only calculate their motivic classes in $\cMot(\kk)$ except when $X=\P^1$. The reason is that, our calculation is based on the calculation of motivic classes of stacks of vector bundles on $X$ (without parabolic structures) with nilpotent endomorphisms. This calculation is performed in~\cite{FedorovSoibelmans} and is, in turn, based on the motivic analogue of Harder's residue formula (see~\cite[Theorem~1.5.1 and Section~4]{FedorovSoibelmans} and~\cite[Theorem~2.2.3]{HarderAnnals}). This formula, which is essentially saying that ``all vector bundles have essentially the same motivic number of Borel reductions'' involves some limiting process and is, therefore, only valid in the completed ring~$\cMot(\kk)$.
\end{Remark}

\subsection{Aftermath}\label{sect:Aftermath} In Section~\ref{sect:DT} we defined the classes $\overline B_\gamma$. These classes should be thought of as the Donaldson--Thomas invariants of the stack $\Higgs(0)$ of parabolic Higgs bundles with nilpotent residues. Note that this stack is the cotangent bundle of $\Bun^{\rm par}(X,D)$, while the stacks $\Higgs(\zeta)$ and $\Conn(\zeta)$ are \emph{twisted cotangent bundles}. We emphasize that $\overline B_\gamma$ do not depend on~$\zeta$,~$\kappa$, and~$\sigma$. The meaning of the formulas in Section~\ref{sect:ExplAnswers} is that the Donaldson--Thomas invariants of these twisted cotangent bundles are obtained by restricting the range of $\gamma$ to the submonoid $\deg_{0,\zeta}\gamma=0$ in the case of $\Higgs(\zeta)$ and to the submonoid $\deg_{1,\zeta}\gamma=0$ in the case of $\Conn(\zeta)$.

Another feature of the formulas is that the motivic classes of the stacks depend on \emph{equations} satisfied by~$\kappa$ and~$\sigma$ rather than on inequalities. In other words, there is no \emph{wall-crossing} in our case. This is not very surprising, as the category $\Higgs(0)$, being a cotangent bundle of $\Bun^{\rm par}(X,D)$, is a 2-dimensional Calabi--Yau category, cf.~\cite{RenSoibelman}.

One can speculate that similar results should be valid for the twisted cotangent stacks to the moduli stack of objects of any reasonable 1-dimensional category. Note that such cotangent stacks were studied by G.~Dobrovolska, V.~Ginzburg, and R.~Travkin in~\cite{DobrovolskaGinzburgTravkin}.

Another example of such a twisted cotangent stack is the category of vector bundles with irregular connections and appropriate level structures. This example is certainly more complicated as the corresponding abelian category has infinite homological dimension. We hope to return to this question in subsequent publications.

The formulas in Section~\ref{sect:ExplAnswers} are explicit but complicated. However, one can see that the motivic classes under considerations belong to the sub-$\lambda$-ring of $\cMot(\kk)$ generated by $\bL$, $X$, and the inverses of $\bL^i-1$ for $i\ge1$. We note that this ring is probably \emph{strictly larger}, than the subring of $\cMot(\kk)$ generated by $\bL$, the symmetric powers $X^{(i)}$, and the inverses of $\bL^i-1$ for $i\ge1$. The reason is that $\cMot(\kk)$ is unlikely to be a special $\lambda$-ring, see Section~\ref{sect:OtherRel}. On the other hand, if $X=\P^1$, then all our motivic classes are rational functions in $\bL$ with denominators being products of $\bL^i-1$ for $i\ge1$.

\subsection{Other results} It is clear from the above formulas that we have a lot of equalities between different motivic classes of Higgs bundles and bundles with connections. In particular, we show in Propositions~\ref{pr:ConnUniversal} and~\ref{pr:ConnUniversal2} that every motivic class of the form $\big[\Higgs_\gamma^{\sigma-{\rm ss}}(\zeta)\big]$ or $\big[\Conn_\gamma^{(\kappa,\sigma)-{\rm ss}}(\zeta)\big]$ is equal to some motivic class of the form $[\Conn_\gamma(\zeta)]$, provided that $\kk$ is not a finite extension of $\Q$. As a~consequence, we derive from results of Crawley-Boevey~\cite{CrawleyBoeveyIndecompPar} a criterion of non-emptiness of our moduli stacks. It is not difficult to see that if $X\ne\P^1_\kk$, then the stack $\big[\Higgs_\gamma^{\sigma-{\rm ss}}(\zeta)\big]$ is non-empty if and only if $\deg_{0,\zeta}\gamma=0$, while the stack $\big[\Conn_\gamma^{(\kappa,\sigma)-{\rm ss}}(\zeta)\big]$ is non-empty if and only if $\deg_{1,\zeta}\gamma=0$. For $X=\P^1_\kk$ the question is much more subtle and is related to the so-called Deligne--Simpson problem. This problem was originally stated for $\kk =\C$ in~\cite{SimsponProducts}. It may be reformulated for an arbitrary algebraically closed field $\kk$ of characteristic~$0$ as follows: given a~sequence of $\gl_r$-conjugacy classes $C_\bullet$ indexed by $D$, does there exist a pair $(E, \nabla)$ consisting of a~rank~$r$ vector bundle and a connection $\nabla$ on $E$ with poles bounded by $D$ such that $\Res_x \nabla\in C_x$ for all $x\in D$? Bundles with connections $(E,\nabla)$ parameterized by $\Conn_{\gamma}\big(\P^1_\kk,D,\zeta\big)$ are exactly bundles with connections such that each residue $\Res_x\nabla$ lies in the closure of a conjugacy class determined by $\gamma$ and $\zeta$ (see~\cite{Crawley-Boevey:Indecomposable}). If this conjugacy class is semisimple for each $x\in D$, then the elements of $\Conn_{\gamma}\big(\P^1_\kk,D,\zeta\big)$ are the solutions of the corresponding Deligne--Simpson problem. For a comprehensive survey of the Deligne--Simpson problem, see \cite{Crawley-Boevey:Indecomposable,Kostov2004Deligne,Simpson:MiddleConv}.

We note that if $\kk$ is not algebraically closed, one can ask a subtler question of whether there is a $\kk$-rational point in $\big[\Higgs_\gamma^{\sigma-{\rm ss}}(\zeta)\big]$ or $\big[\Conn_\gamma^{(\kappa,\sigma)-{\rm ss}}(\zeta)\big]$. We do not know the answer to this question. A somewhat similar question for moduli spaces of quiver representations is considered by V.~Hoskins and F.~Schaffhauser in~\cite{hoskins2017rational}.

At this point, we would like to emphasize that working with not necessarily algebraically closed fields is inevitable in the motivic setup: even if one is only interested in the case $\kk=\C$, one still has to consider all fields of characteristic zero, see Remark~\ref{rm:ArbFields} below.

One of the motivations for this work is the non-abelian Hodge theory of C.~Simpson (see~\cite{SimpsonHarmonicNoncompact}). In this paper, Simpson constructs an equivalence between a category of parabolic bundles with connections and a category of Higgs bundles. In Proposition~\ref{pr:Simpson}, we show that the corresponding stacks have equal motivic classes. We note that neither statement can be derived from the other (cf.~\cite[Remark~1.2.2]{FedorovSoibelmans}). We note also, that it is clear from our results that there are many more equalities of motivic classes, than those that one can guess from the non-abelian Hodge theory. We remark that V.~Hoskins and S.~Pepin Lehalleur have shown in~\cite[Theorem~4.2]{HoskinsLehalleurOnVoevodskyMotive} that the Voevodsky motives of the coarse moduli spaces of bundles with connections and of Higgs bundles are equal in the case when the rank and the degree are coprime. One can ask whether this can be upgraded to the parabolic situation.

\subsection{Other relations with previous work}\label{sect:OtherRel}
We have already noted that our results are closely related with the results of Mellit~\cite{MellitPunctures}. One difference is that Mellit counts the weighted number of points over a finite field, while we work over a field of characteristic zero and calculate motivic classes. Mellit counts the volumes of moduli stacks of Higgs bundles but is not considering bundles with connections. Another difference is that Mellit is not fixing the eigenvalues of Higgs fields.

\looseness=-1 On the other hand, Mellit's answers are simpler as they do not involve Schiffmann's polynomials~$H_\lambda$. In fact, Mellit's simplification of Schiffmann's formula from~\cite{SchiffmannIndecomposable} has nothing to do with parabolic structures. This simplification is the content of Mellit's papers~\cite{MellitIntegrality,MellitPunctures,MR4105090}. We believe that this simplification \emph{does not go through} in the motivic case because Mellit is using the fact that the $\lambda$-ring structure on symmetric functions is special (which, roughly speaking, means that $\lambda_\bullet(xy)$ and $\lambda_\bullet(\lambda_\bullet(x))$ can be expressed in terms of $\lambda_\bullet(x)$ and $\lambda_\bullet(y)$). It is known (see~\cite{LarsenLuntsRational}) that the Grothendieck $\lambda$-ring of varieties is not a special $\lambda$-ring. We do not know whether the Grothendieck $\lambda$-ring of \emph{stacks} $\Mot(\kk)$ and its completion $\cMot(\kk)$ are special. In any case, if one replaces $\cMot(\kk)$ by some its quotient that is a special $\lambda$-ring, then one expects that the Mellit's simplifications are valid in this quotient. Examples of such quotients are the Grothendieck ring of the category of Chow motives and the maximal special quotient of $\cMot(\kk)$ (see~\cite{LarsenLuntsRational}).

The paper~\cite{ChuangDiaconescuDonagiPantev}, although conjectural, contains an alternative approach to the problem via upgrading the computation of the motivic class of Higgs bundles to the problem about motivic Pandharipande--Thomas invariants on the non-compact Calabi--Yau 3-fold associated with the spectral curve.

Note also that Mozgovoy and Schiffmann in~\cite{MozgovoySchiffmanOnHiggsBundles} consider Higgs bundles with a twist by an arbitrary line bundle of degree at least $2g-2$, where $g$ is the genus of $X$. However, they do not consider parabolic structures and do not fix eigenvalues.

Finally we note that the general philosophy of Donaldson--Thomas invariants and the approach via motivic and cohomological Hall algebras (see~\cite{KontsevichSoibelman08,KontsevichSoibelman10}) are applicable to our situation. For more details about the approach that uses motivic Hall algebras we refer the reader to~\cite[Section~1.6, Remark~3.6.3]{FedorovSoibelmans}.

\subsection{Organization of the article}
In Section~\ref{sect:ParBundles} we define the category $\Bun^{\rm par}(X,D)$ of parabolic bundles and its graded stack of objects denoted by the same letter. Most of our stacks below will be stacks over $\Bun^{\rm par}(X,D)$.

In Section~\ref{sect:ParPairs} we study the stack of bundles with endomorphisms. This stack is the main intermediate object in our calculations. First, we calculate the motivic classes of stacks of parabolic bundles with nilpotent endomorphisms with fixed generic type. The calculation is based on Theorem~\ref{th:Factorization}, saying that these motivic classes are products of motivic classes of similar stacks without parabolic structures and of ``local stacks'' independent of the curve. This is a motivic analogue of~\cite[Theorem~5.6]{MellitPunctures}. However, the proof in the motivic case is significantly more involved and, hopefully, more conceptual.

The motivic classes of stacks of vector bundles (without parabolic structures) with nilpotent endomorphisms are calculated in the proof of~\cite[Theorem~1.4.1]{FedorovSoibelmans}. Since the motivic classes of ``local stacks'' are independent of the curve, it is enough to calculate them for $X=\P^1$; we do this using the ideas and results of Mellit~\cite{MellitPunctures}.

Then we use the formalism of plethystic powers to calculate the motivic classes of stacks of parabolic bundles with arbitrary endomorphisms.

In Section~\ref{sect:HiggsnEigenval} we study parabolic Higgs bundles with fixed eigenvalues. If the eigenvalues are equal to zero, then the endomorphisms of a given parabolic bundle and the Higgs fields on this bundle are parameterized by vector spaces whose dimensions differ by some Euler characteristic. Thus, it is easy to relate the motivic classes of the two stacks. If the eigenvalues are not zero, then not every parabolic bundle admits a Higgs fields with these eigenvalues. We give a criterion for existence of such a Higgs field in Lemma~\ref{lm:existence}. This allows us to express the motivic class of parabolic Higgs bundles with fixed eigenvalues in terms of the motivic class of the so-called isoslopy parabolic bundles with endomorphisms (see Proposition~\ref{pr:Sasha}).

The motivic class of isoslopy parabolic bundles with endomorphisms is derived from the results of Section~\ref{sect:ParPairs} with the help of a factorization formula (see Proposition~\ref{pr:IsoslProd}). This is analogous to~\cite[Proposition~3.5.1]{FedorovSoibelmans}.

In Section~\ref{sect:Stability} we use a version of Kontsevich--Soibelman factorization formula to calculate the motivic classes of stacks of semistable Higgs bundles. These depend on two sets of parameters: the eigenvalues and the stability condition. Somewhat surprisingly, these two sets come symmetrically in the answer.

Up to Section~\ref{sect:Stabilization} we work with nonpositive vector bundles, that is, vector bundles having no subbundle of positive degree. Without stability this restriction is inevitable as otherwise the moduli stacks would have infinite motivic volume. With a stability condition we can drop this technical restriction; the motivic classes of semistable parabolic Higgs bundles whose underlying vector bundles are not necessarily nonpositive are calculated in Section~\ref{sect:Stabilization}.

In Section~\ref{sect:Conn} we study the moduli stack of bundles with connections~-- with or without stability condition. The strategy for connections is similar to that for Higgs bundles except that the corresponding stacks are of finite type over $\kk$ even without stability conditions. So we first calculate the motivic classes of bundles with connections with given eigenvalues without stability conditions and without nonpositivity assumptions and then use a version of Kontsevich--Soibelman factorization formula to calculate the motivic classes of stacks of semistable bundles with connections.

{\samepage In Section~\ref{sect:NonEmpty} we give a precise criterion for non-emptiness of moduli stacks of Higgs bundles, bundles with connections, or semistable bundles with connections. The idea is that such a~stack is non-empty if and only if its motivic class is non-zero. Using our explicit formulas we re-write each of our motivic classes as the motivic class of a stack of bundles with connections (without stability conditions). The non-emptiness of such a stack is decided using Lemma~\ref{lm:existence2} and Crawley-Boevey's result~\cite[p.~1334, Corollary]{CrawleyBoeveyIndecompPar}.

}

\section{Preliminaries}
\subsection{Conventions}
We denote by $\kk$ a field of characteristic zero. We denote by $X$ a smooth projective geometrically connected curve over $\kk$ (recall that geometric connectedness means that $X$ remains connected after the base change to an algebraic closure of $\kk$). We denote by $D$ a set of $\kk$-rational points of $X$ and by $\deg D$ the number of elements of $D$.

If $E$ is a vector space or a vector bundle, we denote by $E^\vee$ the dual vector space (resp.~vector bundle). We identify vector bundles with their sheaves of sections. If $F$ is a coherent sheaf, we denote by $\END(F)$ the sheaf of its endomorphisms, we have $\END(F)=F^\vee\otimes F$ if $F$ is a vector bundle.

\subsubsection{Partitions and nilpotent matrices} By a partition we mean a non-increasing sequence of integers $\lambda=\lambda_1\ge\lambda_2\ge\cdots$, where $\lambda_l=0$ for $l\gg0$. Set $|\lambda|:=\sum_i\lambda_i$. For partitions $\lambda$ and $\mu$ we set $\langle\lambda,\mu\rangle=\sum_i\lambda'_i\mu'_i$, where $\lambda'$ and $\mu'$ are the conjugate partitions.

We denote by $\gl_{r,\kk}$ or simply by $\gl_r$ the set of $r\times r$ matrices with entries in $\kk$. We say that a~nilpotent matrix $n\in\gl_{|\lambda|}=\gl_{|\lambda|,\kk}$ is of type $\lambda$, if for all $i\ge1$ we have $\dim\Ker n^i-\dim\Ker n^{i-1}=\lambda_i$. For each partition $\lambda$ choose a nilpotent matrix $n_\lambda$ of type $\lambda$. For concreteness, we can take for $n_\lambda$ the direct sum of nilpotent Jordan blocks, where the number of blocks of size $i\times i$ is equal to $\lambda_i-\lambda_{i+1}$.

A sequence $(w_1,\dots,w_l,\dots)$ we denote by $w_\bullet$.

\subsubsection{Stacks} We will be working with stacks. All our stacks will have affine stabilizers. Our stacks will be Artin stacks locally of finite type over a field except in Section~\ref{sect:Factorization}, where we will have to work with stacks whose points have stabilizers of infinite type. For a stack $\cS$ we often abuse notation by writing $s\in\cS$ to mean that $s$ is an object of the groupoid $\cS(\kk)$, or an object of the groupoid $\cS(K)$, where $K$ is an extension of $\kk$. Following~\cite[Chapter~5]{LaumonMoretBailly} we say that a~$K'$-point~$\xi'$ of~$\cS$ is \emph{equivalent} to a $K''$-point $\xi''$ of $\cS$ if there is an extension $K\supset\kk$ and $\kk$-embeddings $K'\hookrightarrow K$, $K''\hookrightarrow K$ such that $\xi'_K$ is isomorphic to $\xi''_K$ (as an object of $\cS(K)$). The corresponding equivalence classes are called \emph{points of $\cS$}; the set of points is denoted by $|\cS|$. See~\cite[Section~2]{FedorovSoibelmans} for more details.

We write ``morphism of stacks'' to mean ``1-morphism of stacks''. We write ``of finite type'' to mean ``of finite type over $\kk$''.

\subsection{Motivic functions and motivic classes} Recall that in~\cite[Section~2]{FedorovSoibelmans} we defined (following~\cite[Section~1]{Ekedahl09},~\cite{Joyce07}, and~\cite{KontsevichSoibelman08}) the ring of motivic classes of Artin stacks denoted $\Mot(\kk)$.

More generally, for an Artin stack $\cX$ locally of finite type over $\kk$, we defined the $\Mot(\kk)$-module of motivic functions on $\cX$ denoted $\Mot(\cX)$. For a morphism $f\colon \cX\to\cY$ we have the pullback homomorphism $f^*\colon \Mot(\cY)\to\Mot(\cX)$. The pushforward homomorphism $f_!\colon \Mot(\cX)\to\Mot(\cY)$ is defined when $f$ is of finite type. We also defined the ring of completed motivic classes, denoted $\cMot(\kk)$, and $\cMot(\kk)$-modules of completed motivic functions $\cMot(\cX)$ with a~(probably non-injective) morphism $\Mot(\cX)\to\cMot(\cX)$. We also defined the pullbacks and the pushforwards of completed motivic functions.

We usually work with $\Mot(\kk)$ but our final results are formulated in $\cMot(\kk)$.

We defined the notion of a constructible subset of a stack. If $\cX\to\cY$ is a morphism of finite type, and $\cS\subset\cX$ is a constructible subset, we defined the motivic function $[\cS\to\cY]\in\Mot(\cY)$. Recall~\cite[Proposition~2.6.1]{FedorovSoibelmans}:
\begin{Proposition}\label{pr:MotFunEqual}
Assume that we are given $A,B\in\Mot(\cX)$ are such that for all field extensions $K\supset\kk$ and for all $\kk$-morphisms $\xi\colon \Spec K\to\cX$ we have $\xi^*A=\xi^*B$. Then $A=B$.
\end{Proposition}

\begin{Remark}\label{rm:ArbFields}
 The previous proposition is one of the reasons we have to work with arbitrary fields. Indeed, even if we start with $\kk=\C$, to be able to apply the proposition we have to consider all finitely generated extensions of~$\C$; see, for example, Section~\ref{sect:ProofFact}.
\end{Remark}

In Section~\ref{sect:NonEmpty} we will need the following proposition.
\begin{Proposition}\label{pr:NonEmpty}
 An Artin stack of finite type over $\kk$ is non-empty if and only if its motivic class in $\cMot(\kk)$ is not equal to zero.
\end{Proposition}
\begin{proof}
 The `if' direction is obvious. For the other direction assume for a contradiction that $\cS$ is a non-empty Artin stack of finite type over $\kk$ such that $[\cS]=0\in\cMot(\kk)$. This means that for all $m\in\Z$ we have $[\cS]\in F^m\Mot(\kk)$, where $F^\bullet$ is the dimensional filtration on $\Mot(\kk)$. According to~\cite[Propositions~3.5.6 and~3.5.9]{KreschStacks} every Artin stack of finite type with affine stabilizers has a~stratification by global quotients of the form $T/{\rm GL}_n$, where $T$ is a scheme. Thus, replacing $\cS$ with a stratification and clearing the denominators, we may assume that~$\cS$ is a disjoint union of integral affine schemes. Recall from~\cite[Section~2.5]{FedorovSoibelmans} that $\Mot(\kk)$ is the localization of the K-ring of varieties $\Mot_{\rm var}(\kk)$ with respect to the multiplicative set generated by $\bL$ and $\bL^i-1$, where $i>0$. Thus, multiplying $\cS$ by a certain product of these elements, we may assume that the class of $\cS$ in $\Mot_{\rm var}(\kk)$ belongs to the subgroup $F^{m-1}\Mot_{\rm var}(\kk)$ generated by the classes of the varieties of dimension at most $m-1$, where $m=\dim\cS$. Compactifying each top-dimensional connected component of $\cS$ and taking the resolution of singularities, we may assume that $\cS$ is the disjoint union of smooth projective $\kk$-varieties.

 Recall that the Hodge--Deligne polynomial of a smooth projective variety $Y$ is
 \[
 \sum_{p,q=0}^{\dim Y}(-1)^{p+q}h^{p,q}(Y)u^pv^q,
 \]
 where $h^{p,q}=\dim H^q(Y,\wedge^p\Omega_Y)$. This extends uniquely to a~homomorphism $E\colon \Mot_{\rm var}(\kk)\to\Z[u,v]$. Clearly, $E([Y])$ has degree $2m$, if $Y$ is a smooth projective variety of dimension~$m$. On the other hand, $E([Y])$ has degree at most $2m-2$, if $Y$ is any variety of dimension at most $m-1$. We see that, on the one hand $E(\cS)$ has degree~$2m$, on the other hand it has degree $2m-2$. We come to contradiction.
\end{proof}

\subsection{Principal bundles and special groups} Let $H$ be an algebraic group of finite type over $\kk$. Recall that a \emph{principal $H$-bundle} over a $\kk$-stack $\cB$ is a stack $E$ together with a schematic smooth surjective morphism of finite type $E\to\cB$ and an action $a\colon H\times_\kk E\to E$ such that $H$ acts simply transitively on the fibers of $E\to\cB$. More precisely, the simple transitivity means that the morphism $(a,p_2)\colon H\times E\to E\times_\cB E$ is an isomorphism, where $p_2\colon H\times E\to E$ is the projection. A principal bundle is \emph{trivial} if there is an isomorphism $E\approx H\times\cB$ compatible with the action and the projection to~$\cB$. If $E$ is a principal $H$-bundle over $\cB$ and $\cB'\to\cB$ is a morphism, then one gets an induced principal $H$-bundle $E\times_\cB\cB'$ over $\cB'$. The group $H$ is called \emph{special} if every principal $H$-bundle~$E$ over a~scheme $B$ of finite type over~$\kk$ is locally trivial over $B$ in the Zariski topology.

The following lemma is standard, see, e.g., \cite[Section~2.3]{BehrendDhillon}.
\begin{Lemma}\label{lm:SpecialMot}
Let $H$ be a special group and $E\to\cB$ be a principal $H$-bundle, where $\cB$ is an Artin stack of finite type over $\kk$. Then in $\Mot(\kk)$ we have $[E]=[H][\cB]$.
\end{Lemma}
\begin{proof}
 The case when $\cB$ is a scheme is easily proved by Noetherian induction. If $\cB$ is a stack, then, using again~\cite[Propositions~3.5.6 and~3.5.9]{KreschStacks}, we may assume that $\cB$ is a global quotient: $\cB=S/{\rm GL}_n$, where $S$ is a scheme. Then we have the cartesian diagram
 \begin{equation*}
 \begin{CD}
 E' @>>> E\\
 @VVV @VVV\\
 S @>>> \cB,
 \end{CD}
 \end{equation*}
 where $E'=E\times_\cB S$ and $E=E'/{\rm GL}_n$. Applying~\cite[Corollary~2.2.2]{FedorovSoibelmans} with $\cS=\Spec\kk$, we get $[E']=[E][{\rm GL}_n]$ and $S=[\cB][{\rm GL}_n]$. Next, $E'$ is a principal $H$-bundle over the scheme $S$ so $[E']=[H][S]$. Combining these equations we easily get the required statement.
\end{proof}

Recall that a $\kk$-group $U$ is \emph{unipotent} if it can be embedded into a group of strictly upper triangular matrices. Every unipotent subgroup is obtained from the additive group of the 1-dimensional vector space by iterated extensions.

\begin{Lemma}\label{lm:SpecialRad}
Let $H$ be an algebraic group of finite type over $\kk$ and let $U$ be a unipotent subgroup. Assume that $H/U$ is special. Then $H$ is special. In particular, every unipotent group is special.
\end{Lemma}
\begin{proof}
 Assume first that $H$ is a unipotent group. We claim that every principal $H$-bundle over an affine scheme is trivial. We prove this by induction on $\dim H$. If $\dim H=1$, then~$H$ is the additive group and the principal $H$-bundles over a scheme $B$ are classified by the coherent cohomology group $H^1(B,\cO_B)$, which vanishes as soon as $B$ is affine. If $\dim H>1$, then there is a subgroup $H'\subset H$ such that $\dim H'<\dim H$ and $\dim(H/H')<\dim H$. The groups~$H'$ and~$H/H'$ are unipotent. Recall that the principal $H$-bundles over $B$ are classified by the first non-abelian \'etale cohomology group $H_{\text{\'et}}^1(B,H)$. Now the statement follows from the exact sequence $H_{\text{\'et}}^1(B,H')\to H_{\text{\'et}}^1(B,H)\to H_{\text{\'et}}^1(B,H/H')$. In particular, every unipotent group is special.

 Now assume that $H/U$ is special, where $U$ is unipotent, and consider, for an affine $B$, the exact sequence $1=H_{\text{\'et}}^1(B,U)\to H_{\text{\'et}}^1(B,H)\to H_{\text{\'et}}^1(B,H/U)$. We see that a principal $H$-bundle is trivial over an affine scheme if the induced principal $H/U$-bundle is trivial. The lemma follows.
\end{proof}

\begin{Lemma}\label{lm:Zspecial}
 Let $Z_\lambda$ be the centralizer of $n_\lambda$ in ${\rm GL}_{|\lambda|}$. Then $Z_\lambda$ is a special group.
\end{Lemma}
\begin{proof}
 It is well-known that the quotient of $Z_\lambda$ by its unipotent radical is the product of ${\rm GL}_{r_i}$ for some $r_i\in\Z_{>0}$. It remains to note that ${\rm GL}_{r_i}$ are special groups and the product of special groups is special.
\end{proof}

\section{Parabolic bundles}\label{sect:ParBundles}
\subsection{Definitions and notations} Recall that $\kk$ denotes a field of characteristic zero, $X$ stands for a smooth projective geometrically connected curve over $\kk$, and $D$ is a set of rational points of $X$. We often assume that $D\ne\varnothing$; in this case $X$ has a divisor of degree one defined over $\kk$. We will often have to consider sequences indexed by $D\times\Z_{>0}$ or by $D\times\Z_{\ge0}$. A typical notation will be $r_{\bullet,\bullet}$. If $x\in D$, then $r_{x,\bullet}$ stands for the sequence $r_{x,1}, r_{x,2}, \dots$ (or $r_{x,0}, r_{x,1}, \dots$).

The monoid of all sequences $r_{\bullet,\bullet}$ indexed by $D\times\Z_{>0}$ with terms $r_{x,j}$ in a commutative monoid~$S$ and such that $r_{x,j}=0$ for $j\gg0$ (that is, functions on $D\times\Z_{>0}$ with finite support) will be denoted by $S[D\times\Z_{>0}]$.

\begin{Definition}
 A \emph{parabolic bundle} of type $(X,D)$ is a collection $(E,E_{\bullet,\bullet})$, where $E$ is a vector bundle over~$X$ and $E_{x,\bullet}$ is a flag in $E_x$ for $x\in D$:
 \[
 E_x=E_{x,0}\supseteq E_{x,1}\supseteq\dots\supseteq E_{x,l}\supseteq\cdots,\qquad E_{x,l}=0\quad \text{for} \ l\gg0.
 \]
\end{Definition}
We have the category $\Bun^{\rm par}(X,D)$ of parabolic bundles. We sometimes denote a parabolic bundle by a single boldface letter: $\bE=(E,E_{\bullet,\bullet})$. The morphism from $\bE$ to $\bE'$ is a morphism $\phi\colon E\to E'$ such that for all $x\in D$ and $j\ge0$ we have $\phi(E_{x,j})\subset E'_{x,j}$. This category is an additive $\kk$-linear category. The direct sum of $(E,E_{\bullet,\bullet})$ and $(E',E'_{\bullet,\bullet})$ is $(E\oplus E',E_{\bullet,\bullet}\oplus E'_{\bullet,\bullet})$. We note that the decomposition of a parabolic bundle into a direct sum of indecomposable parabolic bundles is unique up to isomorphism, while the isotypic summands are unique; the proof is similar to~\cite[Theorem~3]{Atiyah-KrullSchmidt} (see also~\cite[Proposition~3.1.2]{FedorovSoibelmans}).

We often skip $X$ and $D$ from the notation, writing $\Bun^{\rm par}$ instead of $\Bun^{\rm par}(X,D)$.

Abusing notation, we denote by $\Bun^{\rm par}$ the stack of objects of the category $\Bun^{\rm par}$. Precisely, if $S$ is a $\kk$-scheme, then $\Bun^{\rm par}(S)$ is the groupoid of collections $(E,E_{\bullet,\bullet})$, where $E$ is a vector bundle over $S\times_\kk X$, $E_{x,\bullet}$ is a filtration by vector subbundles of the restriction of $E$ to $S\times_\kk x$ for $x\in D$. Here by a subbundle of $E|_{S\times_\kk x}$ we mean a subsheaf that splits off as a direct summand Zariski locally over $S$.

\subsection[Monoids $\Gamma_+$ and $\Gamma'_+$]{Monoids $\boldsymbol{\Gamma_+}$ and $\boldsymbol{\Gamma'_+}$}\label{sect:Gamma} Consider the free abelian group $\Z\times\Z[D\times\Z_{>0}]\times\Z$ and its subgroup $\Gamma$ consisting of $(r,r_{\bullet,\bullet},d)$ such that for all $x\in D$ we have $\sum_{j=1}^{\infty}r_{x,j}=r$.

Let $\Gamma_+\subset\Gamma$ be the monoid of sequences $(r,r_{\bullet,\bullet},d)$ such that
\begin{enumerate}\itemsep=0pt
\item[(i)] $r\ge0$ and for all $x\in D$ and $j>0$ we have $r_{x,j}\ge0$;
\item[(ii)] if $r=0$, then $d=0$.
\end{enumerate}

Note that it follows from these conditions that $r=0$ implies that $(r,r_{\bullet,\bullet},d)$ is the zero sequence; we denote it by 0. Define the \emph{class function}:
\[
 \cl\colon \ \Bun^{\rm par}\to\Gamma_+,\qquad(E,E_{\bullet,\bullet})\mapsto(\rk E,\dim E_{x,j-1}-\dim E_{x,j},\deg E).
\]
For $\gamma=(r,r_{\bullet,\bullet},d)\in\Gamma$ we set $\rk\gamma:=r$. For a parabolic bundle $\bE$ we set $\rk\bE:=\rk\cl(\bE)$.

For $\gamma\in\Gamma_+$, we denote by $\Bun^{\rm par}_\gamma$ the stack of objects of class $\gamma$; this is an open and closed substack of $\Bun^{\rm par}$. It is often convenient to think of $\Bun^{\rm par}=\bigsqcup\limits_{\gamma\in\Gamma_+}\Bun_\gamma^{\rm par}$ as a $\Gamma_+$-graded stack. Note that $\Bun^{\rm par}_0$ has a single object: the parabolic bundle of rank zero.

Let $\gamma=(r,r_{\bullet,\bullet},d)\in\Gamma_+$. The projection $\Bun_\gamma^{\rm par}(X,D)\to\Bun_{r,d}(X)$ to the stack of rank~$r$ degree $d$ vector bundles on $X$ is schematic and of finite type (in fact, projective). Thus $\Bun_\gamma^{\rm par}(X,D)$ and $\Bun^{\rm par}(X,D)$ are Artin stacks locally of finite type.

Let us call a vector bundle on $X$ \emph{nonpositive} if it does not have a subbundle of positive degree. Recall that in~\cite[Section~3.2]{FedorovSoibelmans}) we called a vector bundle on $X$ HN-nonnegative, if its Harder--Narasimhan spectrum is nonnegative. by~\cite[Lemma~3.2.1(i)]{FedorovSoibelmans} there is an open substack $\Bun^+(X)\subset\Bun(X)$ classifying HN-nonnegative vector bundles. Moreover, by~\cite[Lemma~3.2.1(iii)]{FedorovSoibelmans} the substack of $\Bun^+(X)$ corresponding to vector bundles of rank~$r$ and degree~$d$ is of finite type (this substack was denoted by $\Bun_{r,d}^{\ge0}(X)$ in loc.~cit.).

By~\cite[Lemma~3.2.1(ii)]{FedorovSoibelmans} a bundle $E$ is HN-nonnegative if and only if it has no quotient bundles of negative degree. Thus a vector bundle $E$ is nonpositive if and only if its dual $E^\vee$ is HN-nonnegative. The following lemma is now clear.

\begin{Lemma}\label{lm:PosNeg}
There is an open substack $\Bun^-(X)$ of $\Bun(X)$ classifying nonpositive vector bundles. The assignment $E\mapsto E^\vee$ is an isomorphism between $\Bun^-(X)$ and $\Bun^+(X)$. The component of $\Bun^-(X)$ corresponding to vector bundles of fixed degree and rank is of finite type.
\end{Lemma}

Set
\[
 \Bun^{{\rm par},-}=\Bun^{{\rm par},-}(X,D):=\Bun^{\rm par}(X,D)\times_{\Bun(X)}\Bun^-(X).
\]
In other words, $\Bun^{{\rm par},-}$ is the stack (and the category) of parabolic bundles on $X$ whose underlying vector bundle is nonpositive. Set also
\[
 \Bun^{{\rm par},-}_\gamma=\Bun^{{\rm par},-}_\gamma(X,D):=\Bun^{\rm par}_\gamma(X,D)\times_{\Bun(X)}\Bun^-(X).
\]

\begin{Lemma}\label{lm:ParGamma}
For all $\gamma\in\Gamma_+$ the stack $\Bun^{{\rm par},-}_\gamma(X,D)$ is an Artin stack of finite type.
\end{Lemma}
\begin{proof}
Follows from Lemma~\ref{lm:PosNeg}.
\end{proof}

Let $\Gamma'_+$ be the submonoid of $\Gamma_+$ consisting of sequences with $d\le0$. Clearly, $\Bun_\gamma^{{\rm par},-}\ne\varnothing$ only if $\gamma\in\Gamma'_+$.

\subsection[Categories over $\Bun^{\rm par}$ and $\Gamma_+$-graded stacks]{Categories over $\boldsymbol{\Bun^{\rm par}}$ and $\boldsymbol{\Gamma_+}$-graded stacks}\label{sect:CatsOverPar}

We will consider below many categories with a forgetful functor to $\Bun^{\rm par}$ (e.g., the category of parabolic Higgs bundles). Let $\cC$ be such a category and denote by $\cC$ its stack of objects as well. Assume that the morphism $\cC\to\Bun^{\rm par}$ is of finite type. Define the stacks
\begin{gather}\label{eq:StacksOverPar}
 \cC_\gamma:=\cC\times_{\Bun^{\rm par}}\Bun^{\rm par}_\gamma,\qquad\!\!
 \cC^-:=\cC\times_{\Bun^{\rm par}}\Bun^{{\rm par},-},\qquad\!\!
 \cC_\gamma^-:=\cC\times_{\Bun^{\rm par}}\Bun^{{\rm par},-}_\gamma.\!\!
\end{gather}
By Lemma~\ref{lm:ParGamma}, the stack $\cC_\gamma^-$ is of finite type. The stack $\cC=\bigsqcup\limits_{\gamma\in\Gamma_+}\cC_\gamma$ is $\Gamma_+$-graded, while the stack $\cC^-=\bigsqcup\limits_{\gamma\in\Gamma_+}\cC^-_\gamma$ is $\Gamma'_+$-graded. Moreover, the stack $\cC^-$ is of finite type as a graded stack, that is, its graded components are stacks of finite type.

The group ring $\Mot(\kk)[\Gamma_+]$ is closely related to \emph{quantum tori} (cf.~\cite{KontsevichSoibelman08}). This has a natural basis $e_\gamma$, where~$\gamma$ ranges over $\Gamma_+$; the multiplication is given by $e_\gamma e_{\gamma'}=e_{\gamma+\gamma'}$. The reason for the multiplication in the quantum torus to be commutative is that we are actually working with 2-dimensional Calabi--Yau categories, and hence the skewsymmetrization of the Euler form vanishes (cf.~\cite{RenSoibelman}).

Let $\Mot(\kk)[[\Gamma_+]]$ be the completion of $\Mot(\kk)[\Gamma_+]$ (this can be viewed as the group of $\Mot(\kk)$-valued functions on $\Gamma_+$).

If $\cD$ is a $\Gamma_+$-graded stack of finite type, we consider the generating series
\begin{equation}\label{eq:MotDT}
 [\cD]:=\sum_{\gamma\in\Gamma_+}[\cD_\gamma]e_\gamma\in\Mot(\kk)[[\Gamma_+]].
\end{equation}
We call $[\cD]$ the \emph{graded motivic class} of the stack $\cD$. Recall that in~\cite{KontsevichSoibelman08} the motivic Donaldson--Thomas series was defined as an element of the completed motivic quantum torus. In our case, it associates to a $\Gamma_+$-graded stack $\cD$ an infinite series in the {\it commutative} motivic quantum torus corresponding to the monoid $\Gamma_+$ endowed with the trivial bilinear form. Thus, it coincides with~$[\cD]$.

Sometimes it is convenient to write $e_\gamma$ explicitly as
\[
 e_\gamma=w^r\prod_{x\in D}\prod_{j=1}^{\infty}w_{x,j}^{r_{x,j}}z^d\in\Z\big[\big[w,w_{\bullet,\bullet},z,z^{-1}\big]\big]\subset\Mot(\kk) \big[\big[w,w_{\bullet,\bullet},z,z^{-1}\big]\big],
\]
where $\gamma=(r,r_{\bullet,\bullet},d)$. Here $w=w_{\bullet,\bullet}$ stands for the collection of variables
\[
 w_{x,\bullet}:=(w_{x,1},w_{x,2},\dots,w_{x,j},\dots),\qquad x\in D.
\]
The variables $w$, $z$, and $w_{x,j}$ for $x\in D$, $j\ge1$ are commuting variables.

\begin{Remark}
Note that we do not fix the lengths of flags. Let us fix a function $l\colon D\to\Z_{>0}$ and consider only flags of length at most $l(x)$ at $x$. Let $\Gamma_{+,l}$ be the submonoid of $\Gamma_+$ consisting of sequences $(r,r_{\bullet,\bullet},d)$ such that $r_{x,j}=0$ whenever $j>l(x)$. We have the obvious projection $\Gamma_+\to\Gamma_{+,l}$, which induces a homomorphism $\Pi_l \colon \Mot(\kk)[[\Gamma_+]]\to\Mot(\kk)[[\Gamma_{+,l}]]$. Explicitly, this is just setting $w_{x,j}=0$ whenever $j>l(x)$. Let $\cD$ be as in~\eqref{eq:MotDT}, then $\Pi_l[\cD]=\sum\limits_{\gamma\in\Gamma_{+,l}}[\cD_\gamma]e_\gamma\in\Mot(\kk)[[\Gamma_{+,l}]]$ is the graded motivic class of the substack of~$\cD$ where the lengths of the flags are bounded by $l$. We see that the difference between fixing the length of flags and allowing flags of arbitrary lengths corresponds on the quantum torus side to the difference between polynomials in infinite number of variables and polynomials in finite number of variables.

In our applications, $[\cD]$ will be symmetric in each sequence of variables $w_{x,\bullet}$ (cf.~Remark~\ref{rm:Weyl}). In this case, the difference between fixing and not fixing the lengths corresponds on the side of motivic classes to the difference between symmetric polynomials and symmetric functions, cf.~\cite[Chapter~1, Section~2]{macdonald1998symmetric}.
\end{Remark}

We emphasize that $\Mot(\kk)[[\Gamma_+]]$ is not a ring. However, $\Mot(\kk)[[\Gamma'_+]]\subset\Mot(\kk)[[w,w_{\bullet,\bullet},z^{-1}]]$ is a ring and $[\cC^-]\in\Mot(\kk)[[\Gamma'_+]]$ whenever $\cC$ is a stack of finite type over $\Bun^{\rm par}$. This is in accordance to the general theory in~\cite{KontsevichSoibelman08}, where one fixes a strict sector in $\R^2$ in order to have well-defined Donaldson--Thomas invariants. We can also replace $\Mot(\kk)$ with $\cMot(\kk)$ in all the above constructions.

\subsection{Motivic zeta-functions and plethystic operations}\label{sect:Plethystic}
Following~\cite{KapranovMotivic}, for a variety $Y$ define its \emph{motivic zeta-funcion} by
\begin{equation}\label{eq:MotZeta}
 \zeta_Y(z):=\sum_{n=0}^\infty\big[Y^{(n)}\big]z^n\in\Mot(\kk)[[z]],
\end{equation}
where $Y^{(n)}=Y^n/\Sigma_n$ is the $n$-th symmetric power of $Y$ ($\Sigma_n$ denotes the group of permutations). Consider the group $(1+z\Mot(\kk)[[z]])^\times$, where the group operation is multiplication. According to~\cite[Theorem~2.3]{EkedahlLambdaStacks} $\zeta$ can be uniquely extended to a homomorphism
\[
 \zeta\colon \ \Mot(\kk)\to(1+z\Mot(\kk)[[z]])^\times\colon \ A\mapsto\zeta_A(z)
\]
such that we have $\zeta_{\bL^n A}(z)=\zeta_A\big(\bL^nz\big)$ for all $n\in\Z$ and $A\in\Mot(\kk)$. Clearly,
\[ \zeta_A(z)\equiv1+Ax\pmod {z^2}.\] Thus we have equipped $\Mot(\kk)$ with a $\lambda$-ring structure. Note that $\Mot(\kk)$ is \emph{not} a~special $\lambda$-ring, in particular, $\zeta(AB)$ cannot be expressed in terms of $\zeta(A)$ and $\zeta(B)$ (so some authors would call this a pre-$\lambda$-ring structure).

According to loc.~cit., this homomorphism $\zeta$ is continuous with respect to the dimensional filtration on $\Mot(\kk)$, so it extends to a homomorphism
\begin{equation*}
 \zeta\colon \ \cMot(\kk)\to\big(1+z\cMot(\kk)[[z]]\big)^\times,
\end{equation*}
which coincides with the one constructed in~\cite[Section~1.3.1]{FedorovSoibelmans}.

Let $\Mot(\kk)[[\Gamma_+']]^0$ stand for the series without constant term. We define the \emph{plethystic exponent}
$\Exp\colon \Mot(\kk)[[\Gamma_+']]^0\to(1+\Mot(\kk)[[\Gamma_+']]^0)^\times$ by
\[
 \Exp\left(\sum_{\gamma\in\Gamma_+'} A_\gamma e_\gamma\right)=\prod_{\gamma\in\Gamma_+'}\Exp(A_\gamma e_\gamma)=
 \prod_{\gamma\in\Gamma_+'}\zeta_{A_\gamma}(e_\gamma).
\]
One shows easily that this is an isomorphism of abelian groups. Denote the inverse isomorphism by $\Log$ (the \emph{plethystic logarithm}). Finally, we define the \emph{plethystic power} by
\[
 \Pow\colon \ \big(1+\Mot(\kk)[[\Gamma_+']]^0\big)\times\Mot(\kk)\to1+\Mot(\kk)[[\Gamma_+']]^0\colon \
 (f,A)\mapsto\Exp(A\Log(f)).
\]
We note that we can similarly define $\Exp$, $\Log$, and $\Pow$ for the completed ring $\cMot(\kk)$, which coincide with the operations defined in~\cite{FedorovSoibelmans} when $D=\varnothing$.

\subsection{Parabolic subbundles and quotient bundles}\label{sect:Subobjects} Let $\bE=(E,E_{\bullet,\bullet})$ be a parabolic bundle of type $(X,D)$. We say that $\bE'=(E',E'_{\bullet,\bullet})$ is a \emph{strict} parabolic subbundle of $\bE$ if $E'$ is a saturated subbundle of $E$ (that is, $E/E'$ is torsion free) and for all $x$ and $j$ we have $E'_{x,j}=E_{x,j}\cap E'_x$. Note that strict parabolic subbundles of $\bE=(E,E_{\bullet,\bullet})$ are in bijective correspondence with saturated subbundles of $E$. Let $\bE'=(E',E'_{\bullet,\bullet})$ be a strict parabolic subbundle of $\bE=(E,E_{\bullet,\bullet})$; set $E'':=E/E'$. Then we have a parabolic structure on $E''$ given by $E''_{x,j}:=E_{x,j}/E'_{x,j}$. We call the parabolic bundle $\bE/\bE':=(E'',E''_{\bullet,\bullet})$ the \emph{quotient parabolic bundle} of $\bE$. Thus, the quotient parabolic bundles of $\bE$ are also in bijective correspondence with saturated subbundles of $E$. Finally, in the above situation, we say that
\begin{equation}\label{eq:ExactSeq}
 0\to\bE'\to\bE\to\bE/\bE'\to0
\end{equation}
is a \emph{short exact sequence}. We also say that \emph{$\bE$ is an extension of $\bE/\bE'$ by $\bE'$.} It is clear that in this case we have $\cl(\bE)=\cl(\bE')+\cl(\bE/\bE')$. One can use the short exact sequences above to define the group $K_0(\Bun^{\rm par})$; the class function $\cl$ extends to $\cl\colon K_0(\Bun^{\rm par})\to\Gamma$.

\begin{Remark}The category $\Bun^{\rm par}$ is not abelian. It can be extended to an abelian category by viewing vector bundles with flags as coherent sheaves on orbifold curves. Then the abelian category is the category of coherent sheaves on this orbifold. This extension will not be used in the current paper. On the other hand, if we define short exact sequences in $\Bun^{\rm par}$ as sequences isomorphic to some sequence of the form~\eqref{eq:ExactSeq}, then $\Bun^{\rm par}$ becomes an exact category in the sense of Quillen.
\end{Remark}

Let $\phi\colon E\to F$ be a morphism of vector bundles on $X$. We say that $\phi$ is \emph{generically an isomorphism} if it is an isomorphism at the generic point of $X$. Equivalently, $\phi$ is an isomorphism over a non-empty Zariski open subset of $X$. Another reformulation is that $\phi$ is injective and~$F/\phi(E)$ is a torsion sheaf. Sometimes one says in this situation that $E$ is a lower modification of~$F$.

\begin{Definition}We say that a morphism of parabolic bundles $\phi\colon (E,E_{\bullet,\bullet})\to(F,F_{\bullet,\bullet})$ is \emph{generically an isomorphism} if the underlying morphism $E\to F$ is generically an isomorphism.
\end{Definition}

\begin{Lemma}\label{lm:MorPar} Let $\phi\colon \bE\to\bF$ be a morphism of parabolic bundles. Then there are strict parabolic subbundles $\bE'\subset\bE$ and $\bF'\subset\bF$ such that $\phi$ can be decomposed as
 \[
 \bE\xrightarrow{\phi_1}\bE/\bE'\xrightarrow{\phi_2}\bF'\xrightarrow{\phi_3}\bF,
 \]
 where $\phi_1$ is the canonical projection, $\phi_2$ is generically an isomorphism, $\phi_3$ is the canonical embedding.
\end{Lemma}
\begin{proof}
 Write $\bE=(E,E_{\bullet,\bullet})$, $\bF=(F,F_{\bullet,\bullet})$, let $\phi'\colon E\to F$ be the underlying morphism of vector bundles.
 Note that $\Ker\phi'$ is a vector subbundle of $E$. Indeed, $E/\Ker\phi'$ is isomorphic to a subsheaf of $F$, so it is torsion free.
 Let $\bE'$ be the strict subbundle of $\bE$ whose underlying vector bundle is $\Ker(\phi')$. Let $F'$ be the saturation of the image of $\phi'$ (that is, $F'$ is the unique saturated vector subbundle of $F$ containing $\phi'(E)$ such that the quotient $F'/\phi'(E)$ is a torsion sheaf). Let $\bF'$ be the strict subbundle of $\bF$ whose underlying vector bundle is $F'$. Now the existence of the decomposition is clear.
\end{proof}

\subsection{Generalized degrees and slopes}\label{sect:DegreeSlope} Let $A$ be a $\Q$-vector space (in applications it will be $\kk$ or $\R$). Let $\kappa\in A$, $\zeta=\zeta_{\bullet,\bullet}\in A[D\times\Z_{>0}]$. Then we define the homomorphism $\deg_{\kappa,\sigma}\colon \Gamma\to A$ by
\[
 \deg_{\kappa,\zeta}(r,r_{\bullet,\bullet},d)=\kappa d+\sum_{x\in D}\sum_{j>0}\zeta_{x,j}r_{x,j}.
\]
If $\rk\gamma\ne0$, we define the $(\kappa,\zeta)$-slope of $\gamma$ by $\deg_{\kappa,\zeta}\gamma/\rk\gamma$. We write $\deg_{\kappa,\zeta}\bE$ for $\deg_{\kappa,\zeta}\cl(\bE)$ and call $\deg_{\kappa,\zeta}\bE/\rk\bE$ \emph{the $(\kappa,\sigma)$-slope of $\bE$}.

We say that a parabolic bundle $\bE$ is \emph{$(\kappa,\zeta)$-isoslopy} if the $(\kappa,\zeta)$-slope of any direct summand of $\bE$ is equal to the $(\kappa,\zeta)$-slope of $\bE$.

We remark that it is common to write $\zeta\star\gamma$ for $\deg_{0,\zeta}\gamma$ and $\deg_\zeta\gamma$ for $\deg_{1,\zeta}\gamma$ but we prefer a uniform notation. We also remark that for an exact sequence~\eqref{eq:ExactSeq} we have $\deg_{\kappa,\zeta}\bE=\deg_{\kappa,\zeta}\bE'+\deg_{\kappa,\zeta}(\bE/\bE')$.

\subsection{Parabolic weights and stability conditions}\label{sect:ParWeights}
The following definition should be compared to~\cite[Definition~6.9]{MellitPunctures}.
\begin{Definition}\label{def:StabilityCond}
We say that a sequence $\sigma=\sigma_{\bullet,\bullet}$ of real numbers indexed by $D\times\Z_{>0}$ is a~\emph{sequence of parabolic weights} if for all $x\in D$ we have
\begin{equation}\label{eq:StabCond2}
 \sigma_{x,1}\le\sigma_{x,2}\le\cdots
\end{equation}
and for all $x$ and $j$ we have $\sigma_{x,j}\le\sigma_{x,1}+1$.
\end{Definition}
To every sequence of parabolic weights we will associate a notion of stability on parabolic bundles in Definition~\ref{def:Semistable} below. Thus we denote the set of all sequences of parabolic weights by $\Stab=\Stab(X,D)$.

Fix $\sigma\in\Stab$.

\begin{Definition}\label{def:Semistable}
A parabolic bundle $\bE$ is \emph{$\sigma$-semistable} if for all strict parabolic subbundles $\bE'\subset\bE$ we have
 \begin{equation}\label{eq:ss}
 \frac{\deg_{1,\sigma}\bE'}{\rk\bE'}\le\frac{\deg_{1,\sigma}\bE}{\rk\bE}.
 \end{equation}
\end{Definition}

\begin{Remark}\label{rm:kappa}
We can similarly define semistability for any $\kappa>0$ replacing the condition $\sigma_{x,j}\le\sigma_{x,1}+1$ with $\sigma_{x,j}\le\sigma_{x,1}+\kappa$. However, scaling $\kappa$ and $\sigma$ by the same positive real number scales all the slopes by the same number, so we would get the same notion of semistability. This is why we restrict to the case $\kappa=1$ above. The case $\kappa=0$ would not yield stacks of finite type; however, the possibility of taking $\kappa=0$ will be useful below, when we work with connections (see Section~\ref{sect:StabConn}).

\end{Remark}

\begin{Proposition}\label{pr:HN}
Let $\bE\in\Bun^{\rm par}$ be a parabolic bundle. Then there is a unique filtration $0=\bE_0\subset\bE_1\subset\dots\subset\bE_m=\bE$ by strict parabolic subbundles such that all the quotients $\bE_i/\bE_{i-1}$ are $\sigma$-semistable and we have $\tau_1>\dots>\tau_m$, where $\tau_i$ is the $(1,\sigma)$-slope of $\bE_i/\bE_{i-1}$.
\end{Proposition}
\begin{proof}
 We start with a Lemma.
 \begin{Lemma}\label{lm:ModifDegree}
 Let the morphism $\bE\to\bF$ be generically an isomorphism. Then $\deg_{1,\sigma}\bE\le\deg_{1,\sigma}\bF$.
 \end{Lemma}
 \begin{proof}
 Write $\bE=(E,E_{\bullet,\bullet})$ and $\bF=(F,F_{\bullet,\bullet})$. Let $\phi\colon E\to F$ be the underlying morphism of vector bundles. For $x\in D$ let $d_x$ denote the dimension of the kernel of $\phi_x$. Then
 \[
 \deg E=\deg F-\length(F/\phi(E))\le\deg F-\sum_{x\in D}d_x.
 \]
 On the other hand, for all $x\in D$ and $i>0$ we have $\dim E_{x,i}\le\dim F_{x,i}+d_x$. Hence
 \begin{align*}
 \deg_{1,\sigma}\bE& =\deg E+\sum_{x,j>0}\sigma_{x,j}(\dim E_{x,j-1}-\dim E_{x,j})\\
 & =\deg E+\sum_{x\in D}\left(\sigma_{x,1}\rk E+\sum_{i>0}(\sigma_{x,i+1}-\sigma_{x,i})\dim E_{x,i}\right)\\
& \le \deg F-\sum_{x\in D} d_x+\sum_{x\in D}\left(\sigma_{x,1}\rk F+\sum_{i>0}(\sigma_{x,i+1}-\sigma_{x,i})(\dim F_{x,i}+d_x)\right)\\
& = \deg_{1,\sigma}\bF+\sum_{x\in D}d_x\left(-1+\sum_{i>0}(\sigma_{x,i+1}-\sigma_{x,i})\right)\le\deg_{1,\sigma}\bF.
 \end{align*}
 Lemma~\ref{lm:ModifDegree} is proved.
 \end{proof}

The rest of the proof of Proposition~\ref{pr:HN} is completely analogous to the proof of~\cite[Section~1.3]{HarderNarasimhan} in view of Lemma~\ref{lm:MorPar}.
\end{proof}

In the situation of Proposition~\ref{pr:HN}(i) we say that the filtration is the \emph{Harder--Narasimhan filtration} of $E$ (or \emph{HN-filtration} for short) and $\tau_1>\dots>\tau_m$ is the $\sigma$-HN spectrum of $\bE$. We define $\Bun^{{\rm par},\le\tau}$ and $\Bun^{{\rm par},\ge\tau}$ as full subcategories of $\Bun^{\rm par}$ whose objects are parabolic bundles with the $\sigma$-HN spectrum contained in $(-\infty,\tau]$ (resp.~$[\tau,\infty)$). We emphasize that the categories $\Bun^{{\rm par},\le0}$ and $\Bun^{{\rm par},-}$ should not be confused with each other: they coincide only if $\sigma=0$. The following lemma is standard.
\begin{Lemma}\label{lm:NoMorphismSS}
Let $\bE$ be an object of $\Bun^{{\rm par},\le\tau}$ and $\bE'$ be an object of $\Bun^{{\rm par},\ge\tau'}$, where $\tau<\tau'$. Then $\Hom_{\Bun^{\rm par}}(\bE',\bE)=0$.
\end{Lemma}

\section{Parabolic pairs}\label{sect:ParPairs}
\subsection{Parabolic pairs and their generic Jordan types} The notion of parabolic pair, interesting by itself, will be used as a technical tool for studying parabolic Higgs bundles in Section~\ref{sect:HiggsnEigenval} and parabolic bundles with connections in Section~\ref{sect:Conn}. Our main results in this section are Theorem~\ref{th:MotMellitPunctures} and Corollary~\ref{cor:Pairs}. They give explicit answers for graded motivic classes of stacks of nilpotent parabolic pairs and parabolic pairs respectively. We will also give a simplified answer in the case $X=\P^1$ in Section~\ref{sect:P1ManyPts}.

\begin{Definition}
 A \emph{parabolic pair} $(\bE,\Psi)$ consists of a parabolic bundle
 \[
 \bE=(E,E_{\bullet,\bullet})\in\Bun^{\rm par}(X,D)
 \]
 and an endomorphism $\Psi$ of $\bE$ (that is, an endomorphism of $E$ preserving each $E_{x,j}$). If $\Psi$ is nilpotent we will speak about \emph{nilpotent parabolic pairs}.
\end{Definition}

Parabolic pairs as well as nilpotent parabolic pairs form an additive $\kk$-linear category denoted $\Pair=\Pair(X,D)$ (resp.~$\Pair^{\rm nilp}=\Pair^{\rm nilp}(X,D)$). Again, we abuse notation by denoting the stacks of objects by the same symbols. We define $\Pair_\gamma$, $\Pair_\gamma^{\rm nilp}$, $\Pair_\gamma^-$, $\Pair_\gamma^{{\rm nilp},-}$ etc.\ following the general construction~\eqref{eq:StacksOverPar} of Section~\ref{sect:CatsOverPar}.

The forgetful morphisms $\Pair\to\Bun^{\rm par}$ and $\Pair^{\rm nilp}\to\Bun^{\rm par}$ are schematic and of finite type. In particular, $\Pair^-$ and $\Pair^{{\rm nilp},-}$ are $\Gamma'_+$-graded Artin stacks of finite type (in the graded sense).

Let $K\supset\kk$ be an extension and $(E,E_{\bullet,\bullet},\Psi)\in\Pair^{\rm nilp}(K)$. If we trivialize $E$ at the generic point of $X_K=X\times_\kk\Spec K$, $\Psi$ becomes a $\rk E\times\rk E$ nilpotent matrix. Its Jordan type is a partition $\lambda$ of $\rk E$. Thus we get a locally closed stratification of $\Pair^{\rm nilp}$ according to the generic Jordan type of the nilpotent endomorphism
\[
 \Pair^{\rm nilp}(X,D)=\bigsqcup_\lambda\Pair^{\rm nilp}(X,D,\lambda),
\]
where the disjoint union is over all partitions. In other words, $\Pair^{\rm nilp}(X,D,\lambda)$ classifies nilpotent parabolic pairs such that the endomorphism is generically conjugate to $n_\lambda$ (that is, conjugate to $n_\lambda$ at the generic point of $X$, or, equivalently, at each point of a non-empty Zariski open subset of $X$). We remark that any endomorphism generically conjugate to $n_\lambda$ is necessarily nilpotent.

Again, we define the $\Gamma'_+$-graded stacks $\Pair^{{\rm nilp},-}(X,D,\lambda)$ using the general formalism~\eqref{eq:StacksOverPar} of Section~\ref{sect:CatsOverPar}.

\subsection{Motivic classes of parabolic bundles with nilpotent endomorphisms}\label{sect:NilpEnd} Our goal in this section is to calculate the graded motivic class (that is, the motivic Donaldson--Thomas series, cf.~\cite{KontsevichSoibelman08})
\begin{align}
 \big[\Pair^{{\rm nilp},-}(X,D,\lambda)\big]& =\sum_{\gamma\in\Gamma'_+}\big[\Pair_\gamma^{{\rm nilp},-}(X,D,\lambda)\big]e_\gamma\nonumber\\
& = w^{|\lambda|}\sum_{\gamma=(r,r_{\bullet,\bullet},d)\in\Gamma'_+}\big[\Pair_\gamma^{{\rm nilp},-}(X,D,\lambda)\big]\prod_{x,j}w_{x,j}^{r_{x,j}}z^d.\label{eq:Omega}
\end{align}
The partition $\lambda$ is fixed until the end of Section~\ref{sect:Factorization}. This graded motivic class is calculated as follows. First, in Section~\ref{sect:Factorization} we write this graded motivic class as the product of a term that is independent of the parabolic structures, and the ``local'' terms independent of the curve. The first term has been calculated in~\cite{FedorovSoibelmans}. Since the local terms are independent of the curve, it is enough to calculate them when $X=\P^1$. More precisely, we will work with $\P^1$ and two points with parabolic structures (that is, $D=\{0,\infty\}$) but we will calculate the sum over all partitions (Section~\ref{sect:P1}). This part is very similar to~\cite[Section~5.4]{MellitPunctures}. In Section~\ref{sect:MotEnd} we give the explicit answer for the graded motivic classes under consideration. Using the formalism of plethystic powers we then easily calculate the class of parabolic bundles with not necessarily nilpotent endomorphisms.

\begin{Remark}\label{rm:Weyl}
We will see in Theorem~\ref{th:MotMellitPunctures} that~\eqref{eq:Omega} is a symmetric function in $w_{x,\bullet}$ for each $x\in D$. This can be explained as follows. Note that the ``Weyl group'' $W:=\prod\limits_{x\in D}\Sigma_{\infty}$ acts on~$\Gamma_+$ and~$\Gamma'_+$ in the obvious way (here $\Sigma_\infty$ is the inductive limit of the permutation groups $\Sigma_l$). Using the commutativity of the motivic Hall algebra of the category of representations of the Jordan quiver (the quiver with one vertex and one loop), one can easily show that the motivic classes $\big[\Pair_\gamma^{{\rm nilp},-}(X,D,\lambda)\big]$ are $W$-invariant. Thus, we can re-write~\eqref{eq:Omega} as
\[
 w^{|\lambda|}\mathop{\sum\nolimits'}\limits_{\gamma=(r,r_{\bullet,\bullet},d)\in\Gamma'_+}
 \big[\Pair_\gamma^{{\rm nilp},-}(X,D,\lambda)\big]\prod_{x\in D} m_{r_{x,\bullet}}(w_{x,\bullet})z^d,
\]
where the summation is only over $r_{\bullet,\bullet}$ such that we have $r_{x,j}\ge r_{x,j+1}$ for all $x$ and $j$ (that is, $r_{x,\bullet}$ is a partition of~$|\lambda|$). Here, for a partition $\mu$, $m_\mu$ is the symmetric function equal to the sum of all monomials whose ordered list of exponents is~$\mu$.
\end{Remark}

\subsection[Factorization of graded motivic classes of stacks of nilpotent parabolic pairs]{Factorization of graded motivic classes of stacks\\ of nilpotent parabolic pairs}\label{sect:Factorization}

In this section we factorize~\eqref{eq:Omega} as the product of the global part (depending only on $X$ but not on $D$) and the local parts corresponding to points of $D$ (but independent of $X$). This is a~motivic version of~\cite[Theorem~5.6]{MellitPunctures}. We follow the same ideas, though some parts of Mellit's proof do not work in the motivic case and must be replaced by different arguments. On the other hand, we were able to simplify some parts of Mellit's proof, in particular, by working with stacks.

Note that $\Pair^{{\rm nilp},-}(X,\varnothing,\lambda)$ classifies pairs $(E,\Psi)$, where~$E$ is a nonpositive vector bundle on~$X$, $\Psi$ is an endomorphism of $E$ generically conjugate to $n_\lambda$ but there are no parabolic structures.

\begin{Lemma}\label{lm:invertible}
 The motivic class $\big[\Pair_\gamma^{{\rm nilp},-}\big(\P^1,\varnothing,\lambda\big)\big]\in w^{|\lambda|}\Mot(\kk)\big[\big[z^{-1}\big]\big]$ is invertible.
\end{Lemma}
\begin{proof}
 It is enough to show that the degree 0 part $\big[\Pair_{|\lambda|,0}^{{\rm nilp},-}\big(\P^1,\varnothing,\lambda\big)\big]$ is invertible in $\Mot(\kk)$. Note that a nonpositive vector bundle of degree 0 on $\P^1$ is necessarily trivial. Thus, this degree zero part classifies pairs $(E,\Psi)$, where $E$ is a trivial vector bundle and $\Psi$ is a constant endomorphism conjugate to $n_\lambda$. It is easy to see that this stack is isomorphic to the classifying stack of the centralizer $Z_\lambda$ of $n_\lambda$. By Lemma~\ref{lm:Zspecial} $Z_\lambda$ is special. Now it follows from Lemma~\ref{lm:SpecialMot} or~\cite[Lemma~2.2.3]{FedorovSoibelmans} that $\big[\Pair_{|\lambda|,0}^{{\rm nilp},-}\big(\P^1,\varnothing,\lambda\big)\big]=1/[Z_\lambda]$.
\end{proof}

We need some notation. Note that $\big[\Pair^{{\rm nilp},-}\big(\P^1,\infty,\lambda\big)\big]\in w^{|\lambda|}\Mot(\kk)\big[\big[w_{\infty,\bullet},z^{-1}\big]\big]$. For $x\in D$ let
\[
 \big[\Pair^{{\rm nilp},-}\big(\P^1,\infty,\lambda\big)\big]_x\in w^{|\lambda|}\Mot(\kk)\big[\big[w_{x,\bullet},z^{-1}\big]\big]
\]
denote the result of replacing $w_{\infty,\bullet}$ by $w_{x,\bullet}$ in this series.

\begin{Theorem}\label{th:Factorization}
We have in $\Mot(\kk)[[\Gamma'_+]]$.
\begin{equation*}
 \big[\Pair^{{\rm nilp},-}(X,D,\lambda)\big]=\big[\Pair^{{\rm nilp},-}(X,\varnothing,\lambda)\big]
 \prod_{x\in D}
 \frac{\big[\Pair^{{\rm nilp},-}\big(\P^1,\infty,\lambda\big)\big]_x}
 {\big[\Pair^{{\rm nilp},-}\big(\P^1,\varnothing,\lambda\big)\big]}.
\end{equation*}
\end{Theorem}

The proof of the theorem occupies the rest of Section~\ref{sect:Factorization}; it is based on the local study of stacks in the formal neighborhood of~$D$.

The main idea of the proof is very simple. Let us assume that $D=\{x\}$ is a single rational point of $X$. In Section~\ref{sect:LocStacks} we will define the stack $\Pair^{{\rm loc},{\rm fl}}$ classifying triples $(F,\Phi,F_\bullet)$, where $F$ is a rank $|\lambda|$ vector bundle over the formal completion of $X$ at $x$, $\Phi$ is a nilpotent endomorphism of $F$ generically conjugate to $n_\lambda$, $F_\bullet$ is a flag in $F_x$ preserved by $\Phi(x)$. We have an obvious restriction morphism $\Pair^{{\rm nilp},-}(X,x,\lambda)\to\Pair^{{\rm loc},{\rm fl}}$. We will see in Lemma~\ref{lm:fiber} that this restriction morphism has constant fiber. Thus, one is tempted to write the graded motivic class $\big[\Pair^{{\rm nilp},-}(X,x,\lambda)\big]$ as the product of the graded motivic class of this fiber and of $\Pair^{{\rm loc},{\rm fl}}$. This would quickly lead to the proof of the theorem. Unfortunately, $\Pair^{{\rm loc},{\rm fl}}$ is not an Artin stack as its points have inertia groups of infinite type, so its motivic class does not make sense. The major part of the proof consists of going around this problem.

Let us give the overview of the proof. In Section~\ref{sect:jets} we define and study the schemes of jets into $\gl_{|\lambda|}$. In Section~\ref{sect:LocStacks} we study the local stacks; they are essentially the quotients of the schemes of jets by the group of jets of ${\rm GL}_{|\lambda|}$. In Section~\ref{sect:StratThm} we re-write the theorem as a~statement about motivic classes of graded components. In Section~\ref{sect:Res} we study the fibers of the localization map; this is the main part of the proof. We complete the proof in Section~\ref{sect:ProofFact}.

\subsubsection{Jets}\label{sect:jets} We will denote the non-archimedean local field $\kk((t))$ by $\KK$ and its ring of integers $\kk[[t]]$ by $\OO$. The order of pole at $t=0$ gives rise to a valuation map $\val\colon \KK\to\Z\cup\{-\infty\}$, where $\val(0)=-\infty$. Clearly, $\val$ extends to $\gl_{r,\KK}$ as the maximum of valuations of all matrix elements. Let $J(X)$ denote the jet scheme of a scheme $X$ (this is a scheme of infinite type), and let $J_N(X)$ denote the scheme of order $N-1$ jets. In particular, $J_1(X)=X$.

For an algebraic group $G$ of finite type over $\kk$ we have the jet group $G_\OO:=J(G)$ and the jet group of finite type $J_N(G)$. The $N$-th congruence subgroup $G^{(N)}$ is the kernel of the projection $G_\OO\to J_N(G)$. We also have the ind-group of loops $G_\KK$ containing $G_\OO$. Let $\Delta:=\Spec\OO$ be the formal disc and $\mathring\Delta:=\Spec\KK$ be its generic point (the punctured formal disc). Also set $\Delta_N:=\Spec k[[t]]/t^N$. The groups ${\rm GL}_{r,\OO}$, ${\rm GL}_{r,\KK}$, and $J_N({\rm GL}_r)$ are the groups of automorphisms of the trivial vector bundles on $\Delta$, $\mathring\Delta$, and $\Delta_N$ respectively. For more details on the jet and loop groups we refer the reader to~\cite{SorgerLecturesBundles}.

Set $r=|\lambda|$. Consider the orbit stratification of $\gl_r$ under the adjoint action of ${\rm GL}_r$: $\gl_r=\bigsqcup\limits_{\mu\vdash r}\cO_\mu$, where $\cO_\mu$ is the adjoint orbit containing $n_\mu$. Let $\overline\cO_\mu$ denote the Zariski closure of $\cO_\mu$. Set
\[
 J(\lambda):=J\big(\overline\cO_\lambda\big)-\bigcup_{\cO_\mu\subset\overline\cO_\lambda-\cO_\lambda}J\big(\overline\cO_\mu\big).
\]
Note that $J(\lambda)$ parameterizes morphisms $\Delta\to\gl_r$ such that the image of the generic point of~$\Delta$ is in $\cO_\lambda$, that is, jets that are generically conjugate to~$n_\lambda$.

\begin{Definition}
 We say that a loop $g\in {\rm GL}_{r,\KK}(\kk)={\rm GL}_r(\KK)$ is \emph{kernel-strict} if $g^{-1}n_\lambda g\in\gl_{r,\OO}$ and $g^{-1}$ induces an isomorphism between the $\OO$-modules $\Ker n_\lambda\otimes_\kk\OO$ and $\Ker\big(g^{-1}n_\lambda g\big)$.
\end{Definition}
\begin{Remark} Note that our definition is a little different from~\cite[Definition~3.8]{MellitPunctures}. Mellit's definition of kernel-strictness depends also on a choice of a matrix $\theta$. In terminology of Mellit our $g$ is kernel-strict for $\theta=g^{-1}n_\lambda g$. Note also that the results of Mellit we are using here and below are formulated over finite fields but are valid over any field, proofs being the same.
\end{Remark}

Let $\Phi$ be a $\kk$-point of $J(\lambda)$, then there is a kernel-strict $g\in {\rm GL}_{r,\KK}(\kk)={\rm GL}_r(\KK)$ such that $\Phi=g^{-1}n_\lambda g$. Set $\deg\Phi:=\val(\det g)$. The existence of such $g$ and independence of the degree on the choice of $g$ is proved in~\cite[Lemma~3.7]{MellitPunctures}. It follows also from loc.~cit.~that the degree is nonnegative.

If $K\supset\kk$ is a field extension, we similarly define the degree of a $K$-point of $J(\lambda)$. The degree is compatible with field extensions. Thus we get a stratification of the set of points of $J(\lambda)$: $|J(\lambda)|=\bigsqcup\limits_{d\ge0}J_d(\lambda)$. Let $\pi_N\colon J(\lambda)\to J_N\big(\overline\cO_\lambda\big)$ be the truncation map. Set $J_{d,N}(\lambda):=\pi_N(J_d(\lambda))$.

\begin{Proposition}\label{pr:jets} For a nonnegative integer $d$ and a partition $\lambda$ there is a positive integer $N(d,\lambda)$ such that for $N>N(d,\lambda)$ we have
\begin{enumerate}\itemsep=0pt
\item[$(i)$] For all $\Phi$ in $J_d(\lambda)$ there is a kernel-strict $g$ with $\val(g)<N/2$, $\val\big(g^{-1}\big)<N/2$ such that $g\Phi g^{-1}=n_\lambda$.
\item[$(ii)$] $J_d(\lambda)=\pi_N^{-1}(J_{d,N}(\lambda))$.
\item[$(iii)$] Any two points in the same fiber of the projection $J_d(\lambda)\to J_{d,N}(\lambda)$ are conjugate by an element of ${\rm GL}_{r,\OO}$.
\end{enumerate}
\end{Proposition}
\begin{proof}
 By~\cite[Lemma~3.7]{MellitPunctures} there is $N_0\ge1$ (depending on $\lambda$ and $d$) such that for all $\Phi\in J_d(\lambda)$ there is a kernel-strict $g$ with $\val(g)<N_0$ such that $g\Phi g^{-1}=n_\lambda$. Then $\val(\det g)=d$. Set $N_1:=rN_0+d$. Then by Cramer's rule $\val\big(g^{-1}\big)<N_1$. We also have $\val(g)<N_1$. Take $N(d,\lambda):=4N_1$. With this choice of $N(d,\lambda)$ part~(i) of the proposition is clear.

 Note that (ii) is saying that the degree of an infinite jet depends only on its $N$-th truncation if $N>4N_1$. Thus to prove (ii) we need to show that if $\Phi\in J_d(\lambda)$ and $\Phi'\in J(\lambda)$ is such that $\Phi'\equiv\Phi\pmod{z^{4N_1}}$, then $\Phi'\in J_d(\lambda)$. Choose a kernel-strict $g$ with $\val(g)<N_1$, $\val\big(g^{-1}\big)<N_1$ such that $g\Phi g^{-1}=n_\lambda$. Then $g\Phi'g^{-1}\equiv n_\lambda\pmod{z^{2N_1}}$. We need a lemma.

 \begin{Lemma}\label{lm:conjugate}
 There is $g'\in {\rm GL}_r^{(2N_1)}$ such that $g'g\Phi'g^{-1}(g')^{-1}=n_\lambda$.
 \end{Lemma}
 \begin{proof}
 Set $\Phi'':=g\Phi'g^{-1}$ and $N_2:=2N_1$. Write $\Phi''\equiv n_\lambda+\Phi_{N_2}z^{N_2}\pmod{z^{N_2+1}}$. Since $\Phi''\in J\big(\overline\cO_\lambda\big)$, $\Phi_{N_2}$ belongs to the tangent space to $\overline\cO_\lambda$ at $n_\lambda$, which is naturally identified with $[n_\lambda,\gl_r]$. Thus there is $g_{_{N_2}}\in\gl_r$ such that $[n_\lambda,g_{_{N_2}}]=\Phi_{N_2}$. Then $\big(1+g_{_{N_2}}z^{N_2}\big)\Phi''\big(1+g_{_{N_2}}z^{N_2}\big)^{-1}\equiv n_\lambda\pmod{z^{N_2+1}}$.

 Repeating this process we find $g_{_{N_2+1}},g_{_{N_2+2}},\ldots\in\gl_r$ such that for $j>0$ we have
 \begin{gather*}
 \big(1+g_{_{N_2+j}}z^{N_2+j}\big)\cdots\big(1+g_{_{N_2}}z^{N_2}\big)\Phi''
 \big(1+g_{_{N_2}}z^{N_2}\big)^{-1}\cdots\big(1+g_{_{N_2+j}}z^{N_2+j}\big)^{-1}\\
 \qquad{} \equiv n_\lambda\pmod{z^{N_2+j+1}}.
 \end{gather*}
 It remains to take $g'=\prod\limits_{j=0}^\infty\big(1+g_{_{N_2+j}}z^{N_2+j}\big)$. Lemma~\ref{lm:conjugate} is proved.
 \end{proof}

 We return to the proof Proposition~\ref{pr:jets}. Note that $g^{-1}g'g\in {\rm GL}_{r,\OO}$. It is easy to see that the set of kernel-strict loops is invariant under the multiplication by points of ${\rm GL}_{r,\OO}$ on the right, so $g'g=g\big(g^{-1}g'g\big)$ is kernel-strict. Clearly, $\val(\det(g'g))=\val(\det g)=d$, so (ii) follows. Further, $g^{-1}g'g$ conjugates $\Phi'$ to $\Phi$, so (iii) follows as well. Proposition~\ref{pr:jets} is proved.
\end{proof}

For every $d\ge0$ we fix $N(d,\lambda)$ satisfying the conditions of the above proposition.

\begin{Definition}\label{def:stabilized}
 For $N>N(d,\lambda)$ we call any jet in $J_{d,N}(\lambda)$ \emph{stabilized} and call $d$ its \emph{degree}.
\end{Definition}
According to the proposition, the degree of a stabilized jet is well-defined and any two lifts of a stabilized jet to an infinite jet are conjugate by an element of ${\rm GL}_{r,\OO}$. Note that for every jet $\Phi\in J_d(\lambda)$ its truncation $\pi_N(\Phi)$ is stabilized for $N$ large enough.

\subsubsection{Local stacks}\label{sect:LocStacks} Consider the quotient stack $\Pair^{\rm loc}=\Pair^{\rm loc}(\lambda):=J(\lambda)/{\rm GL}_{r,\OO}$, where ${\rm GL}_{r,\OO}$ acts by conjugation. We skip $\lambda$ from the notation as it is fixed until the end of Section~\ref{sect:Factorization}. Since the degree function on $J(\lambda)$ is ${\rm GL}_{r,\OO}$-invariant, we get the degree function on the points of $\Pair^{\rm loc}$.

Note that $\Pair^{\rm loc}$ classifies pairs $(F,\Phi)$, where $F$ is a rank $r=|\lambda|$ vector bundle over~$\Delta$, $\Phi$~is a nilpotent endomorphism of $F$ generically conjugate to $n_\lambda$. This follows from the fact that every vector bundle on $\Delta$ is trivial. It also follows that every $K$-point of $\Pair^{\rm loc}$ is isomorphic to a point of the form $(\OO^r,\Phi)$, where $\Phi\in\gl_{r,\OO}$.

We emphasize that $\Pair^{\rm loc}$ is not an Artin stack (its isotropy groups are not of finite type).

Define $\Pair^{\rm loc}_N:=J_N\big(\overline\cO_\lambda\big)/J_N({\rm GL}_r)$; this is an Artin stack of finite type. The points of~$\Pair^{\rm loc}_N$ are the pairs $(F,\Phi)$ where $F$ is a vector bundle on~$\Delta_N$, $\Phi$~is an endomorphism of $F$ such that if we trivialize $F$, $\Phi$ becomes a jet with values in $\overline\cO_\lambda$. We say that $(F,\Phi)$ is \emph{stabilized} if $\Phi$ is stabilized in the sense of Definition~\ref{def:stabilized}. Note that this does not depend on the trivialization of~$F$. If $(F,\Phi)$ is stabilized, then we have a well-defined notion of the degree of~$(F,\Phi)$. Explicitly, we can lift $(F,\Phi)$ to a point $(\OO^r,\overline\Phi)$ of $\Pair^{\rm loc}$, and the degree of $(F,\Phi)$ is equal to the degree of $\overline\Phi\in\gl_{r,\OO}$.

We define the stack $\Pair^{{\rm loc},{\rm fl}}$ as the stack classifying triples $(F,\Phi,F_\bullet)$, where $(F,\Phi)$ is a point of $\Pair^{\rm loc}$, $F_\bullet$ is a flag in the fiber $F_0$ preserved by $\Phi(0)$. We define the stack $\Pair^{{\rm loc},{\rm fl}}_N$ as the stack classifying triples $(F,\Phi,F_\bullet)$, where $(F,\Phi)$ is a point of $\Pair^{\rm loc}_N$, $F_\bullet$ is a flag in $F_0$ preserved by $\Phi(0)$.

\subsubsection{Preparation for the proof of Theorem~\ref{th:Factorization}}\label{sect:StratThm} We will assume that $D=x$ is a single rational point of $X$. This will unburden the notation; the general case is proved similarly. Thus we want to prove that
\[
 \big[\Pair^{{\rm nilp},-}(X,x,\lambda)\big]\big[\Pair^{{\rm nilp},-}\big(\P^1,\varnothing,\lambda\big)\big]=\big[\Pair^{{\rm nilp},-}(X,\varnothing,\lambda)\big]\big[\Pair^{{\rm nilp},-}\big(\P^1,\infty,\lambda\big)\big]_x.
\]
Equating the graded components, we see that this reduces to the following proposition.
\begin{Proposition}\label{pr:factorization}
 Let $d$ be a nonpositive integer, $r_\bullet=(r_1,r_2,\dots)$ be a sequence of nonnegative integers such that $\sum_ir_i=r=|\lambda|$. Then we have in $\Mot(\kk)$:
 \begin{gather*}
 \sum_{d'+d''=d}\big[\Pair^{{\rm nilp},-}_{r,r_\bullet,d'}(X,x,\lambda)\times\Pair^{{\rm nilp},-}_{r,d''}\big(\P^1,\varnothing,\lambda\big)\big]\\
 \qquad{} =
 \sum_{d'+d''=d}\big[\Pair^{{\rm nilp},-}_{r,d'}(X,\varnothing,\lambda)\times\Pair^{{\rm nilp},-}_{r,r_\bullet,d''}\big(\P^1,\infty,\lambda\big)\big].
 \end{gather*}
\end{Proposition}
This proposition will be proved in Section~\ref{sect:ProofFact}. We emphasize that the sum is over all $d',d''\in\Z$ with $d'+d''=d$ but the terms are non-zero only if $d',d''\in[d,0]$. We note that the RHS is manifestly independent of $x$. Thus, the LHS is independent of $x$ as well.

\subsubsection[The restriction to the formal neighborhood of $x$]{The restriction to the formal neighborhood of $\boldsymbol{x}$}\label{sect:Res}
We keep the simplifying assumption that $D=\{x\}$ is a single point; we write $r_\bullet$ instead of~$r_{x,\bullet}$. Fix $\gamma=(r,r_{\bullet},d)\in\Gamma_+'$. For $x\in X$ let $\cO_{X,x}$ be the local ring of $x$ and $\hat\cO_{X,x}$ be its formal completion. Set $\Delta_x:=\Spec\hat\cO_{X,x}$. Choose a formal coordinate at $x$, use it to identify $\Delta_x$ with $\Delta$ and the $N$-th infinitesimal neighborhood $\Delta_{x,N}$ of $x$ with $\Delta_N$. Consider the restriction morphism
\begin{equation}\label{eq:res_x_fl}
 \loc_x^{\rm fl}\colon \ \Pair_{r,r_\bullet,d}^{{\rm nilp},-}(X,x,\lambda)\to\Pair^{{\rm loc},{\rm fl}}_{r_\bullet},\qquad (E,\Psi,E_{x,\bullet})
 \mapsto(E|_{\Delta_x},\Psi|_{\Delta_x},E_{x,\bullet}).
\end{equation}
Similarly we have a morphism
\begin{equation}\label{eq:res_x}
 \loc_x \colon \ \Pair_{r,d}^{{\rm nilp},-}(X,\varnothing,\lambda)\to\Pair^{\rm loc},\qquad (E,\Psi)
 \mapsto(E|_{\Delta_x},\Psi|_{\Delta_x}).
\end{equation}
Our nearest goal is to describe the fibers of these morphisms. For a nonpositive integer $e$, let $\Fib_e(X,x)$ denote the open substack of $\Pair^{{\rm nilp},-}_{r,e}(X,\varnothing,\lambda)$ consisting of $(E,\Psi)$ such that $\Psi$ is conjugate to $n_\lambda$ at $x$. Let $\wFib_e(X,x)$ denote the stack of triples $(E,\Psi,s)$, where $(E,\Psi)$ is a~point of $\Fib_e(X,x)$, $s$ is a trivialization of $E$ over $\Delta_x$ such that $\Psi=n_\lambda$ in this trivialization. Recall that in Section~\ref{sect:LocStacks} we defined the notion of degree for the points of $\Pair^{\rm loc}$.

\begin{Lemma}\label{lm:InfFiber}\quad
\begin{enumerate}\itemsep=0pt
\item[$(i)$] The fiber of $\loc_x^{\rm fl}$ over $(F,\Phi,F_\bullet)$ is isomorphic to $\wFib_{d+e}(X,x)$, where $e$ is the degree of~$(F,\Phi)$.
\item[$(ii)$]
Similarly, the fiber of $\loc_x$ over $(F,\Phi)$ is isomorphic to $\wFib_{d+e}(X,x)$, where $e$ is the degree of $(F,\Phi)$.
\end{enumerate}
\end{Lemma}
\begin{proof}
We prove (ii) first. Fix a trivialization of $F$ on the formal disc $\Delta$. Then $\Phi$ becomes an element of $\gl_{r,\KK}$ and we choose a kernel-strict $g$ such that $g\Phi g^{-1}=n_\lambda$. Then $\val(\det g)=e$.

Denote the fiber under consideration by $\wFib$. The fiber can be described as the stack of triples $(E,\Psi,s)$, where $E$ is a nonpositive vector bundle, $\Psi$ is an endomorphism, $s$ is the trivialization of~$E$ over $\Delta_x$ such that in this trivialization we have $\Psi|_{\Delta_x}=\Phi$. Note that such $\Psi$ is automatically conjugate to $n_\lambda$ at the generic point of $X$.

If $(E,\Psi,s)$ is a point of $\wFib$, then $E|_{X-x}$ is trivialized over the punctured disc $\mathring\Delta_x$, and we use the $g$ chosen above to glue $E|_{X-x}$ with the trivial bundle $\kk^r\times\Delta_x$ on $\mathring\Delta_x$ (we recall that $g$ can be viewed as an automorphism of the trivial vector bundle on $\mathring\Delta_x$). We obtain a new vector bundle $E'$ on $X$ with an isomorphism $E'|_{X-x}\simeq E|_{X-x}$ and a trivialization over $\Delta_x$. Thus $\Psi$ gives rise to an endomorphism $\Psi'$ of $E'|_{X-x}$. It is easy to derive from the definition of $g$ that in the given trivialization we have $\Psi'|_{\mathring\Delta_x}=n_\lambda$. Thus $\Psi'$ extends to $x$ and, moreover, in the given trivialization of $E'$ over $\Delta_x$ we have $\Psi'|_{\Delta_x}=n_\lambda$.

Note that $E'$ is nonpositive. Indeed, $\Ker\Psi$ is nonpositive as a subbundle of $E$. Since $g$ is kernel-strict, the isomorphism between $\Ker\Psi$ and $\Ker\Psi'$ extends from $X-x$ to $X$. Thus $\Ker\Psi'$ is also nonpositive. But by~\cite[Proposition~5.3]{MellitPunctures} this implies that $E'$ is nonpositive as well.

Next, we have an isomorphism between $\wedge^rE$ and $\wedge^rE'$ over $X-x$, and it has a zero of order $\val(\det g)$ at $x$. Thus $\deg E'=\deg\wedge^rE'=\deg\wedge^rE+\val(\det g)=\deg E+e=d+e$.

We have constructed a morphism $\wFib\to\wFib_{d+e}(X,x)$. Conversely, given a point $(E',\Psi',s')$ of $\wFib_{d+e}(X,x)$, we use $g$ and $s'$ to construct a new bundle $E$ with an isomorphism to $E'$ over $X-x$ and a trivialization over $\Delta_x$. Then $\Psi'|_{X-x}$ give rise to an endomorphism of $E|_{X-x}$ and we check that it extends to $x$ and, moreover, in the trivialization of $E$ over $\Delta_x$ we have $\Psi|_{\Delta_x}=\Phi$. Now it is easy to see that the two constructions are inverse to each other. This proves (ii).

Now (i) follows from the cartesian diagram
\[
 \begin{CD}
 \Pair_{r,r_\bullet,d}^{{\rm nilp},-}(X,x,\lambda)@>\loc_x^{\rm fl}>>\Pair^{{\rm loc},{\rm fl}}_{r_\bullet}\\
 @VVV @VVV\\
 \Pair_{r,d}^{{\rm nilp},-}(X,\varnothing,\lambda)@>\loc_x>>\Pair^{\rm loc},
 \end{CD}
\]
where the vertical arrows correspond to forgetting the flags.
\end{proof}

Consider now the compositions of~\eqref{eq:res_x_fl} and~\eqref{eq:res_x} with restrictions to the $N$-th infinitesimal neighborhood of $x$.
\begin{equation}\label{eq:res_x_flN}
 \loc_{x,N}^{\rm fl}\colon \ \Pair_{r,r_\bullet,d}^{{\rm nilp},-}(X,x,\lambda)\to\Pair^{{\rm loc},{\rm fl}}_{r_\bullet,N},\qquad (E,\Psi,E_{x,\bullet})
 \mapsto(E|_{\Delta_{x,N}},\Psi|_{\Delta_{x,N}},E_{x,\bullet}).
\end{equation}
Similarly we have a morphism
\begin{equation}\label{eq:res_xN}
 \loc_{x,N}\colon \ \Pair_{r,d}^{{\rm nilp},-}(X,\varnothing,\lambda)\to\Pair^{\rm loc}_N,\qquad (E,\Psi)
 \mapsto(E|_{\Delta_{x,N}},\Psi|_{\Delta_{x,N}}).
\end{equation}

Let $(F,\Phi,F_\bullet)$ be a point of $\Pair^{{\rm loc},{\rm fl}}_{r_\bullet,N}$ and assume that $(F,\Phi)$ is stabilized. By Definition~\ref{def:stabilized} we can find $\Phi'\in\gl_{r,\OO}$ such that $(\OO^r,\Phi')$ lifts $(F,\Phi)$ and $N>N(e,\lambda)$, where $e$ is the degree of $\Phi'$ and $N(e,\lambda)$ is the integer number from Proposition~\ref{pr:jets}, which was fixed just before Definition~\ref{def:stabilized}. Choose a kernel-strict $g\in\gl_{r,\KK}$ such that $g\Phi'g^{-1}=n_\lambda$. Denote by $Z^{(N)}_g$ the intersection $Z_\OO\cap\big(g^{-1}{\rm GL}^{(N)}g\big)\subset {\rm GL}_{r,\KK}$, where $Z=Z_\lambda$ is the centralizer of $n_\lambda$ in ${\rm GL}_{|\lambda|}$ (cf.~Lemma~\ref{lm:Zspecial}). Recall that by Proposition~\ref{pr:jets}(i) we may assume that the orders of the poles of $g$ and $g^{-1}$ are less than $N/2$ so we have $Z^{(N)}_g\subset Z^{(1)}$. This is a pro-unipotent group. Clearly, $Z_\OO$ (and thus~$Z^{(N)}_g$ as well) acts on $\wFib_{d+e}(X,x)$ by changing the trivialization of $E$ on $\Delta_x$.

\begin{Lemma}\label{lm:fiber} \quad
\begin{enumerate}\itemsep=0pt
\item[$(i)$] Let $(F,\Phi,F_\bullet)\in\Pair_{r_\bullet,N}^{{\rm loc},{\rm fl}}$ be such that $(F,\Phi)$ is stabilized, choose $g\in {\rm GL}_{r,\KK}$ as in the previous paragraph. Then the fiber of $\loc_{x,N}^{\rm fl}$ over $(F,\Phi,F_\bullet)$ is isomorphic to $\wFib_{d+e}(X,x)/Z^{(N)}_g$, where $e$ is the degree of $(F,\Phi)$.
\item[$(ii)$] Similarly, the fiber of $\loc_{x,N}$ over $(F,\Phi)$ is isomorphic to $\wFib_{d+e}(X,x)/Z^{(N)}_g$.
\end{enumerate}
\end{Lemma}
\begin{proof}
 Let us prove (ii), the proof of (i) is completely analogous. Denote the fiber under consideration by $\Fib$ and let $\wFib$ be the fiber of $\loc_x$ over $(\OO^r,\Phi')$. Then we have a restriction morphism $\wFib\to\Fib$. It follows from the stability of $(F,\Phi)$ and Proposition~\ref{pr:jets}(iii) that this morphism is surjective. On the other hand, it is easy to see that two points of $\wFib$ map to the same point of $\Fib$ if and only if they differ by the action of an element of $gZ_\OO g^{-1}\cap {\rm GL}^{(N)}$. On the other hand, according to Lemma~\ref{lm:InfFiber}, $\wFib\simeq\wFib_{d+e}(X,x)$. One checks that under this isomorphism, the action of $gZ_\OO g^{-1}\cap {\rm GL}^{(N)}$ on $\wFib$ corresponds to the action of $Z^{(N)}_g$ on $\wFib_{d+e}(X,x)$.
\end{proof}

We need to calculate the motivic class of this fiber. Set $Z_g:=Z_\OO/Z^{(N)}_g$; this is a group of finite type.
\begin{Lemma}\label{lm:MotFiber} We have in $\Mot(\kk)$
\[
 \big[\wFib_{d+e}(X,x)/Z^{(N)}_g\big]=[\Fib_{d+e}(X,x)]/[Z_g].
\]
\end{Lemma}
\begin{proof}For large $M$ we have $Z^{(M)}\subset Z_g^{(N)}$ and this subgroup is normal.

\emph{Claim.} The groups $Z^{(N)}_g/Z^{(M)}$, $Z_\OO/Z^{(M)}$, and $Z_\OO/Z^{(N)}_g$ are special.

\emph{Proof of the claim.} Recall that $Z_g^{(N)}\subset Z^{(1)}$. The group $Z^{(N)}_g/Z^{(M)}$ is special, since every unipotent group is special by Lemma~\ref{lm:SpecialRad}. Next, the quotient of $Z_\OO/Z^{(M)}$ by the unipotent subgroup $Z^{(1)}/Z^{(M)}$ is equal to $Z$ and the statement follows from Lemmas~\ref{lm:SpecialRad} and~\ref{lm:Zspecial}. A~similar argument shows that $Z_\OO/Z^{(N)}_g$ is special.\qed

We continue with the proof of Lemma~\ref{lm:MotFiber}. Next, $\overline\Fib_{d+e}(X,x)/Z^{(M)}$ is a $Z_\OO/Z^{(M)}$-principal bundle over $\Fib_{d+e}(X,x)$. Since $Z_\OO/Z^{(M)}$ is a special group, we get by Lemma~\ref{lm:SpecialMot}
 \[
 \big[\overline\Fib_{d+e}(X,x)/Z^{(M)}\big]=\big[Z_\OO/Z^{(M)}\big][\Fib_{d+e}(X,x)].
 \]
Similarly, $\wFib_{d+e}(X,x)/Z^{(M)}$ is a $Z^{(N)}_g/Z^{(M)}$-principal bundle over $\wFib_{d+e}(X,x)/Z^{(N)}_g$ and we get
 \[
 \big[\wFib_{d+e}(X,x)/Z^{(M)}\big]=\big[Z^{(N)}_g/Z^{(M)}\big]\big[\wFib_{d+e}(X,x)/Z^{(N)}_g\big].
 \]
Finally, $Z_\OO/Z^{(M)}$ is a $Z^{(N)}_g/Z^{(M)}$-principal bundle over $Z_g$ and we have
 \[
 \big[Z_\OO/Z^{(M)}\big]=\big[Z^{(N)}_g/Z^{(M)}\big][Z_g].
 \]
 The lemma follows from these three equations.
\end{proof}

\begin{Lemma}\label{lm:unique}
 Let $N$ be an integer larger than $N(j,\lambda)$ for all $j=0,\dots,-d$. Assume that the fiber of~\eqref{eq:res_xN} over $(F,\Phi)$ is non-empty. Then $(F,\Phi)$ is stabilized. A similar statement holds for the fibers of~\eqref{eq:res_x_flN}.
\end{Lemma}
\begin{proof}
 We prove the statement about the fibers of~\eqref{eq:res_xN}, the other statement being analogous. Let $(E,\Psi)$ be a point of the fiber, then the fiber of $\loc_x$ over $(E|_{\Delta_x},\Psi|_{\Delta_x})$ is non-empty. Then by Lemma~\ref{lm:InfFiber} this fiber is isomorphic to $\wFib_{d+e}(X,x)$, where $e$ is the degree of $(E|_{\Delta_x},\Psi|_{\Delta_x})$ (recall that $e\ge0$). Since $\wFib_{d+e}(X,x)$ classifies nonpositive vector bundles, we get $d+e\le0$. Thus $N>N(e,\lambda)$ and we see that $(F,\Phi)$ is stabilized.
\end{proof}

\subsubsection{Proof of Proposition~\ref{pr:factorization}}\label{sect:ProofFact}
Let us take an integer $N$ larger than $N(j,\lambda)$ for all $j=0,\dots,-d$. It is enough to show that the motivic functions
\[
 A:=\sum_{d'+d''=d}\big[\Pair^{{\rm nilp},-}_{r,r_\bullet,d'}(X,x,\lambda)\times\Pair^{{\rm nilp},-}_{r,d''}\big(\P^1,\varnothing,\lambda\big)
 \to\Pair^{{\rm loc},{\rm fl}}_{r_\bullet,N}\times\Pair^{\rm loc}_N\big]
\]
and
\[
 B:=\sum_{d'+d''=d}\big[\Pair^{{\rm nilp},-}_{r,r_\bullet,d'}\big(\P^1,\infty,\lambda\big)\times\Pair^{{\rm nilp},-}_{r,d''}(X,\varnothing,\lambda)
 \to\Pair^{{\rm loc},{\rm fl}}_{r_\bullet,N}\times\Pair^{\rm loc}_N\big]
\]
are equal. The morphisms are $\loc_{x,N}^{\rm fl}\times\loc_{\infty,N}$ and $\loc_{\infty,N}^{\rm fl}\times\loc_{x,N}$ respectively.

Let $K\supset\kk$ be an extension and let $\xi$ be a $K$-point of the stack $\Pair^{{\rm loc},{\rm fl}}_{r_\bullet,N}\times\Pair^{\rm loc}_{r,N}$ represented by $((F,\Phi,F_\bullet),(F',\Phi'))$. By Proposition~\ref{pr:MotFunEqual} it is enough to show that $\xi^*A=\xi^*B$. Using base change, we may assume that $K=\kk$. According to Lemma~\ref{lm:unique}, these motivic functions are zero unless $(F,\Phi)$ and $(F',\Phi')$ are stabilized.

Let us lift $(F,\Phi)$ and $(F',\Phi')$ to $(\OO^r,\overline\Phi)$ and $(\OO^r,\overline\Phi')$, where $\Phi',\overline\Phi'\in\gl_{r,\OO}$. Choose kernel-strict $g,g'\in {\rm GL}_{r,\KK}$ such that $g$, $g^{-1}$, $g'$, and $(g')^{-1}$ have poles of order less than $N/2$ and such that $g\overline\Phi g^{-1}=g'\overline\Phi'(g')^{-1}=n_\lambda$.

Then, according to Lemmas~\ref{lm:fiber} and~\ref{lm:MotFiber} (applied to $X$ and $\P^1$) we get
\begin{align*}
 \xi^*A& =\frac{\sum\limits_{d'+d''=d}[\Fib_{d'+e}(X,x)][\Fib_{d''+e'}(\P^1,\infty)]}{[Z_g][Z_{g'}]}\\
 & = \frac{\sum\limits_{d'+d''=d+e+e'}[\Fib_{d'}(X,x)][\Fib_{d''}(\P^1,\infty)]}{[Z_g][Z_{g'}]}.
\end{align*}
Similarly,
\begin{align*}
 \xi^*B& =\frac{\sum\limits_{d'+d''=d}[\Fib_{d'+e}(\P^1,\infty)][\Fib_{d''+e'}(X,x)]}{[Z_g][Z_{g'}]}\\
 & = \frac{\sum\limits_{d'+d''=d+e+e'}[\Fib_{d'}(\P^1,\infty)][\Fib_{d''}(X,x)]}{[Z_g][Z_{g'}]}.
\end{align*}
We see that $\xi^*A=\xi^*B$. This completes the proof of Proposition~\ref{pr:factorization} and thus the proof of Theorem~\ref{th:Factorization}.\qed

\begin{Remark} We emphasize that we only worked with motivic classes of Artin stacks of finite type. It seems plausible that one can define motivic classes of stacks like $\big[\,\wFib_d(X,x)\big]$ using some ideas of motivic integration. This would significantly simplify our argument. Unfortunately, we were not able to develop such a formalism.
\end{Remark}

\subsection[Case of $\P^1$ and two points]{Case of $\boldsymbol{\P^1}$ and two points}\label{sect:P1}
Consider the case when $X=\P^1$, $D=\{0,\infty\}$.
\begin{Proposition}\label{pr:P1}
We have in $\Mot(\kk)[[\Gamma'_+]]\subset\Mot(\kk)\big[\big[w,w_{0,\bullet},w_{\infty,\bullet},z^{-1}\big]\big]$
\begin{equation}\label{eq:P10infty}
 \big[\Pair^{{\rm nilp},-}\big(\P^1,\{0,\infty\}\big)\big]=\Exp\left(w\frac{\sum_{j=1}^\infty\sum_{j'=1}^\infty w_{0,j}w_{\infty,j'}}{(\bL-1)(1-z^{-1})}\right),
\end{equation}
where $\Exp$ is the plethystic exponent defined in Section~{\rm \ref{sect:Plethystic}}.
\end{Proposition}
\begin{proof}
We note that the proof in~\cite[Section~5.4]{MellitPunctures} goes through in the motivic case as well. The only difference is that Mellit uses the Hall algebra of the Jordan quiver (that is, the Hall algebra of the category of vector spaces with nilpotent endomorphisms); this Hall algebra has to be replaced with the similar motivic Hall algebra in our case.

Let us give more details. In~\cite[Section~5]{FedorovSoibelmans}, to any smooth projective geometrically connected curve over $\kk$ we associated the Hall algebra of the category of coherent sheaves on this curve, denoted by~$\cH$. Let us take the curve to be $\P^1_\kk$ and let $\cH_0$ be the subalgebra of torsion sheaves supported at $0\in\P^1_\kk$. Since such sheaves are identified with finite dimensional representations of the Jordan quiver, we can view $\cH_0$ as the Hall algebra of the category of such representations. Of course, we could have taken any other curve and any rational point.

Following Mellit, we re-write the LHS of~\eqref{eq:P10infty} in terms of products of certain elements of this algebra. This part of Mellit's proof is geometric, so it is easily carried to the motivic case. The rest of the proof is a calculation in this Hall algebra; the necessary identities in the Hall algebra are easily derived from results of \cite[Section~5]{FedorovSoibelmans}.
\end{proof}
\begin{Remark}
Note that the RHS of~\eqref{eq:P10infty} can be written without plethystic exponents as follows (cf.~\cite[Lemma~5.7.3]{FedorovSoibelmans})
\begin{equation*}
 \prod_{d=0}^{-\infty}\prod_{j=1}^\infty\prod_{j'=1}^\infty\left(
 1+\sum_{i\ge1}\frac{\bL^{i(i-1)}}{\big(\bL^i-1\big)\cdots\big(\bL^i-\bL^{i-1}\big)}z^{id}w^iw_{0,j}^iw_{\infty,j'}^i
 \right).
\end{equation*}
\end{Remark}

\subsection{Motivic modified Macdonald polynomials} For a commutative unital ring $A$, let $\Sym_A[w_\bullet]$ be the ring of symmetric functions with coefficients in $A$ in variables $w_\bullet$. In this section we define axiomatically the images of modified Macdonald polynomials in $\Sym_{\Mot(\kk)[[z]]}[w_\bullet]$.

Consider the modified Macdonald polynomials $\tilde H_\lambda(w_\bullet;q,z)\in\Sym_{\Z[q,z]}[w_\bullet]$. For a definition see, for example,~\cite[Definition~2.5]{MellitPunctures}. It is not clear from this definition that the coefficients of $\tilde H_\lambda(w_\bullet;q,z)$ are integers, but this is well-known (see, e.g.,~\cite{HaglundEtAlOnMacdonaldPoly} and references therein). Note that~$\tilde H_\lambda$ is a symmetric function so, formally speaking, it is not a polynomial.

We denote by $\tilde H^{\rm mot}_\lambda(w_\bullet;z)\in\Sym_{\Mot(\kk)[z]}[w_\bullet]$ the image of the corresponding modified Macdonald polynomial under the homomorphism $\Sym_{\Z[q,z]}[w_\bullet]\to\Sym_{\Mot(\kk)[z]}[w_\bullet]$ sending $q$ to $\bL$; we call these images \emph{motivic modified Macdonald polynomials}.

Define the motivic Hall--Littlewood polynomials as the specialization
\[ H_\lambda^{\rm mot}(w_\bullet):=\tilde H_\lambda^{\rm mot}(w_\bullet;0).\] Thus $H_\lambda^{\rm mot}$ is the image of the usual Hall--Littlewood polynomial under the homomorphism $\Z[q,w_\bullet]\to\Mot(\kk)[w_\bullet]$ sending $q$ to $\bL$. The motivic Hall--Littlewood polynomials can be interpreted as follows: let $\Fl_\lambda$ stand for the scheme of all flags in $\kk^{|\lambda|}$ preserved by $n_\lambda$. Then $\Fl_\lambda$ is graded by the type of the flag. It is not difficult to check that we have
\[
 H^{\rm mot}_\lambda(w_\bullet)=[\Fl_\lambda]\in\Mot(\kk)[w_\bullet]
\]
(cf.~\cite[Theorem~2.12, Corollary~2.13]{MellitPunctures}).

It follows from~\cite[Chapter~3, equation~(2.7)]{macdonald1998symmetric} that $H_\lambda^{\rm mot}$ form a basis of the $\Mot(\kk)$-module $\Sym_{\Mot(\kk)}[w_\bullet]$. Thus $\tilde H_\lambda^{\rm mot}$ also form a basis of $\Sym_{\Mot(\kk)[[z]]}[w_\bullet]$.

\begin{Proposition}\label{pr:axiomMacdonald}\quad
\begin{enumerate}\itemsep=0pt
\item[$(a)$] The motivic modified Macdonald polynomials $\tilde H_\lambda^{\rm mot}$ satisfy the following properties
\begin{enumerate}\itemsep=0pt
\item[$(i)$]
\begin{equation}\label{eq:scalar}
 \Exp\left(\frac{\sum\limits_{j=1}^\infty\sum\limits_{j'=1}^\infty w_{0,j}w_{\infty,j'}}{(\bL-1)(1-z)}\right)=
 \sum_\lambda a_\lambda(z)\tilde H_\lambda^{\rm mot}(w_{0,\bullet};z)\tilde H_\lambda^{\rm mot}(w_{\infty,\bullet};z),
\end{equation}
where $a_\lambda$ are invertible elements of $\Mot(\kk)[[z]]$.
\item[$(ii)$] $\tilde H_\lambda^{\rm mot}=\sum\limits_{\mu\colon \mu'\prec\lambda'}b_{\lambda\mu}H_\mu^{\rm mot}$, where $b_{\lambda\mu}$ are some elements of $\Mot(\kk)[[z]]$ such that $b_{\lambda\lambda}$ are invertible. $($Here $\mu'$ and $\lambda'$ stand for the conjugate partitions, $\prec$ is the usual order on partitions.$)$
\item[$(iii)$] $\tilde H^{\rm mot}_\lambda(1_\bullet;z)=1$, where $1_\bullet$ stands for the sequence $(1,0,\dots,0,\dots)$.
\item[$(iv)$] $\tilde H^{\rm mot}_\lambda$ is homogeneous in $w_\bullet$ of degree $|\lambda|$.
\end{enumerate}
\item[$(b)$] The motivic modified Macdonald polynomials are uniquely determined by properties $(i)$--$(iv)$.
\item[$(c)$] Additionally we have
\begin{equation}\label{eq:scalarLengh}
 a_\lambda(z)=\frac1{\prod\limits_{h\in\Hook(\lambda)}\big(\bL^{a(h)}-z^{l(h)+1}\big)\big(\bL^{a(h)+1}-z^{l(h)}\big)},
\end{equation}
where $\Hook(\lambda)$ stands for the set of hooks of $\lambda$, $a(h)$ and $l(h)$ stand for the armlength and the leglength of the hook $h$ respectively.
\end{enumerate}
\end{Proposition}
\begin{proof}
 (a) It is enough to check the corresponding properties for the usual modified Macdonald polynomials $\tilde H_\lambda\in\Sym_{\Z[q,z]}[w_\bullet]$ and Hall--Littlewood polynomials $H_\lambda\in\Sym_{\Z[q]}[w_\bullet]$. To prove property~(ii) we note first that according to~\cite[Definition~2.5]{MellitPunctures} we have
 \[
 \tilde H_\lambda[(q-1)w_\bullet]=\sum_{\mu\colon \mu'\prec\lambda'}c_{\lambda\mu}(q,z)m_{\mu'}(w_\bullet),
 \]
 where $[(q-1)w_\bullet]$ stands for the plethystic action as in~\cite[Section~2.1]{MellitPunctures}. Recalling that the Hall--Littlewood polynomials are $z=0$ specializations of the modified Macdonald polynomials, we get
 \[
 H_\lambda[(q-1)w_\bullet]=\sum_{\mu\colon \mu'\prec\lambda'}c_{\lambda\mu}(q,0)m_{\mu'}(w_\bullet).
 \]
 Now it is easy to see that an analogue of property~(ii) holds in $\Sym_{\Q[q,z]}(w_\bullet)$:
 we can write $\tilde H_\lambda=\sum\limits_{\mu\colon \mu'\prec\lambda'}b'_{\lambda\mu}H_\mu$, where $b'_{\lambda\mu}$ are some elements of $\Q[q,z]$.
 Next, $H_\lambda$ form a basis in $\Sym_{\Z[q]}(w_\bullet)$ (by~\cite[Chapter~3, equation~(2.7)]{macdonald1998symmetric}), so $\tilde H_\lambda$ form a basis in $\Sym_{\Z[q][[w]]}(w_\bullet)$. It follows that $b'_{\lambda\mu}\in\Z[w][[z]]$ and $b'_{\lambda\lambda}$ are invertible in this ring. Now property~(ii) follows.

 It is sufficient to prove properties (iii), and (iv) in $\Sym_{\Q(q,z)}[w_\bullet]$. Property~(iii) is clear from~\cite[Definition~2.5]{MellitPunctures}. Property~(iv) follows, for example, from the definition of $\tilde H$ given in~\cite{GarsiaHaiman1996remarkable}.

 We first prove an analogue of property~(i) in $\Sym_{\Q(q,z)}[w_\bullet]$. Recall from loc.~cit.~that $\Sym_{\Q(q,z)}[w_\bullet]$ carries the $q,z$-scalar product $(\cdot,\cdot)_{q,z}$. for which
 \[
 \Exp\left(\frac{\sum\limits_{j=1}^\infty\sum\limits_{j'=1}^\infty w_{0,j}w_{\infty,j'}}{(q-1)(1-z)}\right)
 \]
 is the reproducing kernel. This means that if $f_\lambda(w_\bullet;q,z)$ is any graded $\Q(q,z)$-basis in \linebreak $\Sym_{\Q(q,z)}[w_\bullet]$ indexed by partitions and $f^\vee_\lambda(w_\bullet;q,z)$ is the dual basis with respect to $(\cdot,\cdot)_{q,z}$, then
 \begin{equation}\label{eq:RepKer}
 \Exp\left(\frac{\sum\limits_{j=1}^\infty\sum\limits_{j'=1}^\infty w_{0,j}w_{\infty,j'}}{(q-1)(1-z)}\right)=\sum_\lambda
 f_\lambda(w_{0,\bullet};q,z)f^\vee_\lambda(w_{\infty,\bullet};q,z).
 \end{equation}
 Next, by property~(iv) the basis $H_\lambda(w_\bullet;q,z)$ is a graded basis. Thus, by~\cite[Proposition~2.7]{MellitPunctures} the dual of $H_\lambda(w_\bullet;q,z)$ is equal to $a_\lambda(q,z)H_\lambda(w_\bullet;q,z)$ for some $a_\lambda(q,z)\in\Q(q,z)$. Further, in~\cite[Section~2.4]{MellitPunctures} it is shown that
 \begin{equation}\label{eq:alambda}
 a_\lambda(q,z)=\frac1{\prod\limits_{h\in\Hook(\lambda)}\big(q^{a(h)}-z^{l(h)+1}\big)\big(q^{a(h)+1}-z^{l(h)}\big)}.
 \end{equation}
 It is clear from this formula that $a_\lambda(q,z)\in\Z(q)[[z]]$. Now property~(i) follows from~\eqref{eq:RepKer}.

 The condition in part~(c) follows from~\eqref{eq:alambda}.

 Now we prove part~(b). Since $\tilde H_\lambda^{\rm mot}$ form a basis of $\Sym_{\Mot(\kk)[[z]]}[w_\bullet]$, there is a unique $\Mot(\kk)[[z]]$-linear scalar product on $\Sym_{\Mot(\kk)[[z]]}[w_\bullet]$ such that $\langle\tilde H_\lambda^{\rm mot},\tilde H_\mu^{\rm mot}\rangle=\delta_{\lambda\mu}\frac1{a_\lambda}$. (This is the scalar product such that the LHS of~\eqref{eq:scalar} is the reproducing kernel for this product).

 Let $H'_\lambda=H'_\lambda(w_\bullet;z)$ be symmetric functions satisfying conditions of part~(a), we need to show that $H'_\lambda=\tilde H_\lambda^{\rm mot}$. Applying condition (ii), we see that we can write
 \begin{equation}\label{eq:triangle}
 H'_\lambda=\sum_{\mu\colon \mu'\prec\lambda'}c_{\lambda\mu}\tilde H_\mu^{\rm mot},
 \end{equation}
 where $c_{\lambda\lambda}$ is invertible. Condition~(i) shows that we have
 \[
 \sum_\lambda a_\lambda(z)\tilde H_\lambda^{\rm mot}(w_{0,\bullet};z)\tilde H_\lambda^{\rm mot}(w_{\infty,\bullet};z)=
 \sum_\lambda a'_\lambda(z)H'_\lambda(w_{0,\bullet};z)H'_\lambda(w_{\infty,\bullet};z),
 \]
 with invertible $a'_\lambda(z)$. Recalling that $H'_\lambda$ form a graded basis in $\Sym_{\Mot(\kk)[[z]]}[w_\bullet]$, we see that $\langle H'_\lambda,H'_\mu\rangle=\delta_{\lambda\mu}\frac1{a'_\lambda}$. Indeed, $H_\lambda^{\rm mot}$ and $a_\lambda H_\lambda^{\rm mot}$ are dual basis for the scalar product, so the above equality shows that $H'_\lambda$ and $a'_\lambda H'_\lambda$ are dual basis as well.

 We prove that $H'_\lambda=\tilde H_\lambda^{\rm mot}$ by induction on the conjugate partition $\lambda'$ with respect to $\prec$. Thus we assume that $H'_\mu=\tilde H_\mu^{\rm mot}$ whenever $\mu'\prec\lambda'$. Taking the scalar product of~\eqref{eq:triangle} with $\tilde H_\mu^{\rm mot}=H'_\mu$, we see that $c_{\lambda\mu}\langle\tilde H_\mu^{\rm mot},\tilde H_\mu^{\rm mot}\rangle=0$ whenever $\mu\ne\lambda$. We see that $c_{\lambda\mu}=0$ so that $H'_\lambda=c_{\lambda\lambda}\tilde H_\lambda^{\rm mot}$. Now condition (iii) implies that $c_{\lambda\lambda}=1$.
\end{proof}

\subsection{Explicit formulas for the graded motivic classes of nilpotent pairs}\label{sect:MotEnd}
Now we are ready to give the precise formula for $\big[\Pair^{{\rm nilp},-}(X,D,\lambda)\big]$. Recall that for a partition~$\lambda$ we defined $J_\lambda^{\rm mot}(z),H_\lambda^{\rm mot}(z)\in\cMot(\kk)[[z]]$ in~\cite[Section~1.3.2]{FedorovSoibelmans}. In this paper, we will denote them by $J_{\lambda,X}^{\rm mot}(z)$ and $H_{\lambda,X}^{\rm mot}(z)$ respectively to emphasize that they depend on the curve $X$ and to ensure that they are not confused with motivic modified Macdonald polynomials $\tilde H_\lambda^{\rm mot}(w_\bullet;z)$ and with motivic Hall--Littlewood polynomials $H_\lambda^{\rm mot}(w_\bullet)$. Denote by $g$ the genus of~$X$.

\begin{Theorem}\label{th:MotMellitPunctures} We have in $\cMot(\kk)[[\Gamma'_+]]$
\[
 \big[\Pair^{{\rm nilp},-}(X,D,\lambda)\big]=w^{|\lambda|}\bL^{(g-1)\langle\lambda,\lambda\rangle}J_{\lambda,X}^{\rm mot}\big(z^{-1}\big)
 H_{\lambda,X}^{\rm mot}\big(z^{-1}\big)\prod_{x\in D}\tilde H_\lambda^{\rm mot}\big(w_{x,\bullet};z^{-1}\big).
\]
\end{Theorem}
\begin{proof}
According to~\cite[Theorem~1.4.1]{FedorovSoibelmans}, we have
\[
 \sum_\lambda\big[\Pair^{{\rm nilp},+}(X,\varnothing,\lambda)\big]=
 w^{|\lambda|}\sum_\lambda\bL^{(g-1)\langle\lambda,\lambda\rangle}J^{\rm mot}_{\lambda,X}(z)H^{\rm mot}_{\lambda,X}(z),
\]
where the superscript ``$+$'' stands for HN-nonnegative vector bundles (see Section~\ref{sect:Gamma} and~\cite[Section~3.2]{FedorovSoibelmans} for the definition). Inspecting the proof, we see that for each $\lambda$ the summands are equal:
\begin{equation*}
 \big[\Pair^{{\rm nilp},+}(X,\varnothing,\lambda)\big]=w^{|\lambda|}\bL^{(g-1)\langle\lambda,\lambda\rangle}J_{\lambda,X}^{\rm mot}(z)H_{\lambda,X}^{\rm mot}(z).
\end{equation*}
By Lemma~\ref{lm:PosNeg} we get an isomorphism of stacks
\[
 \Pair^{{\rm nilp},+}_{r,d}(X,\varnothing,\lambda)\simeq\Pair^{{\rm nilp},-}_{r,-d}(X,\varnothing,\lambda).
\]
Thus
\[
 \big[\Pair^{{\rm nilp},-}(X,\varnothing,\lambda)\big]=w^{|\lambda|}\bL^{(g-1)\langle\lambda,\lambda\rangle}J_{\lambda,X}^{\rm mot}\big(z^{-1}\big)H_{\lambda,X}^{\rm mot}\big(z^{-1}\big).
\]

To be able to apply Theorem~\ref{th:Factorization}, we need the following lemma.
\begin{Lemma}\label{lm:MacDonald} We have in $\Mot(\kk)[[w_{\infty,\bullet},z^{-1}]]$
\[
 \frac{\big[\Pair^{{\rm nilp},-}\big(\P^1,\infty,\lambda\big)\big]}{\big[\Pair^{{\rm nilp},-}\big(\P^1,\varnothing,\lambda\big)\big]}=
 \tilde H_\lambda^{\rm mot}\big(w_{\infty,\bullet};z^{-1}\big).
\]
\end{Lemma}
\begin{proof}
Our proof is similar to that of~\cite[Theorem~5.5]{MellitPunctures}. Let $H'_\lambda\in\Mot(\kk)[[w_\bullet,z]]$ be the series such that
\[
\frac{\big[\Pair^{{\rm nilp},-}\big(\P^1,\infty,\lambda\big)\big]}{\big[\Pair^{{\rm nilp},-}\big(\P^1,\varnothing,\lambda\big)\big]}=H'_\lambda\big(w_{\infty,\bullet};z^{-1}\big).
\] Denote by $\cC_{\lambda\mu}$ the stack classifying pairs $(E,\Psi)$, where $E$ is a nonpositive vector bundle of rank~$|\lambda|$ on $\P^1$, $\Psi$ is an endomorphism of $E$ generically conjugate to $n_\lambda$ and conjugate to~$n_\mu$ at $x=\infty$. Then $\cC_{\lambda\mu}$ is graded by the degree of $E$, and we have $[\cC_{\lambda\mu}]\in\Mot(\kk)\big[\big[z^{-1}\big]\big]$. Clearly, we have
\begin{equation}\label{eq:H'}
 H'_\lambda\big(w_\bullet;z^{-1}\big)=\sum_\mu\frac{[\cC_{\lambda\mu}]}{\big[\Pair^{{\rm nilp},-}\big(\P^1,\varnothing,\lambda\big)\big]}H_\mu^{\rm mot}(w_\bullet),
\end{equation}
where $H_\mu^{\rm mot}$ are the motivic Hall--Littlewood polynomials. Note that $C_{\lambda\mu}=\varnothing$ unless $\mu'\prec\lambda'$ because
for all $i$ the dimension of the fiber of $\Ker\Psi^i$ is semicontinuous on $\P^1$. Now it is easy to see that $H'_\lambda$ are symmetric functions with coefficients in $\Mot(\kk)[[z]]$. We will use Proposition~\ref{pr:axiomMacdonald}(b) to show that for all $\lambda$ we have $H'_\lambda=\tilde H_\lambda^{\rm mot}$.

To show that $H'_\lambda$ satisfy property~(ii) of Proposition~\ref{pr:axiomMacdonald}(a) it remains to show that $[\cC_{\lambda\lambda}]$ is invertible. This is completely similar to the proof of Lemma~\ref{lm:invertible}.

To show that $H'_\lambda$ satisfy condition~(i) of Proposition~\ref{pr:axiomMacdonald}(a), we note that combining Theorem~\ref{th:Factorization} and Proposition~\ref{pr:P1}, we get
\[
 \Exp\left(\frac{\sum\limits_{j=1}^\infty w_{0,j}w_{\infty,j}}{(\bL-1)(1-z)}\right)=
 \sum_\lambda a'_\lambda(z)H'_\lambda(w_{0,\bullet};z)H'_\lambda(w_{\infty,\bullet};z),
\]
where $a'_\lambda(z^{-1})=\big[\Pair^{{\rm nilp},-}\big(\P^1,\varnothing,\lambda\big)\big]$ and the statement follows.

Condition~(iii) of Proposition~\ref{pr:axiomMacdonald}(a) is obvious. Condition~(iv) follows from~\eqref{eq:H'}. We note also for the future use that it is clear from the argument that we have
\begin{equation}\label{eq:P1emptyset}
 \big[\Pair^{{\rm nilp},-}\big(\P^1,\varnothing,\lambda\big)\big]=a'_\lambda\big(z^{-1}\big)=a_\lambda\big(z^{-1}\big),
\end{equation}
where $a_\lambda(z)$ is given by~\eqref{eq:scalarLengh}. The proof of Lemma~\ref{lm:MacDonald} is complete.
\end{proof}

Now Theorem~\ref{th:Factorization} completes the proof of Theorem~\ref{th:MotMellitPunctures}.
\end{proof}

\begin{Corollary}\label{cor:Pairs} We have in $\cMot(\kk)[[\Gamma'_+]]$
\begin{gather*}
 \big[\Pair^-(X,D)\big]=
 \Pow\left(
 \sum_\lambda w^{|\lambda|}\bL^{(g-1)\langle\lambda,\lambda\rangle}J_{\lambda,X}^{\rm mot}\big(z^{-1}\big)
 H_{\lambda,X}^{\rm mot}\big(z^{-1}\big)\prod_{x\in D}\tilde H_\lambda^{\rm mot}\big(w_{x,\bullet};z^{-1}\big)
 ,\bL
 \right).
\end{gather*}
\end{Corollary}
\begin{proof}
 The argument is similar to the proof of~\cite[Proposition~3.8.1]{FedorovSoibelmans}. In more details, let $(E,\Psi,E_{\bullet,\bullet})$ be a $K$-point of $\Pair^-(X,D)$. According to~\cite[Lemma~3.8.3]{FedorovSoibelmans}, we can uniquely decompose
 \[
 (E,\Psi)\xrightarrow{\simeq}\bigoplus_i R_{\kk(x_i)/K}(E_i,x_i\Id+\Psi_i),
 \]
 where $x_i$ are distinct closed points of $\A_K^1$ (the eigenvalues of $\Psi$), $(E_i,\Psi_i)$ are $\kk(x_i)$-points of the stack $\Pair^{{\rm nilp},-}(X,\varnothing)$, $\kk(x_i)\supset K$ is the residue field of $x_i$, and $R_{\kk(x_i)/K}$ is the pushforward functor. It follows easily from the proof of~\cite[Lemma~3.8.3]{FedorovSoibelmans}, that we can write uniquely
 \[
 (E,\Psi,E_{\bullet,\bullet})\xrightarrow{\simeq}\bigoplus_i R_{\kk(x_i)/K}(E_i,x_i\Id+\Psi_i,E_{i,\bullet,\bullet}),
 \]
 where $(E_i,\Psi_i,E_{i,\bullet,\bullet})$ are $\kk(x_i)$-points of $\Pair^{{\rm nilp},-}(X,D)$. It remains to use a version of~\cite[Lemma~3.8.2]{FedorovSoibelmans}.
 \end{proof}

\subsection[Case of $\P^1$]{Case of $\boldsymbol{\P^1}$}\label{sect:P1ManyPts}

In the case of $X=\P^1$ we can give a more explicit answer. Moreover, we get an answer valid in $\Mot(\kk)[[\Gamma'_+]]$ rather than in its completion $\cMot(\kk)[[\Gamma'_+]]$, which is desirable, since we do not know whether the natural homomorphism $\Mot(\kk)\to\cMot(\kk)$ is injective. We argue as in~\cite[Corollary~5.9]{MellitPunctures}. Combining Theorem~\ref{th:Factorization},~\eqref{eq:P1emptyset}, and~\eqref{eq:scalarLengh}, we get the following formula valid in $\Mot(\kk)[[\Gamma'_+]]$.
\[
 \big[\Pair^{{\rm nilp},-}\big(\P^1,D,\lambda\big)\big]=w^{|\lambda|}\frac{\prod\limits_{x\in D}\tilde H_\lambda^{\rm mot}\big(w_{x,\bullet};z^{-1}\big)} {\prod\limits_{h\in\Hook(\lambda)}\big(\bL^{a(h)}-z^{-l(h)-1}\big)\big(\bL^{a(h)+1}-z^{-l(h)}\big)},
\]
where $\Hook(\lambda)$ stands for the set of hooks of $\lambda$, $a(h)$ and $l(h)$ stand for the armlength and the leglength of the hook $h$ respectively. Arguing as in Corollary~\ref{cor:Pairs}, we get in $\Mot(\kk)[[\Gamma'_+]]$
\begin{equation}\label{eq:P1}
 \big[\Pair^-\big(\P^1,D\big)\big]=
 \Pow\left(
 \sum_\lambda\frac{w^{|\lambda|}\prod\limits_{x\in D}\tilde H_\lambda^{\rm mot}(w_{x,\bullet};z)}
 {\prod\limits_{h\in\Hook(\lambda)} \big(\bL^{a(h)}-z^{-l(h)-1}\big)\big(\bL^{a(h)+1}-z^{-l(h)}\big)} ,\bL
 \right).
\end{equation}

\section{Parabolic Higgs bundles with fixed eigenvalues}\label{sect:HiggsnEigenval}
\subsection{Stacks of Higgs bundles} Let $X$ and $D$ be as above. From now on we denote by $g$ the genus of $X$. Our goal in this section is to calculate the motivic Donaldson--Thomas series of the category of parabolic Higgs bundles. More precisely, we calculate the motivic classes of the moduli stacks of Higgs bundles with fixed eigenvalues and with nonpositive underlying vector bundles. Our argument is similar to~\cite[Sections~3.4--3.5]{FedorovSoibelmans}. The main result is Corollary~\ref{cor:MAIN}. We denote by $\Omega_X$ the canonical line bundle on $X$.

\begin{Definition}
A \emph{parabolic Higgs bundle} of type $(X,D)$ is a triple $(E,E_{\bullet,\bullet},\Phi)$, where $(E,E_{\bullet,\bullet})$ is a point of $\Bun^{\rm par}(X,D)$, $\Phi\colon E\to E\otimes\Omega_X(D)$ is an $\cO_X$-linear morphism (called a \emph{Higgs field on $(E,E_{\bullet,\bullet})$}) such that for all $x\in D$ and $j\ge0$ we have $\Phi_x(E_{x,j})\subset E_{x,j}\otimes\Omega_X(D)_x$.
\end{Definition}

We denote the category (and the Artin stack) of parabolic Higgs bundles by $\Higgs=\Higgs(X,D)$. We define the $\Gamma'_+$-graded stack $\Higgs^-=\Higgs^-(X,D)$ following the general formalism of Section~\ref{sect:CatsOverPar}, that is, $\Higgs^-$ is the open substack of $\Higgs$ corresponding to Higgs bundles with nonpositive underlying vector bundle. Clearly, this stack is of finite type in the graded sense (that is, the graded components are of finite type). In this section $X$ and $D$ are fixed, so we skip them from the notation.

\subsection{Existence of Higgs fields with prescribed residues}\label{Sect:ExistResidue}
To determine a criterion for the existence of a Higgs bundle with prescribed residues, we use an approach similar to~\cite{AtiyahConnections,MihaiMonodromie,MihaiConnexions}. Let $E\to X$ be a vector bundle and let $\Phi\colon E\to E\otimes\Omega_X(D)$ be a morphism. In this case, for all $x\in D$ we have a residue $\Res_x\Phi\in\End(E_x)$.
\begin{Proposition}\label{pr:exist}
\label{existenceHiggs}
Let $E$ be a vector bundle on $X$ and let for $x\in D$, $\rho_x\in\End(E_x)$ be an endomorphism of the fiber of $E$ at $x$. There exists a Higgs field $\Phi\colon E\to E\otimes\Omega_X(D)$ with $\Res_x\Phi=\rho_x$ for all $x\in D$ if and only if
\[
 \sum_{x\in D}\tr(\rho_x\phi_x)=0
\]
for all $\phi\in\End(E)$, where $\tr$ stands for the trace.
\end{Proposition}
\begin{proof}
Consider the short exact sequence of sheaves
\[
 0\rightarrow\cO_X\rightarrow\cK_X\rightarrow\cK_X/\cO_X\rightarrow 0,
\]
where $\cK_X$ is the constant sheaf corresponding to the function field of $X$. Let $\Omega_\cK$ be the constant sheaf of meromorphic differential forms on $X$. We can obtain a new short exact sequence of sheaves:
\[
 0\rightarrow\END(E)\otimes\Omega_X \rightarrow \END(E)\otimes \Omega_\cK \rightarrow\END(E)\otimes(\Omega_\cK/\Omega_X)\rightarrow0
\]
by taking the tensor product of the first sequence with $\END(E)\otimes\Omega_X$. Note that the middle term in this sequence is a constant sheaf, while the last term is an (infinite) direct sum of skyscraper sheaves. That is, $\END(E)\otimes(\Omega_\cK/\Omega_X)\cong\bigoplus_{x\in X} (i_x)_*(\End(E_x)\otimes(\Omega_\cK/\Omega_X)_x)$, where $(\Omega_\cK/\Omega_X)_x$ is the vector space of polar parts at $x$ of meromorphic 1-forms, the summation is taken over all closed points of $X$, and $i_x\colon x\rightarrow X$ is the inclusion. Passing to the long exact sequence for cohomology we obtain the following exact sequence of vector spaces:
\[
\End(E)\otimes\Omega_\cK\rightarrow \bigoplus_{x\in X}\End(E_x)\otimes(\Omega_\cK/\Omega_X)_x\rightarrow H^1(X, \END(E)\otimes\Omega_X) \rightarrow 0.
\]
This implies that $H^1(X, \END(E)\otimes \Omega_X)$ may be presented as the quotient of
\[ \bigoplus_{x\in X}\End(E_x)\otimes(\Omega_\cK/\Omega_X)_x\] by the image of $\End(E)\otimes\Omega_\cK$ (compare with the adelic description of cohomology given in~\cite[Chapter~2, Section~5]{SerreAlgGrClFields}). Further note that the required Higgs field $\Phi$ always exists locally, defined as $\Phi_x =\rho_x\frac{dz_x}{z_x}$, where $z_x$ is an \'etale coordinate near $x$ if $x\in D$ and $\Phi_x=0$ if $x\notin D$. Under the above presentation of $H^1(X, \END(E)\otimes \Omega_X)$, the local solutions $\Phi_x$ define a cohomology class $a(E,D,\rho_\bullet)\in H^1(X, \END(E)\otimes\Omega_X)$. Moreover, it follows from the exact sequence that $a(E,D,\rho_\bullet)=0$ if and only if $\Phi$ can be defined globally.

Serre duality defines a bilinear pairing $H^1(X,\END(E)\otimes\Omega_X)\times\End(E)\rightarrow\kk$. Using the above presentation for $H^1(X, \END(E)\otimes \Omega_X)$ this pairing may be evaluated on $\phi\in\End(E)$ as
\[
\langle a(E,D,\rho_\bullet),\phi\rangle = \sum_{x\in X}\Res_x\tr(\Phi_x\phi_x) = \sum_{x\in D}\tr(\rho_x\phi_x).
\]
Since the pairing is perfect, $\sum\limits_{x\in D}\tr(\rho_x \phi_x) = 0$ for all $\phi\in\End(E)$ if and only if $a(E,D,\rho_\bullet)=0$. The proof is complete.
\end{proof}

\subsection{Parabolic Higgs bundles with fixed eigenvalues}\label{sect:ExistEigen}
Recall that $\kk[D\times\Z_{>0}]$ is the set of all $\kk$-valued sequences $\zeta=\zeta_{\bullet,\bullet}=(\zeta_{x,j})$ indexed by $D\times\Z_{>0}$ such that $\zeta_{x,j}=0$ for $j\gg0$.\footnote{According to our convention we should denote $\zeta$ by $\zeta_{\bullet,\bullet}$ but it does not look nice in the formulas.} For $\zeta\in\kk[D\times\Z_{>0}]$ let $\Higgs(\zeta)=\Higgs(X,D,\zeta)$ denote the full subcategory of $\Higgs$ (and its stack of objects) corresponding to collections $(E,E_{\bullet,\bullet},\Phi)$ such that $(\Phi-\zeta_{x,j}1)(E_{x,{j-1}})\subset E_{x,j}\otimes\Omega_X(D)_x$ for all $x\in D$ and $j>0$. Again, the $\Gamma_+'$-graded stack $\Higgs^-(\zeta)$ is defined following the formalism of Section~\ref{sect:CatsOverPar}.

Let $\zeta\in\kk[D\times\Z_{>0}]$ and let $\gamma=(r,r_{\bullet,\bullet},d)\in\Gamma_+$. Recall from Section~\ref{sect:DegreeSlope} that we have set
\[
 \deg_{0,\zeta}\gamma:=\sum_{x\in D}\sum_{j=1}^{\infty}\zeta_{x,j}r_{x,j}\in\kk.
\]

\begin{Lemma}\label{lm:existence}
Let $\bE\in\Bun^{\rm par}(\kk)$ and $\zeta\in\kk[D\times\Z_{>0}]$. There exists an object $(\bE,\Phi)\in\Higgs(\zeta)(\kk)$ if and only if $\deg_{0,\zeta}\bE'=0$ for any direct summand $\bE'$ of $\bE$.
\end{Lemma}
Note that, in particular, this condition implies that $\deg_{0,\zeta}\bE=0$.
\begin{proof}
The proof is the same as the proof of~\cite[Theorem~7.1]{Crawley-Boevey:Indecomposable} after replacing $b(E)$ with 0 and replacing~\cite[Theorem~7.2]{Crawley-Boevey:Indecomposable} with Proposition~\ref{pr:exist}. (The Atiyah class $b(E)$ represents the obstruction to existence of a connection (without poles) on $E$. Thus, it is absent for Higgs fields essentially because every vector bundle possesses the zero Higgs field.)
\end{proof}

\subsection{Parabolic pairs with isoslopy underlying parabolic bundles}\label{sect:Isoslopy} Recall that for $\kappa\in\kk$ and $\zeta\in\kk[D\times\Z_{>0}]$, a parabolic bundle $\bE\in\Bun^{\rm par}$ is $(\kappa,\zeta)$-isoslopy if
\[
 \frac{\deg_{\kappa,\zeta}\bE'}{\rk\bE'}=\frac{\deg_{\kappa,\zeta}\bE}{\rk\bE}
\]
whenever $\bE'$ is a direct summand of $\bE$. Similarly to~\cite[Lemma~3.2.2]{FedorovSoibelmans}, one checks that the notion of $(\kappa,\zeta)$- isoslopy parabolic bundle is invariant with respect to field extensions. Thus for each $\gamma\in\Gamma_+'$ we have a well-defined subset $\Bun^{{\rm par},(\kappa,\zeta)-{\rm iso},-}_\gamma\subset|\Bun^{{\rm par},-}_\gamma|$. As in~\cite[Lemma~3.2.3]{FedorovSoibelmans} we show that this subset is constructible.

Let $\Pair^{(\kappa,\zeta)-{\rm iso},-}_\gamma\!$ be the preimage of $\Bun^{{\rm par},(\kappa,\zeta)-{\rm iso},-}_\gamma\!$ under the projection $|\Pair|\!\to\!|\Bun^{\rm par}|$.
\begin{Proposition}\label{pr:Sasha}
For $\gamma=(r,r_{\bullet,\bullet},d)\in\Gamma_+$, set $\chi(\gamma):=(g-1)r^2+\sum\limits_{x\in D}
\sum\limits_{j<j'}r_{x,j}r_{x,j'}$. Then
\[
[\Higgs_\gamma^-(\zeta)]=
\begin{cases}
 \bL^{\chi(\gamma)}\big[\Pair^{(0,\zeta)-{\rm iso},-}_\gamma\big] & \text{if }\deg_{0,\zeta}\gamma=0,\\
 0 &\text{otherwise.}
\end{cases}
\]
\end{Proposition}
\begin{proof}
The case of $\deg_{0,\zeta}\gamma\ne0$ is obvious in view of Lemma~\ref{lm:existence}. Assume that $\deg_{0,\zeta}\gamma=0$. It is enough to show the equality of motivic functions in $\Mot\big(\Bun^{{\rm par},-}_\gamma\big)$:
\begin{equation}\label{eq:IsoslopyHiggs}
 \big[\Higgs_\gamma^-(\zeta)\to\Bun^{{\rm par},-}_\gamma\big]=
 \bL^{\chi(\gamma)}\big[\Pair^{(0,\zeta)-{\rm iso},-}_\gamma\to\Bun^{{\rm par},-}_\gamma\big].
\end{equation}
Let $K\supset\kk$ be a field extension. Let $\xi\colon \Spec K\to\Bun^{{\rm par},-}_\gamma$ be a point represented by a parabolic bundle $\bE=(E,E_{\bullet,\bullet})$. In view of Proposition~\ref{pr:MotFunEqual}, we only need to check that the $\xi$-pullbacks of~\eqref{eq:IsoslopyHiggs} are equal. If $\bE$ is not $(0,\zeta)$-isoslopy, then, by Lemma~\ref{lm:existence}, the pullbacks are equal to zero, so we assume that $\bE$ is $(0,\zeta)$-isoslopy.

Let $\higgs(\bE,\zeta)$ denote the space of Higgs fields on $\bE$ with eigenvalues $\zeta$ (that is, the $\bE$-fiber of the projection $\Higgs(\zeta)\to\Bun^{\rm par}$). By Lemma~\ref{lm:existence}, $\higgs(\bE,\zeta)$ is non-empty, so it is a torsor over the vector space $\higgs(\bE,0)$. Thus,
\begin{equation*}
 \xi^*\big[\Higgs_\gamma^-(\zeta)\to\Bun^{{\rm par},-}_\gamma\big]=\bL^{\dim \higgs(\bE,0)}.
\end{equation*}
On the other hand, we have
\[
 \xi^*\big[\Pair^{(0,\zeta)-{\rm iso},-}_\gamma\to\Bun^{{\rm par},-}_\gamma\big]=\bL^{\dim\End(\bE)}.
\]
It remains to prove the following lemma.

\begin{Lemma}\label{lm:PairHiggs}
Let $\bE\in\Bun^{\rm par}_\gamma$ be a parabolic bundle. Then
\[
 \dim\End(\bE)-\dim\higgs(\bE,0)=-\chi(\gamma).
\]
\end{Lemma}
\begin{proof}
Write $\bE=(E,E_{\bullet,\bullet})$. Let $\END(\bE)\subset\END(E)$ be the subsheaf of endomorphisms preserving flags. One checks easily that the trace pairing gives an isomorphism between the dual sheaf $\END(\bE)^\vee\otimes\Omega_X$ and $\HIGGS(\bE,0)$, where $\HIGGS(\bE,0)$ stands for the sheaf of Higgs fields on~$\bE$ with zero eigenvalues. Thus by Riemann--Roch theorem we have
\begin{align*}
 \dim\End(\bE)-\dim\higgs(\bE,0)& =h^0(X,\END(\bE))-h^0\big(X,\END(\bE)^\vee\otimes\Omega_X\big)\\
 & =(1-g)\rk\END(\bE)+\deg\END(\bE).
\end{align*}
It remains to calculate $\deg\END(\bE)$. For $x\in D$ consider the fiber $E_x$, its ring of endomorphisms $\End(E_x)$, its subspace $V_x$ of endomorphisms preserving the flag $E_{x,\bullet}$, and the quotient of vector spaces $W_x:=\End(E_x)/V_x$. Further, consider the torsion sheaf $W:=\oplus_{x\in D}(i_x)_*W_x$, where, as before, $i_x\colon x\to X$ is the inclusion. We have an exact sequence
 \[
 0\to\END(\bE)\to\END(E)\to W\to0,
 \]
 so
 \[
 \deg\END(\bE)=\deg\END(E)-\length(W)=0-\sum_{x\in D}
 \sum_{j<j'}r_{x,j}r_{x,j'}
 \]
and the lemma follows.
\end{proof}

The lemma completes the proof of Proposition~\ref{pr:Sasha}.
\end{proof}

\begin{Proposition}\label{pr:IsoslProd} We have in $\Mot(\kk)[[\Gamma'_+]]$
\[
 [\Pair^-]=
 \prod_{\tau\in\kk}\left(
 \sum_{\substack{\gamma\in\Gamma'_+\\ \deg_{\kappa,\zeta}\gamma=\tau\rk\gamma}}\big[\Pair^{(\kappa,\zeta)-{\rm iso},-}_\gamma\big]e_\gamma
 \right).
\]
\end{Proposition}

We note that the product makes sense because for a $\gamma\in\Gamma_+'$ there are only finitely many ways to write $\gamma$ as the sum of elements of $\Gamma_+'$. Also, the order of the multiples is irrelevant, since we are working with a commutative quantum torus.

\begin{proof}
The proof is almost the same as the proof of~\cite[Lemma~3.5.3 ]{FedorovSoibelmans} (see also~\cite[Proposition~3.5.1]{FedorovSoibelmans}).
\end{proof}

We need some notation. Let us write
\[
 \bL\cdot\Log\left(\sum_\lambda w^{|\lambda|} J_{\lambda,X}^{\rm mot}\big(z^{-1}\big)H_{\lambda,X}^{\rm mot}\big(z^{-1}\big)\prod_{x\in D}\tilde H_\lambda^{\rm mot}\big(w_{x,\bullet};z^{-1}\big)\right)=
 \sum_{\gamma\in\Gamma'_+}\overline B_\gamma e_\gamma,
\]
where $\Log$ is the plethystic logarithm defined in Section~\ref{sect:Plethystic}, the summation is over all partitions. We note that $\overline B_\gamma$ are $W$-invariant, where $W=\prod\limits_{x\in D}\Sigma_\infty$ (cf.~Remark~\ref{rm:Weyl}). Note also that $\overline B_0=0$ by the definition of plethystic logarithm.
\begin{Definition}\label{def:DT}
 The motivic classes $\overline B_\gamma\in\cMot(\kk)$ are called \emph{motivic Donaldson--Thomas invariants} of the pair $(X,D)$.
\end{Definition}

\begin{Corollary}\label{cor:isoslopy} For each $\tau\in\kk$ we have in $\cMot(\kk)[[\Gamma'_+]]$
\[
 \sum_{\substack{\gamma\in\Gamma'_+\\ \deg_{\kappa,\zeta}\gamma=\tau\rk\gamma}}\big[\Pair^{(\kappa,\zeta)-{\rm iso},-}_\gamma\big]e_\gamma=
 \Exp\left(\sum_{\substack{\gamma\in\Gamma'_+\\ \deg_{\kappa,\zeta}\gamma=\tau\rk\gamma}}\overline B_\gamma e_\gamma\right),
\]
where $\Exp$ is the plethystic exponent defined in Section~{\rm \ref{sect:Plethystic}}.
\end{Corollary}
\begin{proof}
First of all, using Corollary~\ref{cor:Pairs} and properties of plethystic operations, we get
\[
 [\Pair^-]=\Exp\left(\sum_{\gamma\in\Gamma'_+}\overline B_\gamma e_\gamma\right)=
 \prod_{\tau\in\kk}\Exp\left(\sum_{\substack{\gamma\in\Gamma'_+\\ \deg_{\kappa,\zeta}\gamma=\tau\rk\gamma}}\overline B_\gamma e_\gamma\right).
\]
Now, it remains to use Proposition~\ref{pr:IsoslProd} and equate the slopes (cf.~\cite[Lemma~3.7.1]{FedorovSoibelmans}).
\end{proof}

\begin{Corollary}\label{cor:MAIN} We have in $\cMot(\kk)[[\Gamma'_+]]$
\[
 \sum_{\gamma\in\Gamma'_+}\bL^{-\chi(\gamma)}[\Higgs^-_\gamma(\zeta)]e_\gamma=
 \Exp\left(\sum_{\substack{\gamma\in\Gamma'_+\\ \deg_{0,\zeta}\gamma=0}}\overline B_\gamma e_\gamma\right),
\]
where $\Exp$ is the plethystic exponent defined in Section~{\rm \ref{sect:Plethystic}}.
\end{Corollary}
\begin{proof}
Use Proposition~\ref{pr:Sasha} and Corollary~\ref{cor:isoslopy}.
\end{proof}

\subsection[Case of $\P^1$]{Case of $\boldsymbol{\P^1}$} Assume now that $X=\P^1$. Then we have a simpler result. Moreover, it is more precise in the sense that we get an answer in $\Mot(\kk)$ rather than in $\cMot(\kk)$. Define the Donaldson--Thomas invariants $B_\gamma\in\Mot(\kk)$ by
\begin{equation}\label{eq:DT_P1}
 \bL\cdot\Log\left(\sum_\lambda\frac{w^{|\lambda|}\prod\limits_{x\in D}\tilde H_\lambda^{\rm mot}\big(w_{x,\bullet};z^{-1}\big)}
 {\prod\limits_{h\in\Hook(\lambda)}
 \big(\bL^{a(h)}-z^{-l(h)-1}\big)\big(\bL^{a(h)+1}-z^{-l(h)}\big)}\right)=
 \sum_{\gamma\in\Gamma'_+}B_\gamma e_\gamma.
\end{equation}
We have precisely the same formula as in Corollary~\ref{cor:MAIN}, where $\overline B_\gamma$ is replaced with $B_\gamma$ (thus, the formula is valid in $\Mot(\kk)$). The proof is the same as of Corollary~\ref{cor:MAIN} except that one uses~\eqref{eq:P1} instead of Corollary~\ref{cor:Pairs}. Comparing Corollary~\ref{cor:Pairs} with~\eqref{eq:P1}, we see that the images of $B_\gamma$ in $\cMot(\kk)$ are equal to $\overline B_\gamma$.

\section{Stability conditions for Higgs bundles}\label{sect:Stability}

\subsection{Harder--Narasimhan filtration} Recall that in Section~\ref{sect:ParWeights} we defined the set $\Stab$ of sequences of parabolic weights. To every sequence of parabolic weights we associated a stability condition on parabolic bundles in Definition~\ref{def:StabilityCond}. We want to extend this to Higgs bundles and to calculate the motivic classes of stacks of semistable parabolic Higgs bundles with nonpositive underlying vector bundles. Let~$X$ and~$D$ be as before and let $\sigma\in\Stab$.

\begin{Definition}\quad
\begin{enumerate}\itemsep=0pt
\item[(i)] A parabolic Higgs bundle $(\bE,\Phi)$ is \emph{$\sigma$-semistable} if~\eqref{eq:ss} is satisfied for all strict subbundles preserved by $\Phi$.
\item[(ii)]
A parabolic Higgs bundle $(\bE,\Phi)=(E,E_{\bullet,\bullet},\Phi)$ is \emph{$\sigma$-nonpositive-semistable}, if the underlying vector bundle of $\bE$ is nonpositive and~\eqref{eq:ss} is satisfied for all strict subbundles $\bE'=(E',E'_{\bullet,\bullet})$ such that $\Phi$ preserves $\bE'$ and $E/E'$ is a nonpositive vector bundle.
\end{enumerate}
\end{Definition}

The notion of $\sigma$-nonpositive-semistable Higgs bundle is similar to that of nonnegative-semi\-stable Higgs bundle (see~\cite[Section~3.3]{FedorovSoibelmans} and~\cite{MozgovoySchiffmanOnHiggsBundles}). We emphasize that a $\sigma$-nonpositive-semistable parabolic Higgs bundle is not necessarily $\sigma$-semistable; cf.~\cite[Remark~3.3.1]{FedorovSoibelmans}.

Denote the substack of $\Higgs(\zeta)$ corresponding to $\sigma$-semistable (resp.~$\sigma$-nonpositive-semi\-stable) parabolic Higgs bundles by $\Higgs^{\sigma-{\rm ss}}(\zeta)$ (resp.~$\Higgs^{\sigma-{\rm ss},-}(\zeta)$). An argument similar to~\cite[Lemma~3.7]{Simpson1} shows that these are open substacks of $\Higgs(\zeta)$ and $\Higgs^-(\zeta)$ respectively. Note that if $(\bE,\Phi)$ is a parabolic Higgs bundle and $\bE'\subset\bE$ is a strict parabolic subbundle preserved by $\Phi$, then we get an induced Higgs field on $\bE/\bE'$; denote it $\Phi'$. Then $(\bE/\bE',\Phi')\in\Higgs^-(\zeta)$. One can use this construction to give $\Higgs^-(\zeta)$ the structure of an exact category. The proof of the following proposition is completely similar to the proof of Proposition~\ref{pr:HN}.

\begin{Proposition}\label{pr:HN3}\quad
\begin{enumerate}\itemsep=0pt
\item[$(i)$]
If $(\bE,\Phi)\in\Higgs(\zeta)$ is a parabolic Higgs bundle with eigenvalues $\zeta$, then there is a unique filtration $0=\bE_0\subset\bE_1\subset\dots\subset\bE_m=\bE$ by strict parabolic subbundles preserved by $\Phi$ such that all the quotients $\bE_i/\bE_{i-1}$ with induced Higgs fields are $\sigma$-semistable parabolic Higgs bundles and we have $\tau_1>\dots>\tau_m$, where $\tau_i$ is the $(1,\sigma)$-slope of $\bE_i/\bE_{i-1}$.
\item[$(ii)$] If $(\bE,\Phi)\in\Higgs^-(\zeta)$, then there is a unique filtration of $\bE$ as in~(i) by strict parabolic subbundles preserved by $\Phi$ with quotients being $\sigma$-nonpositive-semistable parabolic Higgs bundles.
\end{enumerate}
\end{Proposition}

\subsection{Kontsevich--Soibelman factorization formula}\label{sect:KS}
The general formalism of~\cite{KontsevichSoibelman08} implies the following factorization formula valid in $\Mot(\kk)[[\Gamma'_+]]$. One can also give a direct proof along the lines of the proof of~\cite[Proposition~3.6.1]{FedorovSoibelmans}\footnote{Note that all but countably many multiples are equal to one. We can understand the countable product as a~clockwise product as in~\cite{KontsevichSoibelman08,KontsevichSoibelman10}. Note, however, that this is a product in a commutative ring.}
\begin{equation}\label{eq:KS}
\sum_{\gamma\in\Gamma'_+}\bL^{-\chi(\gamma)}[\Higgs^-_\gamma(\zeta)]e_\gamma=
\prod_{\tau\in\R}\left(
\sum_{\substack{\gamma\in\Gamma'_+\\ \deg_{1,\sigma}\gamma=\tau\rk\gamma}}\bL^{-\chi(\gamma)}\big[\Higgs^{\sigma-{\rm ss},-}_\gamma(\zeta)\big]e_\gamma
\right).
\end{equation}
Now, taking the plethystic logarithms of both sides and using Corollary~\ref{cor:MAIN}, we get the following statement.
\begin{Proposition}\label{pr:expl-ss>=0} We have in $\cMot(\kk)[[\Gamma'_+]]$
 \[
 \sum_{\substack{\gamma\in\Gamma'_+\\ \deg_{1,\sigma}\gamma=\tau\rk\gamma}}\bL^{-\chi(\gamma)}\big[\Higgs^{\sigma-{\rm ss},-}_\gamma(\zeta)\big]e_\gamma=
 \Exp\left(\sum_{\substack{\gamma\in\Gamma'_+\\ \deg_{0,\zeta}\gamma=0\\ \deg_{1,\sigma}\gamma=\tau\rk\gamma}}\overline B_\gamma e_\gamma\right).
 \]
If $X=\P^1$, then the same formula holds in $\Mot(\kk)[[\Gamma'_+]]$ with $\overline B_\gamma$ replaced by $B_\gamma$, where $B_\gamma$ are defined by~\eqref{eq:DT_P1}.
\end{Proposition}

\section{Stabilization}\label{sect:Stabilization}
\subsection{Stabilization of semistable Higgs bundles} Let $X$ and $D$ be as before. We will be assuming that $D\ne\varnothing$. Note that this implies that $X$ has a $\kk$-rational divisor of degree one. Set $\delta:=\max(2g-2+\deg D,0)$. Fix a stability condition $\sigma\in\Stab$. Our goal in this section is to calculate the motivic class of the moduli stack of $\sigma$-semistable parabolic Higgs bundles without nonnegativity assumption. The main result in this section is Theorem~\ref{th:ExplAnsw}. Recall that in the end of Section~\ref{sect:ParWeights} we defined the categories $\Bun^{{\rm par},\le\tau}$ and $\Bun^{{\rm par},\ge\tau}$. These are the full subcategories of $\Bun^{\rm par}$ whose objects are parabolic bundles with the $\sigma$-HN spectrum contained in $(-\infty,\tau]$ and~$[\tau,\infty)$ respectively.

We start with the following analogue of~\cite[Lemma~3.1]{MozgovoySchiffmanOnHiggsBundles}.
\begin{Lemma}\label{lm:gap}
Let $\bE\in\Bun^{\rm par}$ be a parabolic bundle with the $\sigma$-HN-spectrum $\tau_1>\tau_2>\dots>\tau_m$. Assume that for some $i\in\{1,\dots,m-1\}$ we have $\tau_i-\tau_{i+1}>\delta$. Then there are no $\sigma$-semistable Higgs bundles of the form $(\bE,\Phi)$.
\end{Lemma}
\begin{proof}
Assume the contrary. We have an exact sequence
\[
 0\to\bE^{\ge}\to\bE\to\bE^{\le}\to0,
\]
where $\bE^{\ge}\in\Bun^{{\rm par},\ge\tau_i}$, $\bE^{\le}\in\Bun^{{\rm par},\le\tau_{i+1}}$. Then $\Phi$ induces a morphism $\Phi':\bE^{\ge}\to\bE^{\le}\otimes\Omega_X(D)$. Note that $\bE^{\le}\otimes\Omega_X(D)\in\Bun^{{\rm par},\le\tau_{i+1}+\delta}$. By Lemma~\ref{lm:NoMorphismSS}, $\Phi'=0$ and we see that $\bE^{\ge}$ is preserved by $\Phi$ contradicting $\sigma$-semistability of $(\bE,\Phi)$.
\end{proof}

Next, we have an analogue of~\cite[Lemma~3.2]{MozgovoySchiffmanOnHiggsBundles}.
\begin{Lemma}\label{lm:MozSchif}
Let $(\bE,\Phi)$ be a $\sigma$-semistable Higgs bundle. Assume that $\deg_{1,\sigma}\bE<-\frac{r(r-1)}2\delta$, where $r=\rk\bE$. Then $\bE\in\Bun^{{\rm par},\le0}$.
\end{Lemma}
\begin{proof}
Let $\tau_1>\tau_2>\dots>\tau_m$ be the $\sigma$-HN-spectrum of $\bE$. Denote by $r_i$ the jumps of the ranks of $\sigma$-HN-filtration. By Lemma~\ref{lm:gap} we have $\tau_i\ge\tau_1-(i-1)\delta$. We have
\[
-\frac{r-1}2\delta>\frac{\deg_{1,\sigma}\bE}r=\frac{\sum\limits_{i=1}^m\tau_ir_i}r\ge\frac{\sum\limits_{i=1}^m(\tau_1-(i-1)\delta)r_i}r\ge
\frac mr\tau_1-\frac{r-1}2\delta
\]
and the statement follows.
\end{proof}

Set $|\sigma|:=\sum\limits_{x\in D}(\sup_i\sigma_{x,i}-\sigma_{x,1})$. We remark that we always have $|\sigma|\le\deg D$. We have an analogue of~\cite[Corollary~3.3]{MozgovoySchiffmanOnHiggsBundles}.
\begin{Lemma}\label{lm:ss=ss>=0}
 Let $\gamma=(r,r_{\bullet,\bullet},d)\in\Gamma_+$ be such that $d<-r|\sigma|-\frac{r(r-1)}2\delta$. Then
 \[
 \Higgs_\gamma^{\sigma-{\rm ss}}(\zeta)=\Higgs_\gamma^{\sigma-{\rm ss},-}(\zeta).
 \]
\end{Lemma}
\begin{proof}
 For all $x$ and $i$ replace $\sigma_{x,i}$ by $\sigma_{x,i}-\sigma_{x,1}$. This does not change $|\sigma|$ and the notion of semistability but we now have $\sigma_{x,i}\ge0$ for all $x$ and $i$. Next
 \begin{equation}\label{eq:|sigma|}
 \deg_{1,\sigma}\gamma=d+\sum_x\sum_{i=1}^\infty\sigma_{x,i}r_{x,i}\le d+
 \sum_x\Big(\sup_i\sigma_{x,i}\Big)\left(\sum_{i=1}^\infty r_{x,i}\right)=d+r|\sigma|.
 \end{equation}

 Let $(\bE,\Phi)\in\Higgs_\gamma^{\sigma-{\rm ss}}(\zeta)$. By~\eqref{eq:|sigma|} we have{\samepage
 \[
 \deg_{1,\sigma}\gamma<-\frac{r(r-1)}2\delta.
 \]
 By Lemma~\ref{lm:MozSchif} we have $\bE\in\Bun^{{\rm par},\le0}\subset\Bun^{{\rm par},-}$ (the last inclusion follows from $\sigma_{x,i}\ge0$).}

 Conversely, assume that $(\bE,\Phi)\in\Higgs_\gamma^{\sigma-{\rm ss},-}(\zeta)$. Assume for contradiction that $(\bE,\Phi)$ is not $\sigma$-semistable. Then by Proposition~\ref{pr:HN3}(i) we have an exact sequence $0\to\bE'\to\bE\to\bE''\to0$ in~$\Bun^{\rm par}$ such that~$\Phi$ preserves $\bE'$, and $(\bE'',\Phi'')$ is $\sigma$-semistable, where $\Phi''$ is the induced Higgs field. Using~\eqref{eq:|sigma|} we get
 \[
 \frac{\deg_{1,\sigma}\bE''}{\rk\bE''}<\frac{\deg_{1,\sigma}\bE}{\rk\bE}\le\frac{d+r|\sigma|}r-\frac{r-1}2\delta\le
 -\frac{\rk\bE''-1}2\delta.
 \]
 Now it follows from Lemma~\ref{lm:MozSchif} that $\bE''\in\Bun^{{\rm par},\le0}$.

\looseness=1 Write $\bE''=(E'',E''_{\bullet,\bullet})$. Since $(\bE,\Phi)$ is $\sigma$-nonpositive-semistable, $E''$ cannot be nonpositive. Thus there is $E'''\subset E''$ with $\deg E'''>0$. Let $\bE'''=(E''',E'''_{\bullet,\bullet})$ be the corresponding parabolic subbundle of $\bE''$. Then $\deg_{1,\sigma}\bE'''\ge\deg E'''>0$, which gives contradiction with $\bE''\in\Bun^{{\rm par},\le0}$.
\end{proof}

Set $\mathbf1=(0,0_{\bullet,\bullet},1)\in\Gamma$ (here $0_{\bullet,\bullet}$ is the sequence of zeroes indexed by $D\times\Z_{>0}$). If $\gamma\in\Gamma_+$ and $\gamma\ne0$, then $\gamma+N\mathbf1\in\Gamma_+$ for all $N\in\Z$.
\begin{Corollary}\label{cor:Stabilization}
 Let $\gamma=(r,r_{\bullet,\bullet},d)\in\Gamma_+$, $\gamma\ne0$, and $N>|\sigma|+\frac{r-1}2\delta+d/r$. Then
 \[
 \Higgs_\gamma^{\sigma-{\rm ss}}(\zeta)\simeq\Higgs_{\gamma-Nr\mathbf1}^{\sigma-{\rm ss},-}(\zeta).
 \]
\end{Corollary}
\begin{proof}
 Since $X$ has a divisor of degree one, it has a line bundle of degree $N$. Tensorisation with this line bundle gives $\Higgs_\gamma^{\sigma-{\rm ss}}(\zeta)\simeq\Higgs_{\gamma-Nr\mathbf1}^{\sigma-{\rm ss}}(\zeta)$.
 Now Lemma~\ref{lm:ss=ss>=0} completes the proof.
\end{proof}

Recall from Definition~\ref{def:DT} the Donaldson--Thomas invariants $\overline B_\gamma\in\cMot(\kk)$. For each $\tau\in\R$ define the elements $H_\gamma(\zeta,\sigma)\in\cMot(\kk)$, where $\gamma\in\Gamma_+$ is such that the $(1,\sigma)$-slope of $\gamma$ is $\tau$ (or $\gamma=0$), by the following formula.
\begin{equation}\label{eq:ExplAnswer}
\sum_{\substack{\gamma\in\Gamma_+'\\ \deg_{0,\zeta}\gamma=0\\ \deg_{1,\sigma}\gamma=\tau\rk\gamma}}\bL^{-\chi(\gamma)}H_\gamma(\zeta,\sigma)e_\gamma=
 \Exp\left(\sum_{\substack{\gamma\in\Gamma_+'\\ \deg_{0,\zeta}\gamma=0\\ \deg_{1,\sigma}\gamma=\tau\rk\gamma}}
 \overline B_\gamma e_\gamma
 \right).
\end{equation}
Thus $H_\gamma(\zeta,\sigma)$ is defined for all $\gamma$ such that $\deg_{0,\zeta}\gamma=0$. Note that $H_0(\zeta,\sigma)=1$. Now we can formulate our first main result.
\begin{Theorem}\label{th:ExplAnsw} Let $\gamma=(r,r_{\bullet,\bullet},d)\in\Gamma_+$, $\gamma\ne0$.
\begin{enumerate}\itemsep=0pt
\item[$(i)$] The elements $H_\gamma(\zeta,\sigma)$ are periodic in the following sense: for $d<-|\sigma|-\frac{r-1}2\delta$ we have $H_\gamma(\zeta,\sigma)=H_{\gamma-r\mathbf1}(\zeta,\sigma)$.
\item[$(ii)$] The stack $\Higgs_\gamma^{\sigma-{\rm ss}}(\zeta)$ is of finite type and we have in $\cMot(\kk)$
\begin{equation}\label{eq:ThExpl1}
 \big[\Higgs_\gamma^{\sigma-{\rm ss}}(\zeta)\big]=H_{\gamma-Nr\mathbf1}(\zeta,\sigma)
\end{equation}
whenever $N$ is large enough, provided that $\deg_{0,\zeta}\gamma=0$ {\rm(}it suffices to take $N>|\sigma|+\frac{r-1}2\delta+d/r${\rm)}. If $\deg_{0,\zeta}\gamma\ne0$, then the stack is empty.
\end{enumerate}
\end{Theorem}
\begin{proof}
For part~(ii) combine Corollary~\ref{cor:Stabilization} with Proposition~\ref{pr:expl-ss>=0}. Part~(i) is clear from \linebreak part~(ii).
\end{proof}

An immediate corollary of the above theorem and formula~\eqref{eq:ExplAnswer} is the following curious observation.
\begin{Corollary}\label{cor:EqualMot}
 Assume that we are given $\gamma\in\Gamma_+$, sets of eigenvalues $\zeta$ and $\zeta'$, and sequences of parabolic weights $\sigma$, $\sigma'$. Let $\tau$ and $\tau'$ be $(1,\sigma)$ and $(1,\sigma')$-slopes of $\gamma$ respectively. Assume also that
 \begin{gather*}
 \{\gamma'\in\Gamma_+\colon \deg_{0,\zeta}\gamma'=0, \deg_{1,\sigma}\gamma'=\tau\rk\gamma\}=
 \{\gamma'\in\Gamma_+\colon \deg_{0,\zeta'}\gamma'=0, \deg_{1,\sigma'}\gamma'=\tau'\rk\gamma\}.\!
 \end{gather*}
 Then we have an equality of motivic classes
 \[
 \big[\Higgs_\gamma^{\sigma-{\rm ss}}(\zeta)\big]=\big[\Higgs_\gamma^{\sigma'-{\rm ss}}(\zeta')\big].
 \]
\end{Corollary}

\subsection[Case of $\P^1$]{Case of $\boldsymbol{\P^1}$}\label{sect:StabP1} If $X=\P^1$, we obtain simpler and more precise results. Namely, if we define elements $H_\gamma(\zeta,\sigma)$ by the same formula~\eqref{eq:ExplAnswer} but with $B_\gamma$ instead of $\overline B_\gamma$, then~\eqref{eq:ThExpl1} holds in $\Mot(\kk)$.

\section{Motivic classes of parabolic connections}\label{sect:Conn}
\subsection{Stacks of parabolic connections} Let $X$ and $D$ be as above. Our goal in this section is to calculate the motivic classes of the moduli stacks of parabolic bundles with connections with prescribed eigenvalues of residues. In Section~\ref{sect:StabConn} we put stability conditions on these moduli stacks and calculate the motivic classes of substacks of semistable parabolic bundles with connections. Our argument is similar to the argument for Higgs bundles.

Let $E$ be a vector bundle on $X$. A \emph{connection} on $E$ with \emph{poles bounded by $D$} is a morphism of sheaves of abelian groups $\nabla\colon E\to E\otimes\Omega_X(D)$ satisfying Leibniz rule. In this case for $x\in D$ one defines the residue of the connection $\res_x\nabla\in\End(E_x)$.

\begin{Definition}
A \emph{parabolic connection} of type $(X,D)$ is a triple $(E,E_{\bullet,\bullet},\nabla)$, where $(E,E_{\bullet,\bullet})$ is a point of $\Bun^{\rm par}(X,D)$, $\nabla\colon E\to E\otimes\Omega_X(D)$ is a connection on $E$ such that for all $x\in D$ and $j\ge0$ we have $(\res_x\nabla)(E_{x,j})\subset E_{x,j}$.
\end{Definition}

We denote the category (and the Artin stack) of parabolic connections by $\Conn\!=\!\Conn(X,D)$. In this section $X$ and $D$ are fixed, so we skip them from the notation.

Recall that $\kk[D\times\Z_{>0}]$ is the set of all $\kk$-valued sequences $\zeta=\zeta_{\bullet,\bullet}=(\zeta_{x,j})$ indexed by $D\times\Z_{>0}$ such that $\zeta_{x,j}=0$ for $j\gg0$. For $\zeta\in\kk[D\times\Z_{>0}]$ let $\Conn(\zeta)=\Conn(X,D,\zeta)$ denote the full subcategory of $\Conn$ (and its stack of objects) corresponding to collections $(E,E_{\bullet,\bullet},\nabla)$ such that $(\res_x\nabla-\zeta_{x,j}1)(E_{x,{j-1}})\subset E_{x,j}$ for all $x\in D$ and $j>0$. We call the points of $\Conn(\zeta)$ the parabolic bundles with connections \emph{with eigenvalues $\zeta$.}

Let $\zeta\in\kk[D\times\Z_{>0}]$ and let $\gamma=(r,r_{\bullet,\bullet},d)\in\Gamma_+$. Recall from Section~\ref{sect:DegreeSlope} that
\[
 \deg_{1,\zeta}\gamma=d+\sum_{x\in D}\sum_{j=1}^{\infty}\zeta_{x,j}r_{x,j}\in\kk.
\]
The following lemma is~\cite[Theorem~7.1]{Crawley-Boevey:Indecomposable} if $\kk=\C$. The proof in the general case is completely similar.

\begin{Lemma}\label{lm:existence2}
Let $\bE\in\Bun^{\rm par}(\kk)$ and $\zeta\in\kk[D\times\Z_{>0}]$. There exists an object $(\bE,\nabla)\in\Conn(\zeta)(\kk)$ if and only if $\deg_{1,\zeta}\bE'=0$ for any direct summand $\bE'$ of $\bE$.
\end{Lemma}
Note that, in particular, for every $(\bE,\nabla)\in\Conn(\zeta)(\kk)$ we have $\deg_{1,\zeta}\bE=0$.
Recall from Section~\ref{sect:Isoslopy} the notion of $(\kappa,\zeta)$-isoslopy parabolic bundle and the stacks $\Pair^{(\kappa,\zeta)-{\rm iso}}_\gamma$. Recall also that for $\gamma=(r,r_{\bullet,\bullet},d)\in\Gamma_+$, we have set $\chi(\gamma):=(g-1)r^2+\sum\limits_{x\in D}\sum\limits_{j<j'}r_{x,j}r_{x,j'}$.
\begin{Proposition}\label{pr:Sasha2}
We have
\[
[\Conn_\gamma(\zeta)]=
\begin{cases}
 \bL^{\chi(\gamma)}\big[\Pair^{(1,\zeta)-{\rm iso}}_\gamma\big]& \text{if }\deg_{1,\zeta}\gamma=0,\\
 0 & \text{otherwise}.
\end{cases}
\]
\end{Proposition}
\begin{proof}
The proof is completely analogous to the proof of Proposition~\ref{pr:Sasha} with Lemma~\ref{lm:existence} replaced by Lemma~\ref{lm:existence2}.
\end{proof}

\subsection{Stabilization of isoslopy parabolic bundles}\label{sect:Stabilization2}
As in Section~\ref{sect:Stabilization} we will be assuming that $D\ne\varnothing$. Recall that this implies that $X$ has a $\kk$-rational divisor of degree one. As before, set $\delta:=\max(2g-2+\deg D,0)$.

Recall that every vector bundle $E$ on $X$ has a unique HN-filtration and the slopes of the quotients form a sequence called the HN-spectrum of $E$. We start with the following analogue of~\cite[Lemma~4.1]{MozgovoySchiffmanOnHiggsBundles}.
\begin{Lemma}\label{lm:gap2}
Let $\bE=(E,E_{\bullet,\bullet})\in\Bun^{\rm par}$ be a parabolic bundle such that $E$ has HN-spectrum $\tau_1>\tau_2>\dots>\tau_m$. Assume that for some $i\in\{1,\dots,m-1\}$ we have $\tau_i-\tau_{i+1}>\delta$. Then $\bE$ is decomposable.
\end{Lemma}
\begin{proof}One shows that the extensions of a parabolic bundle $\bE''$ by a parabolic bundle $\bE'$ (in the sense of Section~\ref{sect:Subobjects}) are classified by a vector space $\Ext^1(\bE'',\bE')$ dual to $\Hom(\bE',\bE''(\Omega_X(D))$. Let $0=E_0\subset E_1\subset\dots\subset E_m=E$ be the Harder--Narasimhan filtration of $E$. Let $\bE_i$ be the strict parabolic subbundle with the underlying vector bundle $E_i$. We have an exact sequence $0\to\bE_i\to\bE\to\bE/\bE_i\to0$. Note that by the assumption the Harder--Narasimhan spectrum of $E_i$ is contained in $[\tau_i,\infty)$, while the Harder--Narasimhan spectrum of $(E/E_i)(\Omega_X(D))$ is contained in $(-\infty,\tau_i)$. It follows that $\Hom\bigl(E_i,(E/E_i)(\Omega_X(D))\bigr)=0$. Thus
\[
 \Ext^1(\bE/\bE_i,\bE_i)=\Hom\bigl(\bE_i,(\bE/\bE_i)(\Omega_X(D))\bigr)^\vee=0.
\]
Thus $\bE\simeq\bE_i\oplus(\bE/\bE_i)$ is decomposable.
\end{proof}

Next, we have an analogue of~\cite[Corollary~4.2]{MozgovoySchiffmanOnHiggsBundles} whose proof is similar to loc.~cit.~and to that of Lemma~\ref{lm:MozSchif}.
\begin{Lemma}\label{lm:MozSchif2}
Let $\bE\in\Bun^{\rm par}$ be indecomposable and $\cl(\bE)=(r,r_{\bullet,\bullet},d)$. Assume that $d<-\frac{r(r-1)}2\delta$. Then $\bE\in\Bun^{{\rm par},-}$.
\end{Lemma}
\begin{proof}
Write $\bE=(E,E_{\bullet,\bullet})$. Let $\tau_1>\tau_2>\dots>\tau_m$ be the HN-spectrum of $E$. Denote by $r_i$ the jumps of the ranks of HN-filtration. By Lemma~\ref{lm:gap} we have $\tau_i\ge\tau_1-(i-1)\delta$. We have
\[
-\frac{r-1}2\delta>\frac dr=\frac{\sum\limits_{i=1}^m\tau_ir_i}r\ge\frac{\sum\limits_{i=1}^m(\tau_1-(i-1)\delta)r_i}r\ge
\frac mr\tau_1-\frac{r-1}2\delta
\]
and the statement follows.
\end{proof}

Fix $\zeta\in\kk[D\times\Z_{>0}]$. Let $|\bullet|$ be any norm on the $\Q$-vector subspace of $\kk$ generated by the components of~$\zeta$. If~$\kk$ is embedded into $\C$, we can take the usual absolute value for $|\bullet|$. We set $|\zeta|:=\sum\limits_{x\in D}(\max_i|\zeta_{x,i}|)$. We have an analogue of~\cite[Lemma~3.2.3(i)]{FedorovSoibelmans} (cf.~also Lemma~\ref{lm:ss=ss>=0}).
\begin{Lemma}\label{lm:ss=ss>=02}
 Let $\gamma=(r,r_{\bullet,\bullet},d)\in\Gamma_+$ be such that $d<-2r|\zeta|-\frac{r(r-1)}2\delta$. Then
 \[
 \Bun_\gamma^{{\rm par},(1,\zeta)-{\rm iso}}=\Bun_\gamma^{{\rm par},(1,\zeta)-{\rm iso},-}.
 \]
\end{Lemma}
\begin{proof}
 Take $\bE=(E,E_{\bullet,\bullet})\in\Bun_\gamma^{{\rm par},(1,\zeta)-{\rm iso}}$. Assume for a contradiction that $\bE\notin\Bun^{{\rm par},-}$. By Lemma~\ref{lm:MozSchif2}, $\bE$ is decomposable. Let $\bE'$ be an indecomposable summand of $\bE$ such that $\bE'\notin\Bun^{{\rm par},-}$. By the definition of isoslopy bundles, we have $\frac{\deg_{1,\zeta}\bE'}{\rk\bE'}=\frac{\deg_{1,\zeta}\bE}r$. Write $\cl(\bE')=(r',r'_{\bullet,\bullet},d')$. We have
 \[
 \frac{d'}{r'}\le\frac{\deg_{1,\zeta}\bE'}{\rk\bE'}+|\zeta|=\frac{\deg_{1,\zeta}\bE}r+|\zeta|\le\frac dr+2|\zeta|<
 -\frac{(r-1)}2\delta\le-\frac{(r'-1)}2\delta
 \]
and Lemma~\ref{lm:MozSchif2} gives a contradiction.
\end{proof}

Recall that we have $\mathbf1=(0,0_{\bullet,\bullet},1)\in\Gamma$.
\begin{Corollary}\label{cor:ExplAnswer2}
 Let $\gamma=(r,r_{\bullet,\bullet},d)\in\Gamma_+$, $\gamma\ne0$, and $N>2|\zeta|+\frac{r-1}2\delta+d/r$. Then $\Pair_\gamma^{(1,\zeta)-{\rm iso}}\simeq\Pair_{\gamma-Nr\mathbf1}^{(1,\zeta)-{\rm iso},-}$. In particular, $\Pair_\gamma^{(1,\zeta)-{\rm iso}}$ is a constructible subset of $\Pair_\gamma$ of finite type.
\end{Corollary}
\begin{proof}
 Since $X$ has a divisor of degree one, it has a line bundle of degree $N$. Tensorisation with this line bundle gives
 $\Pair_\gamma^{(1,\zeta)-{\rm iso}}\simeq\Pair_{\gamma-Nr\mathbf1}^{(1,\zeta)-{\rm iso}}$. Now Lemma~\ref{lm:ss=ss>=02} completes the proof.
\end{proof}

Define the elements $C_\gamma(\zeta)\in\cMot(\kk)$, where $\gamma$ ranges over $\Gamma_+$ by the following formula (cf.~\eqref{eq:ExplAnswer})
\begin{equation}\label{eq:ExplAnswer2}
\sum_{\substack{\gamma\in\Gamma_+'\\ \deg_{1,\zeta}\gamma=\tau\rk\gamma}}\bL^{-\chi(\gamma)}C_\gamma(\zeta)e_\gamma=
 \Exp\left(\sum_{\substack{\gamma\in\Gamma_+'\\ \deg_{1,\zeta}\gamma=\tau\rk\gamma}}
 \overline B_\gamma e_\gamma
 \right).
\end{equation}
Now we can formulate our second main result. Recall that $\Conn(\zeta)=\varnothing$ unless $\deg_{1,\zeta}\gamma=0$.
\begin{Theorem}\label{th:ExplAnsw2} Let $\gamma=(r,r_{\bullet,\bullet},d)\in\Gamma_+$, $\gamma\ne0$ and $\deg_{1,\zeta}\gamma=0$. Let $\zeta$ be an element of $\kk[D\times\Z_{>0}]$ and let $|\bullet|$ be a norm on the $\Q$-vector subspace of $\kk$ generated by the components of~$\zeta$. Then
\begin{enumerate}\itemsep=0pt
\item[$(i)$] The elements $C_\gamma(\zeta)$ are periodic in the following sense: for $d<-2|\zeta|-\frac{r-1}2\delta$ we have $C_\gamma(\zeta)=C_{\gamma-r\mathbf1}(\zeta)$.
\item[$(ii)$]
The stack $\Conn_\gamma(\zeta)$ is of finite type and we have
\[
 [\Conn_\gamma(\zeta)]=C_{\gamma-Nr\mathbf1}(\zeta),
\]
whenever $N$ is large enough {\rm(}it suffices to take $N>2|\zeta|+\frac{r-1}2\delta+d/r${\rm)}.
\end{enumerate}
\end{Theorem}
\begin{proof}
 Combine Proposition~\ref{pr:Sasha2}, Corollary~\ref{cor:ExplAnswer2}, and Corollary~\ref{cor:isoslopy}.
\end{proof}

\subsubsection[Case of $\P^1$]{Case of $\boldsymbol{\P^1}$}\label{sect:StabP12} If $X=\P^1$, we obtain simpler and more precise results. Define $B_\gamma\in\Mot(\kk)$ by~\eqref{eq:DT_P1}. Define $C_\gamma(\zeta)\in\Mot(\kk)$ by the same formula~\eqref{eq:ExplAnswer2} but with $\overline B_\gamma$ replaced by $B_\gamma$. Then Theorem~\ref{th:ExplAnsw2} holds in $\Mot(\kk)$.

\subsection{Stability conditions for bundles with connections}\label{sect:StabConn} Recall that in Definition~\ref{def:StabilityCond} we defined the notion of a sequence of parabolic weights. For non-resonant connections one can work with more general sequences of parabolic weights. Let us give the definitions.
\begin{Definition}
 We say that $\zeta=\zeta_{\bullet,\bullet}\in\kk[D\times\Z_{>0}]$ is \emph{non-resonant} if for all $x\in X$ and all $i,j>0$ we have $\zeta_{x,i}-\zeta_{x,j}\notin\Z_{\ne0}$.
\end{Definition}
The importance of this definition is in the following lemma.
\begin{Lemma}\label{lm:non-resonant}
Let $\zeta\in\kk[D\times\Z_{>0}]$ be non-resonant and let $\phi$ be a morphism in $\Conn(\zeta)$ such that the underlying morphism of vector bundles is generically an isomorphism. Then the underlying morphism of vector bundles is an isomorphism.
\end{Lemma}
\begin{proof}
One easily reduces to the case $\kk=\C$. Take $x\in D$. Since $\zeta$ is non-resonant, one can find a subset $\Omega$ of $\C$ containing $\{\zeta_{x,j}|j>0\}$ and such that the exponential function induces a bijection between $\Omega$ and $\C-0$. Then it is well-known that every regular connection on the punctured formal disc has a unique extension to the puncture such that the eigenvalues of the residues are in $\Omega$. The statement follows.
\end{proof}

Define the space of extended sequences of parabolic weights $\Stab'$ as the set of pairs $(\kappa,\sigma)$, where $\kappa\in\R_{\ge0}$, $\sigma=\sigma_{\bullet,\bullet}$ is a sequence of real numbers, indexed by $D\times\Z_{>0}$, such that for all $x\in D$ we have~\eqref{eq:StabCond2}.

\begin{Definition}
Let $(\kappa,\sigma)\in\Stab'$. A parabolic connection $(\bE,\nabla)$ is \emph{$(\kappa,\sigma)$-semistable} if for all strict parabolic subbundles $\bE'\subset\bE$ preserved by $\nabla$ we have
 \begin{equation*}
 \frac{\deg_{\kappa,\sigma}\bE'}{\rk\bE'}\le\frac{\deg_{\kappa,\sigma}\bE}{\rk\bE}.
 \end{equation*}
\end{Definition}

We denote by $\Conn^{(\kappa,\sigma)-{\rm ss}}(\zeta)$ the substack of $\Conn(\zeta)$ corresponding to $(\kappa,\sigma)$-semistable connections. An argument similar to~\cite[Lemma~3.7]{Simpson1} shows that this is an open substacks of $\Conn(\zeta)$.

\begin{Proposition}\label{pr:HN2} Assume that $(\kappa,\sigma)\in\Stab'$. Assume also that either $\zeta$ is non-resonant, or $\kappa=1$, $\sigma\in\Stab$. If $(\bE,\nabla)\in\Conn(\zeta)$, then there is a unique filtration $0=\bE_0\subset\bE_1\subset\dots\subset\bE_m=\bE$ by strict parabolic subbundles preserved by $\nabla$ such that all the quotients $\bE_i/\bE_{i-1}$ with induced connections are $(\kappa,\sigma)$-semistable parabolic bundles with connections and we have $\tau_1>\dots>\tau_m$, where $\tau_i$ is the $(\kappa,\sigma)$-slope of $\bE_i/\bE_{i-1}$.
\end{Proposition}
\begin{proof}
In the non-resonant case the proof is completely analogous to the proof of~\cite[Section~1.3]{HarderNarasimhan} in view of Lemma~\ref{lm:MorPar} and the following lemma.
\begin{Lemma}
 Let $\zeta$ be non-resonant and let $\bE\to\bF$ be a morphism in $\Conn(\zeta)$, which is generically an isomorphism. Then $\deg_{\kappa,\sigma}\bE\le\deg_{\kappa,\sigma}\bF$.
\end{Lemma}
\begin{proof}
 Write $\bE=(E,E_{\bullet,\bullet})$ and $\bF=(F,F_{\bullet,\bullet})$. Let $\phi\colon E\to F$ be the underlying morphism of vector bundles. By Lemma~\ref{lm:non-resonant}, $\phi$ is an isomorphism. Thus $\dim E_{x,j}\le\dim F_{x,j}$ for all~$x$ and~$j$. Therefore
 \begin{align*}
 \deg_{\kappa,\sigma}\bE& =\kappa\deg E+\sum_{x,j>0}\sigma_{x,j}(\dim E_{x,j-1}-\dim E_{x,j}) \\
 & = \kappa\deg E+\sum_{x\in D}\left(\sigma_{x,1}\rk E+\sum_{i>0}(\sigma_{x,i+1}-\sigma_{x,i})\dim E_{x,i}\right)\\
& \le \kappa\deg F+\sum_{x\in D}\left(\sigma_{x,1}\rk F+\sum_{i>0}(\sigma_{x,i+1}-\sigma_{x,i})\dim F_{x,i}\right)=\deg_{\kappa,\sigma}\bF.\tag*{\qed}
 \end{align*}\renewcommand{\qed}{}
\end{proof}

In the resonant case, the proof is completely analogous to the proof of~\cite[Section~1.3]{HarderNarasimhan} in view of Lemma~\ref{lm:ModifDegree} (cf.\ Propositions~\ref{pr:HN} and~\ref{pr:HN3}).
\end{proof}

\begin{Remark}\label{rm:resonant}
 More generally, If $\zeta$ is resonant, one can work with any $(\kappa,\sigma)\in\Stab'$ such that $\sigma_{x,j}-\sigma_{x,1}\le\kappa$ for all $x$ and $j$. However, the notion of stability does not change if we scale~$(\kappa,\sigma)$. Thus, we can always assume that $\kappa=1$, in which case $\sigma\in\Stab$, or $\kappa=0$, in which case $\sigma=0$. The latter case corresponds to the trivial stability condition; the corresponding motivic class has been calculated in Theorem~\ref{th:ExplAnsw2}.
\end{Remark}

Similarly to Proposition~\ref{pr:expl-ss>=0} the Kontsevich--Soibelman factorization formula implies the following proposition.
\begin{Proposition}\label{pr:expl-ss>=02} Let $(\kappa,\sigma)\in\Stab'$. Assume that either $\zeta$ is non-resonant or $\kappa=1$ and $\sigma\in\Stab$. Then we have in $\Mot(\kk)[[\Gamma_+]]$
\begin{equation}\label{eq:KS2}
\sum_{\gamma\in\Gamma_+}\bL^{-\chi(\gamma)}[\Conn_\gamma(\zeta)]e_\gamma=
\prod_{\tau\in\R}\left(
\sum_{\substack{\gamma\in\Gamma_+\\ \deg_{\kappa,\sigma}\gamma=\tau\rk\gamma}}\bL^{-\chi(\gamma)}\big[\Conn^{(\kappa,\sigma)-{\rm ss}}_\gamma(\zeta)\big]e_\gamma
\right).
\end{equation}
\end{Proposition}

Define the elements $C_\gamma(\zeta,\kappa,\sigma)\in\cMot(\kk)$, where $\gamma$ ranges over $\Gamma_+$ by the following formula (cf.~\eqref{eq:ExplAnswer}) valid for any $\tau\in\kk$, $\tau'\in\R$
\begin{equation}\label{eq:ExplAnswer3}
\sum_{\substack{\gamma\in\Gamma_+\\ \deg_{1,\zeta}\gamma=\tau\rk\gamma\\ \deg_{\kappa,\sigma}\gamma=\tau'\rk\gamma }}\bL^{-\chi(\gamma)}C_\gamma(\zeta,\kappa,\sigma)e_\gamma=
 \Exp\left(\sum_{\substack{\gamma\in\Gamma_+\\ \deg_{1,\zeta}\gamma=\tau\rk\gamma\\ \deg_{\kappa,\sigma}\gamma=\tau'\rk\gamma }}
 \overline B_\gamma e_\gamma
 \right).
\end{equation}
Note that we have $C_\gamma(\zeta,0,0)=C_\gamma(\zeta)$, where $C_\gamma(\zeta)$ are defined in~\eqref{eq:ExplAnswer2}. Now we can formulate our third main result.
\begin{Theorem}\label{th:ExplAnsw3} Assume that $\gamma\in\Gamma_+$, $\gamma\ne0$, and $\deg_{1,\zeta}\gamma=0$. Then for $N>3|\zeta|+(\rk\gamma-1)\delta/2$ we have
\[
 \big[\Conn_\gamma^{(\kappa,\sigma)-{\rm ss}}(\zeta)\big]=C_{\gamma-Nr\mathbf1}(\zeta,\kappa,\sigma).
\]
\end{Theorem}
We note that if $\kappa=0$, and $\sigma=0$, then this theorem is essentially Theorem~\ref{th:ExplAnsw2}.
\begin{proof}
 We defined $C_\gamma(\zeta,\kappa,\sigma)$ when $\gamma\in\Gamma_+'$. It will be convenient for us to extend this notation to the case when
 $\gamma$ is a multiple of $\mathbf1$ by setting $C_{n\mathbf1}(\zeta,\kappa,\sigma)=1$. We will prove that $\big[\Conn_\beta^{(\kappa,\sigma)-{\rm ss}}(\zeta)\big]=C_{\beta-Nr\mathbf1}(\zeta,\kappa,\sigma)$ whenever $\beta\in\Gamma_+$, $\deg_{1,\zeta}\beta=0$, and $N>3|\zeta|+(\rk\beta-1)\delta/2$ by induction on $\rk\beta$. The base case of $\rk\beta=0$ is obvious. Note that replacing $\gamma$ with $\gamma-N(\rk\gamma)\mathbf1$ shifts the $(\kappa,\sigma)$-slopes by $\kappa N$. Consider the product in the RHS of~\eqref{eq:KS2} and
\begin{equation}\label{eq:2}
 \prod_{\tau\in\R}\left(
 \sum_{\substack{\gamma\in\Gamma_+\\ \deg_{\kappa,\sigma}\gamma=\tau\rk\gamma\\ \deg_{1,\zeta}\gamma=N\rk\gamma}}\bL^{-\chi(\gamma)}[C_{\gamma-N\rk\gamma\mathbf1}(\zeta,\kappa,\sigma)]e_\gamma
\right).
\end{equation}
We note that this product makes sense because we set $C_{-N\rk\gamma\mathbf1}(\zeta,\kappa,\sigma)=1$ in the beginning of the proof. Note also that $\deg_{1,\zeta}\beta=0$ implies that $d/r\le|\zeta|$, where $\beta=(r,r_{\bullet,\bullet},d)$, so $N>2|\zeta|+(r-1)\delta/2+d/r$. Using~\eqref{eq:KS2} and Theorem~\ref{th:ExplAnsw2} we easily see that the coefficients of these products at $e_\beta$ are both equal to $\bL^{-\chi(\beta)}[\Conn_\beta(\zeta)]$. Expanding the product in the RHS of~\eqref{eq:KS2}, we see that the coefficient at $e_\beta$ is equal to
\[
 \bL^{-\chi(\beta)}\big[\Conn_\beta^{(\kappa,\sigma)-{\rm ss}}(\zeta)\big]+
 \sum_{\beta_1,\dots,\beta_n}\prod_{i=1}^n\bL^{-\chi(\beta_i)}\big[\Conn_{\beta_i}^{(\kappa,\sigma)-{\rm ss}}(\zeta)\big],
\]
where the summation is over all decompositions of $\beta$ into the sum of $n\ge2$ non-zero elements of~$\Gamma_+$ such that $\deg_{1,\zeta}\beta_i=0$ and $\deg_{\kappa,\sigma}\beta_i=\tau\rk\beta_i$. Similarly, expanding the product~\eqref{eq:2}, we see that the coefficient at $e_\beta$ is equal to
\[
 \bL^{-\chi(\beta)}[C_{\beta-N\rk\beta\mathbf1}(\zeta,\kappa,\sigma)]+
 \sum_{\beta_1,\dots,\beta_n}\prod_{i=1}^n\bL^{-\chi(\beta_i)}[C_{\beta_i-N\rk\beta_i\mathbf1}(\zeta,\kappa,\sigma)].
\]
It remains to show that the respective terms in the sums are equal. Clearly, we have $N>3|\zeta|+(\rk\beta_i-1)\delta/2$ and the statement follows from the induction hypothesis.
\end{proof}

As usual, if $X=\P^1$ we obtain similar formulas valid in $\Mot(\kk)$ by replacing $\overline B_\gamma$ with $B_\gamma$ in~\eqref{eq:ExplAnswer3}.

\section[Equalities of motivic classes and non-emptiness of moduli stacks]{Equalities of motivic classes and non-emptiness\\ of moduli stacks}\label{sect:NonEmpty}
\subsection{Equalities between motivic classes of stacks}

In this section we will give a criterion for non-emptiness of stacks $\Conn_\gamma^{(\kappa,\sigma)-{\rm ss}}(\zeta)$ or $\Higgs_\gamma^{\sigma-{\rm ss}}(\zeta).\!$ For the stacks $\Conn_\gamma(\zeta)$ such a criterion follows easily from~\cite{Crawley-Boevey:Indecomposable,CrawleyBoeveyIndecompPar}. We reduce the case, when stability condition is present, to the case of $\Conn_\gamma(\zeta)$ using some equalities between motivic classes and Proposition~\ref{pr:NonEmpty}. We will also discuss a relation with Simpson's non-abelian Hodge theory.

It follows from Theorem~\ref{th:ExplAnsw3} that the motivic class of $\Conn_\gamma^{(\kappa,\sigma)-{\rm ss}}(\zeta)$ depends only on the submonoid of $\Gamma_+$ given by the equations $\deg_{1,\zeta}\gamma'=0$, $\deg_{\kappa,\sigma}\gamma'=\tau\rk\gamma'$, where $\tau=\deg_{\kappa,\sigma}\gamma/\rk\gamma$. Using this fact, one can give a lot of examples of seemingly unrelated moduli stacks $\Conn_\gamma^{(\kappa,\sigma)-{\rm ss}}(\zeta)$ having the same motivic class. An analogous statement for moduli spaces of parabolic Higgs bundles is the content of Corollary~\ref{cor:EqualMot} above. Finally, we can get a lot of equalities between motivic classes of parabolic Higgs bundles and motivic classes of connections. In the following proposition, we show that motivic classes of the form $[\Conn_\gamma(\zeta)]$ are universal.
This proposition should be compared to Proposition~\ref{cor:EqualMot}.

\begin{Proposition}\label{pr:ConnUniversal}
Assume that $\kk$ is not a finite extension of $\Q$. Let $\zeta\in\kk[D\times\Z_{>0}]$. Let $\gamma\in\Gamma_+$. Let $(\kappa,\sigma)\in\Stab'$. Assume that $\zeta$ is non-resonant or $\kappa=1$ and $\sigma\in\Stab$. Set $\tau:=\deg_{\kappa,\sigma}\gamma/\rk\gamma$. Then there is $\zeta'\in\kk[D\times\Z_{>0}]$ such that
\begin{equation}\label{eq:EqualMot2}
 \{\gamma'\in\Gamma_+\colon \deg_{1,\zeta}\gamma'=0, \deg_{\kappa,\sigma}\gamma'=\tau\rk\gamma'\}=
 \{\gamma'\in\Gamma_+\colon \deg_{1,\zeta'}\gamma'=0\}.
\end{equation}
Moreover, we have $\big[\Conn_\gamma^{(\kappa,\sigma)-{\rm ss}}(\zeta)\big]=[\Conn_\gamma(\zeta')]$.
\end{Proposition}
\begin{proof}
 Choose $x\in D$ and set $\sigma'_{y,j}=\sigma_{y,j}$ if $y\ne x$, $\sigma'_{x,j}=\sigma_{x,j}-\tau$. Since for all $(r',r'_{\bullet,\bullet},d)\in\Gamma'$ we have $\sum_jr'_{x,j}=r'$, we see that $\deg_{\kappa,\sigma}\gamma'=\tau\rk\gamma$ if and only if $\deg_{\kappa,\sigma'}\gamma'=0$.

 Now, let $U$ be the $\Q$-vector subspace of $\R$ generated by $1$ and all the numbers $\sigma'_{y,j}$, where~$y$ ranges over~$D$ and $j$ ranges over positive integers. Similarly, let $V$ be the $\Q$-vector subspace of~$\kk$ generated by all the components $\zeta_{y,j}$ of $\zeta$. Since $\kk$ is not a~finite extension of $\Q$, there is a~$\Q$-linear embedding $U\oplus V\to\kk$. Moreover, we may assume that $(\kappa,1)$ maps to $1$. This follows from a general fact: if $L_1$ is a finite dimensional $\Q$-vector space, $L_2$ is an infinite dimensional $\Q$-vector space, $v_1$ and $v_2$ are non-zero vectors of $L_1$ and $L_2$ respectively, then there is a $\Q$-linear embedding $L_1\hookrightarrow L_2$ such that $v_1$ maps to $v_2$.

 Next, $(\sigma'_{\bullet,\bullet},\zeta)$ is an element of $(U\oplus V)[D\times\Z_{>0}]$ and we define $\zeta'$ to be its image in $\kk[D\times\Z_{>0}]$ under the embedding. It is clear that we have~\eqref{eq:EqualMot2}. Indeed, the equations
 $\deg_{1,\zeta}\gamma'=0$, $\deg_{\kappa,\sigma'}\gamma'=0 $ can be written as the following equation in $U\oplus V$:
 \[
 d'(1,\kappa)+\sum_{y\in D}\sum_{j\ge0}(r'_{y,j-1}-r'_{y,j})(\zeta_{y,,j},\sigma'_{y,j})=0,
 \]
 where $\gamma'=(r',r'_{\bullet,\bullet},d')$. But the image of this equation in $\kk$ is exactly $\deg_{1,\zeta'}\gamma'=0$.
 Now the equality of motivic classes follows from Theorems~\ref{th:ExplAnsw2} and~\ref{th:ExplAnsw3} (see formulas~\eqref{eq:ExplAnswer2} and~\eqref{eq:ExplAnswer3}).
\end{proof}

Similarly, we have the following proposition.

\begin{Proposition}\label{pr:ConnUniversal2}
Assume that $\kk$ is not a finite extension of $\Q$. Let $\zeta\in\kk[D\times\Z_{>0}]$. Let $\sigma\in\Stab$. Let $\gamma\in\Gamma_+$. Set $\tau:=\deg_{1,\sigma}\gamma/\rk\gamma$. Then there is $\zeta'\in\kk[D\times\Z_{>0}]$ such that
\begin{equation}\label{eq:EqualMot3}
 \{\gamma'\in\Gamma_+\colon \deg_{0,\zeta}\gamma'=0, \, \deg_{1,\sigma}\gamma'=\tau\rk\gamma'\}=
 \{\gamma'\in\Gamma_+\colon \deg_{1,\zeta'}\gamma'=0\}.
\end{equation}
Moreover, we have $\big[\Higgs_\gamma^{\sigma-{\rm ss}}(\zeta)\big]=[\Conn_\gamma(\zeta')]$.
\end{Proposition}

\begin{proof}
 The proof is analogous to that of Proposition~\ref{pr:ConnUniversal} except that one finds an embedding $U\oplus V\hookrightarrow\kk$ taking $(0,\kappa)$ to $1$ and uses Theorem~\ref{th:ExplAnsw} and~\eqref{eq:ExplAnswer} instead of Theorem~\ref{th:ExplAnsw3} and~\eqref{eq:ExplAnswer3}.
\end{proof}

\begin{Remark}
 Many motivic classes of parabolic Higgs bundles and parabolic connections are equal to classes of the form $\big[\Higgs^{\sigma-{\rm ss}}(0)\big]$, cf.~\cite[Theorem~1.2.1]{FedorovSoibelmans}. However, whether these classes are universal (that is, whether one can write every motivic class $\big[\Conn_\gamma^{(\kappa,\sigma)-{\rm ss}}(\zeta)\big]$ and $\big[\Higgs_\gamma^{\sigma-{\rm ss}}(\zeta)\big]$ in this form) is not clear; the reason is that there are restrictions on $\sigma\in\Stab$.
\end{Remark}

\subsection{Simpson's non-abelian Hodge theory} Assume now that $\kk=\C$. Let $\sigma\in\Stab$ and $\zeta\in\C[D\times\Z_{>0}]$. Let $\Re\zeta$ and $\Im\zeta$ denote the real and imaginary parts of $\zeta$. Assume that either $\sigma-2\Re\zeta\in\Stab$ or $\sigma+2\sqrt{-1}\Im\zeta$ is non-resonant. In this case, for $\gamma\in\Gamma_+$ we have moduli stacks
\[
 \Higgs_\gamma^{\sigma-{\rm ss}}(\zeta)\qquad \text{and}\qquad \Conn_\gamma^{(1,\sigma-2\Re\zeta)-{\rm ss}}\big(\sigma+2\sqrt{-1}\Im\zeta\big).
\]
Assume that $\deg_{1,\sigma}\gamma=0$. Then, according to the results of~\cite{SimpsonHarmonicNoncompact}, the corresponding categories are equivalent (see especially that table on p.~720 in loc.~cit.). Note also that this equivalence of categories can be upgraded to a diffeomorphism of coarse moduli spaces (cf.~\cite{Biquard1997Higgs,BiquardBoalch2004, Nakajima1996Hyperkaehler}).

\begin{Proposition}\label{pr:Simpson}
 We have in $\cMot(\C)$: $\big[\Higgs_\gamma^{\sigma-{\rm ss}}(\zeta)\big]=\big[\Conn_\gamma^{(1,\sigma-2\Re\zeta)-{\rm ss}}\big(\sigma+2\sqrt{-1}\Im\zeta\big)\big]$.
\end{Proposition}
\begin{proof}
Note that the system of equations $\deg_{0,\zeta}\gamma'=\deg_{1,\sigma}\gamma'=0$ is equivalent to the system of three real equations
$\deg_{0,\Re\zeta}\gamma'=\deg_{0,\Im\zeta}\gamma'=\deg_{1,\sigma}\gamma'=0$. This system is, in turn, equivalent to the system $\deg_{1,\sigma-2\Re\zeta}\gamma'=\deg_{1,\sigma+2\sqrt{-1}\Im\zeta}\gamma'=0$. It remains to use Theorems~\ref{th:ExplAnsw} and~\ref{th:ExplAnsw3}.
\end{proof}

We emphasize that this result does not follow from the diffeomorphism of coarse moduli spaces or from the equivalence of categories. Nor the diffeomorphism of coarse moduli spaces or equivalence of categories can be derived from our result. We would also like to mention~\cite[Theorem~4.2]{HoskinsLehalleurOnVoevodskyMotive}, where the equality of Voevodsky motives is proved in the case when parabolic structures are absent and the rank and degree are coprime.

\subsection{Indecomposable parabolic bundles and non-emptiness of moduli stacks}
\subsubsection{Indecomposable parabolic bundles}\label{sect:Indecomp} Here we recall some results of~\cite{CrawleyBoeveyIndecompPar}. Recall that $X$ is a smooth projective curve of genus $g$, $D\subset X(\kk)$ is a non-empty set. Let $\gamma\in\Gamma_+$. We would like to know whether there exists an indecomposable parabolic bundle of class $\gamma$. The following simple statement is noted in~\cite[Introduction]{Crawley-Boevey:Indecomposable}.
\begin{Lemma}\label{lm:g>0}
 Assume that $\kk$ is algebraically closed. If $g>0$, then for all $\gamma\in\Gamma_+$ there is an indecomposable parabolic bundle of class $\gamma$.
\end{Lemma}
\begin{proof}
 It is well-known that there is an indecomposable vector bundle on $X$ of rank $\rk\gamma$. Now one extends it arbitrarily to a parabolic bundle of class $\gamma$.
\end{proof}
Next, let $X=\P^1$. Fix $\gamma=(r,r_{\bullet,\bullet},d)\in\Gamma'$ and choose a sequence of positive integers $w_\bullet$ indexed by $D$ such that $r_{x,j}=0$ for $j\ge w_x$. Consider the star-shaped graph $G_{w_\bullet}$ with vertices~$v_*$ and~$v_{x,j}$ where $x\in D$, $j$ is between~1 and~$w_x-1$. The vertex $v_*$ is connected to all the vertices of the form $v_{x,1}$, the vertex $v_{x,i}$ is connected to $v_{x,i\pm1}$ (see picture).
$$
\begin{tikzpicture}

\draw[fill](0,5) circle [radius=0.1];
\node [below] (02) at (0,5) {};
\node [above] (01) at (0,5) {};
\node [right] (0) at (0,5) {};
\node [left] (*) at (-0.1,5) {$v_*$};

\draw[fill](2,8) circle [radius=0.1];
\node [left] (11l) at (2,8) {};
\node [right] (11r) at (2,8) {};
\draw [thick, -] (01) -- (11l);

\draw[fill](4,8) circle [radius=0.1];
\node [left] (12l) at (4,8) {};
\node [right] (12r) at (4,8) {};
\draw [thick, -] (12l) -- (11r);

\draw[opacity=0, fill](6.25,8) circle [radius=0.1];
\node [left] (13l) at (6.25,8) {};
\draw [thick, -] (13l) -- (12r);

\draw[fill](6.5,8) circle [radius=0.05];
\draw[fill](7,8) circle [radius=0.05];
\draw[fill](7.5,8) circle [radius=0.05];

\draw[opacity=0, fill](7.75,8) circle [radius=0.1];
\node [right] (1w12r) at (7.75,8) {};

\draw[fill](10,8) circle [radius=0.1];
\node [left] (1w11l) at (10,8) {};
\node [right] (1w11r) at (10,8) {};
\draw [thick, -] (1w11l) -- (1w12r);

\draw[fill](2,6) circle [radius=0.1];
\node [above] (vx1) at (2,6.1) {$v_{x,1}$};
\node[left] (21l) at (2,6) {};
\node[right] (21r) at (2,6) {};
\draw [thick, -] (1.8,6) -- (0.3,5.1);

\draw[fill](4,6) circle [radius=0.1];
\node [above] (vx2) at (4,6.1) {$v_{x,2}$};
\node [left] (22l) at (4,6) {};
\node [right] (22r) at (4,6) {};
\draw [thick, -] (22l) -- (21r);

\draw[opacity=0, fill](6.25,6) circle [radius=0.1];
\node [left] (23l) at (6.25,6) {};
\draw [thick, -] (23l) -- (22r);

\draw[fill](6.5,6) circle [radius=0.05];
\draw[fill](7,6) circle [radius=0.05];
\draw[fill](7.5,6) circle [radius=0.05];

\draw[opacity=0, fill](7.75, 6) circle [radius=0.1];
\node [right] (2w22r) at (7.75, 6) {};

\draw[fill](10,6) circle [radius=0.1];
\node [above] (vxw) at (10,6.1) {$v_{x,w_x-1}$};
\node [left] (2w21l) at (10,6) {};
\node [right] (2w21r) at (10,6) {};
\draw [thick, -] (2w21l) -- (2w22r);

\draw[fill](2,5) circle [radius=0.05];
\draw[fill](2,4) circle [radius=0.05];
\draw[fill](2,3) circle [radius=0.05];

\draw[fill](4,5) circle [radius=0.05];
\draw[fill](4,4) circle [radius=0.05];
\draw[fill](4,3) circle [radius=0.05];

\draw[fill](10,5) circle [radius=0.05];
\draw[fill](10,4) circle [radius=0.05];
\draw[fill](10,3) circle [radius=0.05];

\draw[fill](2,2) circle [radius=0.1];
\node[left] (k1l) at (2,2) {};
\node[right] (k1r) at (2,2) {};
\draw [thick, -] (0.15,4.7) -- (k1l);

\draw[fill](4,2) circle [radius=0.1];
\node [left] (k2l) at (4,2) {};
\node [right] (k2r) at (4,2) {};
\draw [thick, -] (k2l) -- (k1r);

\draw[opacity=0, fill](6.25,2) circle [radius=0.1];
\node [left] (k3l) at (6.25,2) {};
\draw [thick, -] (k3l) -- (k2r);

\draw[fill](6.5,2) circle [radius=0.05];
\draw[fill](7,2) circle [radius=0.05];
\draw[fill](7.5,2) circle [radius=0.05];

\draw[opacity=0, fill](7.75,2) circle [radius=0.1];
\node [right] (kwk2r) at (7.75,2) {};

\draw[fill](10,2) circle [radius=0.1];
\node [left] (kwk1l) at (10,2) {};
\node [right] (kwk1r) at (10,2) {};
\draw [thick, -] (kwk1l) -- (kwk2r);

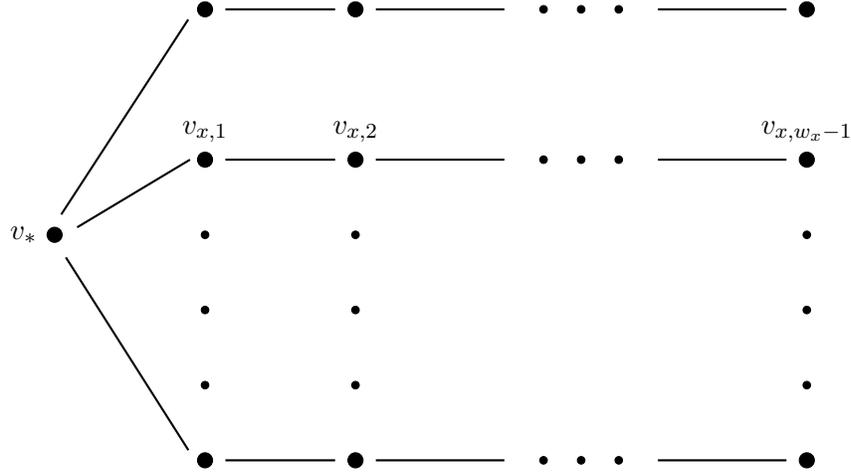
\node[] at (6,0.9) {\emph{Star-shaped graph}};
\end{tikzpicture}
$$

Consider the Kac--Moody Lie algebra $\fg_{w_\bullet}$ associated to the generalized Cartan matrix defined by this graph (see, e.g.,~\cite[Section~1]{Kac82}). Let $\Lambda_{w_\bullet}$ be the root lattice of $\fg_{w_\bullet}$; we identify it with the free abelian group generated by the set of vertices. Then $\gamma$ gives rise to an element of $\Lambda_{w_\bullet}$ given by
\[
 \rho_{\gamma,w_\bullet}:=rv_*+\sum_{x\in D}\sum_{j=1}^{w_x-1}\left(\sum_{i=1}^jr_{x,i}\right)v_{x,j}.
\]

Now~\cite[p.~1334, Corollary]{CrawleyBoeveyIndecompPar} can be re-formulated as follows.
\begin{Proposition}\label{pr:ExistIndecomp}
 In the above notation, there is a non-zero indecomposable parabolic bundle $\bE\in\Bun^{\rm par}_\gamma$ if and only if $\rho_{\gamma,w_\bullet}$ is a root of $\fg_{w_\bullet}$.
\end{Proposition}

We see that $\rho_{\gamma,w_\bullet}$ does not depend on $d$. Thus if there is an indecomposable parabolic bundle $\bE$ with $\cl(\bE)=(r,r_{\bullet,\bullet},d)$, then for any $d'$ there is an indecomposable parabolic bundle of class $(r,r_{\bullet,\bullet},d')$. Secondly, we see that the property of $\rho_{\gamma,w_\bullet}$ being a root does not depend on the choice of $w_\bullet$ as long as the components of $w_\bullet$ are large enough. By a slight abuse of terminology we say that $\gamma$ is a root in this case.

\begin{Remark}
In fact, one can consider an infinite star-shaped graph $G_D$ with $\deg D$ infinite rays, and the corresponding Kac--Moody Lie algebra $\fg_D$, which is the inductive limit of $\fg_{w_\bullet}$. Then we have a homomorphism $\rho$ from $\Gamma_+$ to the root lattice of $\fg_D$ and the classes of indecomposable parabolic bundles are exactly the $\rho$-preimages of roots.
\end{Remark}

\subsubsection{Non-emptiness of moduli stacks}
Now we can give a full answer to the question of when $\Higgs_\gamma^{\sigma-{\rm ss}} (\zeta)$, $\Conn_\gamma(\zeta)$ and $\Conn^{(\kappa,\sigma)-{\rm ss}}_\gamma (\zeta)\!\!$ are non-empty. The first statement follows immediately from results of Crawley--Boevey.

\begin{Theorem}\label{th:NonEmpty}
Assume that $\gamma\in\Gamma_+$ and $\zeta\in\kk[D\times\Z_{>0}]$.
\begin{enumerate}\itemsep=0pt
\item[$(i)$] If $g=g(X)>0$, then $\Conn_\gamma(\zeta)$ is non-empty if and only if $\deg_{1,\zeta}\gamma=0$.
\item[$(ii)$] If $g=0$, then $\Conn_\gamma(\zeta)$ is non-empty if and only if $\gamma$ can be written as $\sum\limits_{i=1}^n\gamma_i$, where $\gamma_i\in\Gamma'$ are roots and $\deg_{1,\zeta}\gamma_i=0$.
\end{enumerate}
\end{Theorem}
\begin{proof}
 Note that a $\kk$-stack $\cX$ is non-empty if and only if $\cX\times_\kk\overline\kk$ is non-empty, where $\overline\kk$ is the algebraic closure of $\kk$. Thus, we can assume that $\kk$ is algebraically closed from the very beginning. By Lemma~\ref{lm:existence2}, $\Conn_\gamma(\zeta)$ is non-empty if and only if there is a $(1,\zeta)$-isoslopy parabolic bundle of class $\gamma$ such that $\deg_{1,\zeta}\gamma=0$. Now (i) follows from Lemma~\ref{lm:g>0}, while (ii) follows from Proposition~\ref{pr:ExistIndecomp}.
\end{proof}

\begin{Theorem}\label{th:NonEmpty2}
Assume that $\gamma\in\Gamma_+$, $\zeta\in\kk[D\times\Z_{>0}]$, and $\sigma\in\Stab$.
\begin{enumerate}\itemsep=0pt
\item[$(i)$] If $g=g(X)>0$, then $\Higgs^{\sigma-{\rm ss}}_\gamma(\zeta)$ is non-empty if and only if $\deg_{0,\zeta}\gamma=0$.
\item[$(ii)$]
If $g=0$, then $\Higgs^{\sigma-{\rm ss}}_\gamma(\zeta)$ is non-empty if and only if $\gamma$ can be written as $\sum\limits_{i=1}^n\gamma_i$, where $\gamma_i\in\Gamma'$ are roots, $\deg_{0,\zeta}\gamma_i=0$, and the $(1,\sigma)$-slope of each $\gamma_i$ is equal to the $(1,\sigma)$-slope of $\gamma$.
\end{enumerate}
\end{Theorem}
\begin{proof}
 By Proposition~\ref{pr:NonEmpty}, $\Higgs^{\sigma-{\rm ss}}_\gamma(\zeta)$ is non-empty if and only if $\big[\Higgs^{\sigma-{\rm ss}}_\gamma(\zeta)\big]\ne0$. Let $\zeta'$ be as in Proposition~\ref{pr:ConnUniversal2}. Applying Proposition~\ref{pr:NonEmpty} again, we see that $\Higgs^{\sigma-{\rm ss}}_\gamma(\zeta)$ is non-empty if and only if $\Conn_\gamma(\zeta')$ is non-empty. It remains to use Theorem~\ref{th:NonEmpty} and~\eqref{eq:EqualMot3}.
\end{proof}

\begin{Theorem}\label{th:NonEmpty3}
Assume that $\gamma\in\Gamma_+$, $\zeta\in\kk[D\times\Z_{>0}]$ and $(\kappa,\sigma)\in\Stab'$. Assume that either~$\zeta$ is non-resonant, or $\kappa=1$ and $\sigma\in\Stab$.
\begin{enumerate}\itemsep=0pt
\item[$(i)$] If $g=g(X)>0$, then $\Conn^{(\kappa,\sigma)-{\rm ss}}_\gamma(\zeta)$ is non-empty if and only if $\deg_{1,\zeta}\gamma=0$.
\item[$(ii)$] If $g=0$, then $\Conn^{(\kappa,\sigma)-{\rm ss}}_\gamma(\zeta)$ is non-empty if and only if $\gamma$ can be written as $\sum\limits_{i=1}^n\gamma_i$, where $\gamma_i\in\Gamma'$ are roots, $\deg_{1,\zeta}\gamma_i=0$, and the $(\kappa,\sigma)$-slope of each $\gamma_i$ is equal to the $(\kappa,\sigma)$-slope of $\gamma$.
\end{enumerate}
\end{Theorem}
\begin{proof}
 Same as of Theorem~\ref{th:NonEmpty2} except that one uses Proposition~\ref{pr:ConnUniversal} instead of Proposition~\ref{pr:ConnUniversal2} and~\eqref{eq:EqualMot2} instead of~\eqref{eq:EqualMot3}.
\end{proof}

\subsection*{Acknowledgements}
We thank E.~Diaconescu, J.~Heinloth, O.~Schiffmann, and especially A.~Mellit for useful discussions and correspondence. We thank P.~Boalch for a useful comment on an earlier version. A part of this work was done while R.F.~was visiting Max Planck Institute of Mathematics in Bonn, and a part when he was visiting A.~Mellit at the University of Vienna. The work of R.F.~was partially supported by NSF grant DMS--1406532. A.S.~and Y.S.~thank IHES for excellent research conditions and hospitality. The work of Y.S.~was partially supported by NSF grants and Munson--Simu Faculty Award at Kansas State University. The authors would like to thank the anonymous referees for carefully reading the paper and for useful comments.


\pdfbookmark[1]{References}{ref}
\LastPageEnding

\end{document}